\documentclass[11pt,letter]{amsart}

\usepackage{graphicx}
\usepackage[utf8]{inputenc}
\usepackage[T1]{fontenc}
\usepackage{lmodern}
\usepackage{bm}
\usepackage{amsmath}
\usepackage{amssymb}
\usepackage{microtype}
\usepackage[boxed]{algorithm2e}

\usepackage{epstopdf}
\usepackage{tikz}
\usepackage{pgfplots}
\pgfplotsset{compat=newest}

\usepackage[utf8]{inputenx}

\usepackage{appendix}
\usepackage{lmodern}
\usepackage{textcomp}
\usepackage[english]{babel}
\usepackage[T1]{fontenc}
\usepackage{bm}
\usepackage{amssymb,amsthm,amsmath}
\usepackage{indentfirst}
\usepackage{color}
\usepackage{graphicx}
\usepackage[shortlabels]{enumitem}
\usepackage{caption}
\usepackage{empheq}
\usepackage{xcolor}
\usepackage{color}
\usepackage{amssymb}
\usepackage{booktabs}
\usepackage{graphicx}
\usepackage[font=small,labelfont=bf,labelsep=period,tableposition=top]{caption}
\usepackage{hyperref}

\newtheorem{theorem}{Theorem}[section]
\newtheorem{lemma}[theorem]{Lemma}
\newtheorem{proposition}[theorem]{Proposition}
\newtheorem{remark}[theorem]{Remark}
\newtheorem{definition}{Definition}[section]

\usepackage{ifthen}
\provideboolean{shownotes}\setboolean{shownotes}{true}
\newcommand{\margnote}[1]{\ifthenelse{\boolean{shownotes}}
	{\marginpar{\raggedright\tiny\texttt{#1}}}{}}

\hypersetup
{   bookmarks=false,
	colorlinks=true,
	urlcolor=blue,
	citecolor=blue,
	pdfauthor = {Andrea Aspri},
	pdftitle = {dislocation draft},
	pdfcreator = {LaTeX, TeXmaker, pdfLaTeX, Hyperref}
}

\usepackage{aliascnt}

\let\RE\Re
\let\Re=\undefined
\DeclareMathOperator{\Re}{\RE e}
\let\IM\Im
\let\Im=\undefined
\DeclareMathOperator{\Im}{\IM m}
\newcommand{\abs}[1]{\left|#1\right|}
\newcommand{\norm}[1]{\left\|#1\right\|}
\newcommand{\set}[1]{\left\{#1\right\}}
\newcommand{\ve}{\varepsilon}

\newcommand{\uep}{{u_\ve}}
\newcommand{\vep}{{v_\ve}}
\newcommand{\wep}{{w_\ve}}
\newcommand{\owep}{{\overline{w}_\ve}}
\newcommand{\gep}{{\gamma_\ve}}

\newcommand{\Wep}{{W_\ve}}

\newcommand{\dx}{{\;dx}}
\newcommand{\dmu}{{\;d \mu}}

\newcommand{\Oe}{{\Omega_\ve}}
\newcommand{\Oep}{{\Omega^{'}_\ve}}
\newcommand{\Ses}{{\Sigma_\ve}}
\newcommand{\COe}{{\Oe \backslash \Oep}}
\newcommand{\OOe}{{\Omega \backslash \Oe}}
\newcommand{\OOep}{{\Omega \backslash \Oep}}

\newcommand{\divz}{(\nabla\cdot)^m}
\renewcommand{\div}{(\nabla\cdot)}

\newcommand{\hzn}{H_0^m(\Omega)}
\newcommand{\tensor}{\mathbb{M}}

\newcommand{\const}{C(\kappa,\Omega,K,\delta_0,L_0,\norm{f}_{L^\infty(\Omega)},\alpha)}

\newcommand{\functional}[2]{\mathcal{J}(#1;#2)}
\newcommand{\normalderivative}[1]{\frac{\partial #1}{\partial n}}
\newcommand{\normalderivativem}[2]{\frac{\partial^{#2}#1}{\partial n^{#2}}}

\usepackage{dsfont}
\newcommand{\N}{\mathds{N}}

\newcommand{\R}{\mathds{R}}

\usepackage{aliascnt}

\newtheorem{assumption}[theorem]{Assumption}
\usepackage{enumitem}
\begin{document}
\title[Asymptotic Expansions for Higher Order Elliptic Equations]{Asymptotic Expansions for Higher Order Elliptic Equations with an Application to Quantitative Photoacoustic Tomography}

\author[A. Aspri {\em et al.}]{Andrea Aspri}
\address{ Johann Radon Institute for Computational and Applied Mathematics (RICAM)}
\email{andrea.aspri@ricam.oeaw.ac.at}

\author[]{Elena Beretta}
\address{Dipartimento di Matematica, Politecnico di Milano and Department of Mathematics, NYU-Abu Dhabi}
\email{elena.beretta@polimi.it}

\author[]{Otmar Scherzer}
\address{Faculty of Mathematics, University of Vienna}
\email{otmar.scherzer@univie.ac.at}

\author[]{Monika Muszkieta}
\address{Faculty of Pure and Applied Mathematics, Wroclaw University of Science and Technology}
\email{monika.muszkieta@pwr.edu.pl}

\date{\today}
%%% End

\keywords{Asymptotic expansions, higher order equations, thin stripes, inverse problem, quantitative photoacoustic}
%\vskip.15cm
%

\subjclass[2010]{Primary 35G15; Secondary 35C20, 35R30}

%%%%%%%%%%%%%%%%%%%%%%%%%%%%%%
%%% Titlepage
%%%%%%%%%%%%%%%%%%%%%%%%%%%%%%
%\maketitle
%\thispagestyle{empty}
%%\hspace*{8em}
%\parbox[t]{20em}{\footnotesize
%\hspace*{-1ex}$^1$ Johann Radon Institute for Computational\\
%\hspace*{1em}and Applied Mathematics (RICAM)\\
%Altenbergerstraße 69\\
%A-4040 Linz, Austria
%}
%\hspace*{3em}
%\parbox[t]{12em}{\footnotesize
%	\hspace*{-1ex}$^2$Departments of Mathematics\\
%		New York University Abu Dhabi
%%}
%%\hspace*{3em}
%\parbox[t]{14em}{\footnotesize
%\hspace*{-1ex}$^3$
%Dipartimento di Matematica\\
%Politecnico di Milano}
%}
%%\hspace*{8em}
%%\hspace*{8em}
%\ \\ \\
%\parbox[t]{20em}{\footnotesize
%\hspace*{-1ex}$^4$Faculty of Mathematics\\
%University of Vienna\\
%Oskar-Morgenstern-Platz 1\\
%A-1090 Vienna, Austria}
%\hspace*{3em}
%\parbox[t]{20em}{\footnotesize
%	\hspace*{-1ex}$^5$
%	Faculty of Pure and Applied Mathematics\\
%	Wroclaw University of Science and Technology\\
%    Wyb. Wyspianskiego 27, PL-50-370 Wroclaw, Poland}
\begin{abstract}
In this paper, we derive new asymptotic expansions for the solutions of higher order elliptic equations in the presence of small inclusions.
As a byproduct, we derive a topological derivative based algorithm for the 
%detection of discontinuities of higher order derivatives of noisy images.
reconstruction of piecewise smooth functions.
This algorithm can be used for edge detection in imaging, topological optimization, 
and for inverse problems, such as Quantitative Photoacoustic Tomography, for which we demonstrate the effectiveness of our asymptotic expansion method 
numerically.
\end{abstract}

\maketitle
%%% End
\section{Introduction} \label{sec:intro}

In this paper, we propose a new algorithm based on \emph{topological derivatives}
(see for example \cite{AmmBrerKanLee15,AmmKan04,LarFehMas12,SokZoc99,SokZol92}   for a review on topological optimization) which allows for detection of discontinuities
and its \emph{derivatives} of a given function $f$. 
The basic idea is that $f$ can be viewed as a piecewise smooth function and edges in $f$ and its derivatives can be modeled accurately 
by a set of singularities along small line segments. 
More precisely, given a noisy image $f$, in order to find a smoothed version $u$ of $f$, we consider the following functional 
for a fixed order $m\in\N$,
\begin{equation} \label{eq:Jeps_f}
\functional{u}{v} := \frac{1}{2}\int_\Omega (u-f)^2 \dx + \frac{\alpha}{2}\int_\Omega v \abs{\nabla^m u}^2 \dx,
\end{equation}
where $v\in L^\infty(\Omega)$, $v>0$. The case $m=1$ has been already studied in \cite{BerGraMusSch14}.
The first term in \eqref{eq:Jeps_f} measures the fidelity to the given image $f$ while the second one is a regularization term which accounts for the presence of discontinuities. The minimizer $u\in H^m_0(\Omega)$ of \eqref{eq:Jeps_f}, where $H^m_0(\Omega)$ is the classical Sobolev space of traces of derivatives which are zero on $\partial\Omega$ up to $m-1$, is also a weak solution of
	\begin{equation}\label{eq:ell_eq}
	\begin{cases} 
	\vspace{0.1cm}
	u + \alpha\ (-1)^m  \divz (v \nabla^m u)= f &\text{ in } \Omega,\\
	u = \normalderivative{u} =\cdots=\normalderivativem{u}{m-1}= 0 &\text{ on } \partial \Omega.
	\end{cases}
	\end{equation}
where $\divz$ is the divergence operator applied $m$ times, $n$ is the outward normal vector on $\partial\Omega$, $\frac{\partial}{\partial n}$ the normal derivative, and, recursively, $\normalderivativem{u}{l}=\frac{\partial}{\partial n}(\normalderivativem{u}{l-1})$, for $l=2,\cdots,m-1$.
In order to detect discontinuities in $f$ and its derivatives, we study the variation of the minimizer of the functional
with respect to small variations of $v$. We denote by $v_\ve$ and $u_\ve$ the perturbation of $v$ and $u$, respectively. Obviously, $u_\ve$ satisfies the same equation in \eqref{eq:ell_eq}, substituting $v$ with $v_\ve$. With small variations, we mean that $v_\ve$ differ from $v$ by a constant times a
characteristic function of small support - in our case a thin neighborhood of a line segment. Specifically, denoting with $\Omega_\ve(y,\tau)$ the thin strip centered at $y$ and along the direction $\tau$, and with $\kappa$ a real number such that $0<\kappa<\frac{1}{2}$, we have that 
	\begin{equation*}
	   (\vep-v)(x)=\begin{cases}
	   \kappa-1 & x \in \overline{\Omega}_{\varepsilon}(y,\tau),\\
	   0 & \text{ otherwise}.
	   \end{cases}	
	\end{equation*}
This procedure introduces discontinuities in derivatives of order $m$ of the minimizer along line segments and an accurate asymptotic analysis of the minimizer with respect to $v_\ve$ allows to compare the perturbed functional, $\functional{u_\ve}{v_\ve}$, with the unperturbed one, $\functional{u}{v}$, and to derive a rigorous formula of the topological derivative of the functional. Precisely, we prove that  
	\begin{equation*}
		\functional{\uep}{\vep}=\functional{u}{v} + 2\ve^3\alpha(\kappa-1)\mathbb{M}\nabla^mu(y)\cdot\nabla^mu(y)+o(\ve^3).
	\end{equation*}
where $\mathbb{M}$ is a tensor of order $2m$, which contains geometrical information on the direction of the line segment along which the discontinuity occurs.

The idea of ``drilling small holes''
in the domain for edge detection,  when $m=1$, has been introduced in \cite{Mus09} and has been successfully implemented in \cite{GraMusSch13},
where also a conceptual connection to Mumford-Shah minimization and the Ambrosio-Tortorelli relaxation 
\cite{AmbTor90} has been highlighted. Later, these results have been extended to the case where the edge set is covered with thin stripes rather than with balls in \cite{BerGraMusSch14}. 
Therefore, the main novelty in this paper is to use higher order derivatives discontinuities of the minimizer $u$ to recover more accurately image discontinuities. In order to get these results we need to overcome several significant difficulties, by combining known methods and
recent results on higher order elliptic equations with discontinuous coefficients, see \cite{Bar16}, and tensors, see \cite{ComGolLimMou08}, and generalizing some of them, in a novel way.

In the first part of our paper we derive an asymptotic expansion for the perturbed functional $u \to \mathcal{J}(\cdot,\vep)$ due to the presence of a small measurable set via compensated compactness, following the ideas of \cite{CapVog06} where an asymptotic expansion of perturbations of solutions to the conductivity equation has been derived.
Here, however,  the minimizer of the functional is the solution to a partial differential equation
of order $2m$, where $m>1$, with discontinuous leading coefficients, for which the theoretical basis is much less advanced (see \cite{Bar16}).  Also, to our knowledge,  asymptotic expansions of solutions of higher order elliptic equations have been derived only in the case of diametrically small domains \cite{AmsNovGoe14} while the analysis in the case of arbitrary measurable small domains, and in the form of thin strips like the ones considered here, is new and not at all straightforward. In particular, we obtain a full characterization of the Polya-Szego polarization tensor in the case of thin strips generalizing the results obtained in \cite{BerFraVog03,BerFra06}, and \cite{BerBonFraMaz12} for the conductivity equation and the linearized system of elasticity.

As explained in \cite{LarFehMas12}, topological gradient methods have been successfully applied to different areas of applications: topology optmization problems, image processing, image reconstructions to cite a few.\\
In the second part of the paper, we use the above results in an application to \emph{Quantitative Photoacoustic Tomography} (qPAT) (see
\cite{BerMusNaeSch16}). This is an imaging technique that excites a specimen by electromagnetic waves
and records the resulting ultrasound see \cite{Wan09} for an extensive treatment of the experimental issues and
see \cite{Kuc14,KucKun08} for an extensive mathematical analyis. Since the electromagnetic pulses used in PAT are
very short, the complete energy is deposited almost instantaneously compared to the travel times of the induced
acoustic waves, and therefore the pressure wave $p$ can be assumed to have been generated by an \emph{initial pressure}
$\mathcal{H}(x)$, that is, it satisfies the wave equation
\begin{equation*}
\begin{aligned}
\partial_{tt}p(x,t) - \Delta p(x,t) &= 0 \\
\partial_t p(x,0) &= 0 \\
p(x,0) &= \mathcal{H}(x)
\end{aligned}
\end{equation*}
and $p|_{\mathcal{M}}$, where $\mathcal{M}$ denotes a \emph{measurement surface}, can be obtained from ultrasound
measurements.

By solving an \emph{inverse problem for the wave equation} (see, e.g, \cite{KucKun08} for a review on mathematical and numerical
techniques), these ultrasound measurements can be used to estimate the \emph{initial pressure} $\mathcal{H}(x)$.  Since
\begin{equation} \label{eq:def_E}
\mathcal{H}(x) = \Gamma(x) \mathcal{E}(x) = \Gamma(x) \mu(x) u(x),
\end{equation}
in the particular case where $\Gamma$ is constant, $\mathcal{H}(x)$ and the \emph{absorbed energy} $\mathcal{E}(x)$ are proportional, which is in turn  proportional to the product of \emph{optical absorption coefficient} $\mu(x)$ and the fluence $u(x)$ (the time-integrated laser power received at $x$). Therefore, the initial pressure visualizes contrast in $\mu$.
The coefficient $\Gamma(x)$ is called \emph{Gr\"uneisen coefficient} (or \emph{photoacoustic efficiency}
since it describes the efficiency of conversion from absorbed energy to acoustic signal) \cite{CoxLauArrBea12,CoxLauBea09,WanWu07}.

\emph{Quantitative Photoacoustic Tomography} (qPAT) consists in determining spatially
heterogeneous Gr\"uneisen, absorption and diffusion coefficients, $\Gamma$, $\mu$ and $D$, from photoacoustic
measurements of the absorbed energy $\mathcal{H}=\Gamma \mu u$, where $u$ satisfies
\begin{equation} \label{qpatproblem}
\begin{aligned}
-\nabla \cdot(D\nabla u) + \mu u &= 0 \text{ in } \Omega\subset \mathbb{R}^2,\\
u&=g  \text{ on } \partial\Omega.
\end{aligned}
\end{equation}
Here $g$ is a given function which describes the illumination pattern.

Many facets of qPAT have been considered already in the literature; an incomplete list is
\cite{YuaJia06,CoxLauBea09,YaoSunJia09,BalUhl10,BalUhl12,CoxLauArrBea12,HuaYaoLiuRonJia12,HuaRonYaoQiWuXuJia13,RenGaoZhao13,MamRen14,PulCoxArrKaiTar15}.
In general the problem of estimating $\Gamma$, $\mu$ and $D$ is ill-posed since it
admits an infinite number of solution pairs (see \cite{BalUhl10,BalUhl12}).

As one motivation for this paper serves qPAT with piecewise constant material parameters $\mu$ and $D$ and a constant known $\Gamma$ as considered previously
in \cite{NaeSch14,BerMusNaeSch16}. Opposed to the general setting, if $\mu$ and $D$ are piecewise
constant functions, then they can be uniquely determined from knowledge of the absorbed energy $\mathcal{E}$ if in addition the values of the two parameters
at the boundary are known \cite{NaeSch14}. Moreover, it is a useful fact that the union of the jumps of 
$\mu$ and $D$ are contained in the set of discontinuities of derivatives up to the $2$-nd order of the qPAT measurement
absorbed energy $\mathcal{E}$, which then takes the role of the data function $f$ in \eqref{eq:Jeps_f}. \\
In order to recover its discontinuities and consequently the discontinuities of $\mu$ and $D$, we use the topological derivative approach described in the first part. This allows to detect discontinuities  of $\mathcal{E}$ more accurately than in \cite{BerMusNaeSch16} where a variational method based on an Ambrosio–Tortorelli relaxation
of a Mumford–Shah-like functional has been considered. In fact, the numerical algorithm that we present here is sufficiently robust to identify both $\mu$ and $D$ even in the presence of noise in the data $f$, see Section \ref{sec:numerics} for more details and comparisions.

The outline of the paper is as follows: In \autoref{sec:notation} we introduce the basic notation and state the
general assumptions used throughout the paper. In \autoref{le:tas}, we study the properties of the  minimizers of the functional \eqref{eq:Jeps_f} and of its perturbed version. This is, in fact, the first step in order to introduce, in \autoref{sec:pol}, the tensor of order $2m$ which is the basic
analytical tool needed to characterize the topological gradient (see \autoref{sec:top_gradients}) and to describe the variations of the minimizers of the functionals \eqref{eq:Jeps_f}. Finally, we complement this analytical theory with numerical examples. In fact, \autoref{sec:numerics} is devoted to the presentation of generalized versions of the numerical algorithms described in \cite{BerGraMusSch14}, which are applied to Quantitative Photoacoustic Tomography.

\section{Notation and main assumptions}
\label{sec:notation}
In this section we recall the notation regarding tensors and functional spaces and the main assumptions used along this paper. \ \\ \ \\

We first recall that a tensor of order $m$ can also be represented as hypermatrices by choosing a basis, i.e., $A_{i_m,\cdots,i_m}$, where $i_k=1,\cdots,d$ and $d$ is the dimension of vector space, see for example \cite{ComGolLimMou08}. In the sequel $d=2$.	
	\begin{description}
		\item[Tensors and hypermatrices]\ \\
		\begin{itemize}
			\item Latin capital letters, e.g., $A$, $E$, $V$, indicates tensors of order $m$. The blackbord bord letters, e.g. $\mathbb{M}$, represents tensors of order $2m$.  
			\item Given a vector of components $(i_1,\cdots i_m)$, with $m\geq 2$, we use the symbol $\bm{i}:=(i_1,\cdots, i_m)$ in order to shorten the notation of indices for tensors. For instance, we will often write $A_{\bm{i}}$ instead of $A_{i_1,\cdots,i_m}$. 
			\item Let $A, B \in \R^{2 \times 2\times ...\times 2}$ be two tensors of order $m$, then $A \cdot B = \sum_{\bm{i},\bm{j}=1}^2 a_{\bm{i}\bm{j}} b_{\bm{i}\bm{j}}$ denotes the usual scalar
			product, i.e.,
				\begin{equation*}
					A \cdot B = \sum_{\substack{\bm{i}=(i_1,\cdots,i_m)=1 \\   \bm{j}=(j_1,\cdots,j_m)=1}}^2 a_{i_1,\cdots,i_m,j_1,\cdots,j_m} b_{i_1,\cdots,i_m,j_1,\cdots,j_m}
				\end{equation*}
%			 By $A B = \begin{bmatrix}
%			AB
%			\end{bmatrix}_{1 \leq \bm{i},\bm{j} \leq 2} =
%			\begin{bmatrix}
%			\sum_{\bm{k}=1}^2 a_{\bm{i}\bm{k}} b_{\bm{k}\bm{j}}
%			\end{bmatrix} \in \R^{2 \times 2\times\cdots \times 2}$
%			we denote tensor multiplication.
			\item Let $A, B \in \R^{2 \times 2\times \cdots \times 2}$ be two tensors of order $m$, then $\mathbb{M}_{\bm{i}\bm{j}} := A \otimes B = \begin{bmatrix}
			a_{\bm{i}}b_{\bm{j}}
			\end{bmatrix}$ is a tensor of order $2m$.
			\item Let $A \in \R^{2 \times 2\times\cdots \times 2}$ be a tensor of order $m$, then $\abs{A}:=\left(\sum_{\bm{i}=(i_1,\cdots,i_m)=1}^2 a_{\bm{i}}^2 \right)^{\frac{1}{2}}$ denotes the Frobenius norm of $A$.
		\end{itemize}
		\item[Function Spaces:]\ \\
		\begin{itemize}
			\item $H^j(\Omega)$ denotes the Sobolev space, where all derivatives up to order $j \in \N$ are square integrable
			(see for instance \cite{Ada75}).
			\item $H_0^j(\Omega) := \overline{C_0^\infty(\Omega)}^{H^j(\Omega)}$ is the subspace of functions, which satisfy homogenuous
			boundary condition. See again \cite{Ada75}.
			\item $\nabla^m$ u represents the tensor of the derivatives of order $m$ of the function $u$, i.e.,  
			\begin{equation*}
			\nabla^mu=(\nabla^m u)_{i_1,\cdots,i_m}=\frac{\partial^mu}{\partial x_{i_1}\cdots \partial x_{i_m}}
			\end{equation*}
			where $i_k=1,2$, for $k=1,\cdots,m$.
			\item $\div$ is the divergence operator.
			\item Let $A : \Omega \to \R^{2 \times 2\times\cdots\times 2}$ be a tensor of order $m$, then, we define
			\begin{equation*}
			(\div  A)_{i_1,\cdots,i_{m-1}}: = \sum_{i_m=1}^{2}\partial_{x_{{i_m}}} A_{{i_1\cdots i_m}}.
			\end{equation*}
			\item We define the divergence operator, $\divz A$, applied $m$ times inductively, i.e., $\divz=(\nabla\cdot)((\nabla^{m-1}\cdot)A)$.
			\item The polyharmonic operator is defined inductively by $\Delta^m u=\Delta(\Delta^{m-1}u)$, with $m\geq 2$.
			\item Denoting with $n$ the outward normal vector on $\partial\Omega$ and with  $\frac{\partial}{\partial n}$ the normal derivative, the symbol $\normalderivativem{u}{m}$ is defined recursively by $\normalderivativem{u}{m}=\frac{\partial}{\partial n}(\normalderivativem{u}{m-1})$.
			\item $\zeta$ and $f$ will denote bounded functions such that
			$\zeta \in L_+^\infty(\Omega):=\{\zeta\in L^{\infty}(\Omega):\zeta> 0\, \text{a.e. in }\Omega \}$ and $f \in L^\infty(\Omega)$.
		\end{itemize}
		\item[Parameters:]\ \\
		\begin{itemize}
			\item $\alpha>0$ is fixed during the whole paper and has the role of a regularization parameter.
		\end{itemize}
	\end{description}
Let us now state the main assumptions.
\begin{assumption}[Domains]\label{ass:assumption_domains}\ \\
  \begin{enumerate}[(i)]
    \item $\Omega$ denotes an open, connected and bounded subset of $\R^2$ with Lipschitz boundary $\partial \Omega$.
    \item $\mathcal{B}_\rho(y)$ denotes a 2-dimensional ball of radius $\rho$ and center $y \in \R^2$.
    \item $K$ denotes a closed subset of $\Omega$ with positive 2D-Hausdorff measure.
    \item $\delta_0 > 0$ and $L_0 \subseteq \Omega \backslash K$ is an open domain with smooth boundary satisfying
          \begin{equation} \label{eq:lnull}
           \text{dist} (\overline{L}_0, \partial \Omega \cup K) \geq \delta_0 > 0.
          \end{equation}
    \item $0 < \ve < 1$, $y \in L_0$ and $\tau \in \mathbb{S}^1$ be fixed, such that
          \begin{equation} \label{eq:thinstripe}
            \Oe(y,\tau):= \set{x \in \R^2 : \text{dist}(x,\Ses(y,\tau))< \ve^2},
          \end{equation}
          with
          \begin{equation} \label{eq:Ses}
           \Ses(y,\tau):=\set{x \in \R^2 : x = y+\rho \tau,\ -\ve \leq\rho\leq \ve},
          \end{equation}
          is contained in $L_0$. In particular $\Oe(y,\tau)$  does not contain $K$.
          Moreover, we denote by $\Oep$ the rectangular box around $\Ses(y,\tau)$ and the two caps by
          $\COe$ (see \autoref{fig:L}).
          Since $0 < \ve < 1$, $y \in L_0$ and $\tau \in \mathbb{S}^1$ are fixed throughout the paper,
          we omit the dependencies of $\Ses(y,\tau)$ and $\Oe(y,\tau)$ on $y$ and $\tau$.
    \item $n_\ve$, $n$ denote the outward normal vectors to $\partial \Oe$ and
          $\partial \Omega$, respectively.
 \end{enumerate}
\end{assumption}
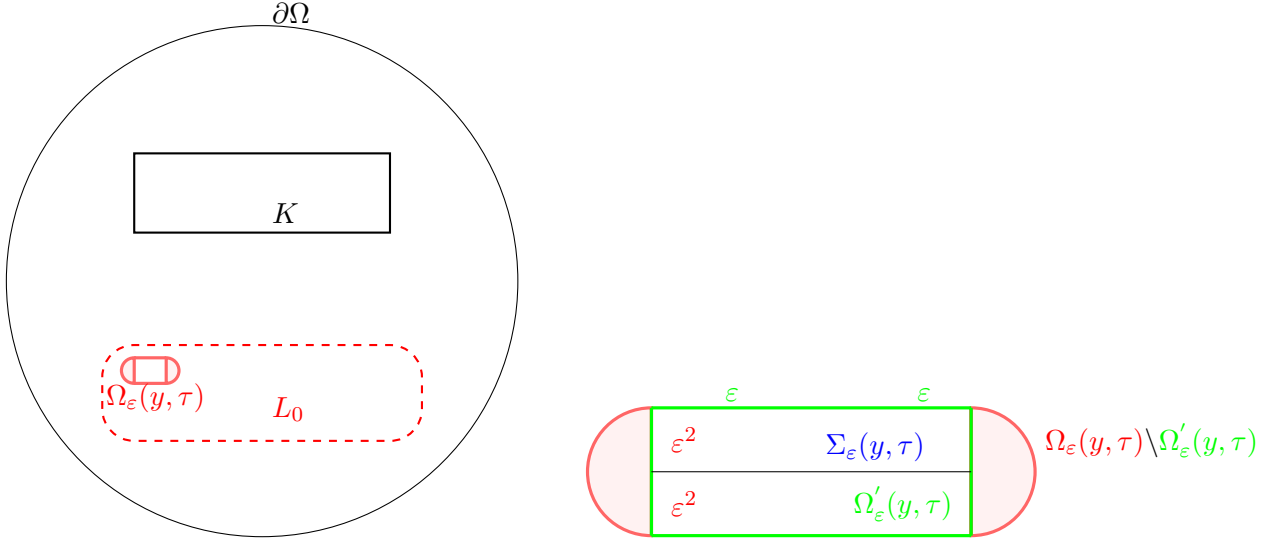
\begin{figure}
  \begin{center}
  \begin{tikzpicture}
   \draw (0,0) circle (0.2\textwidth);
   \draw [color=black] (0,0.21\textwidth) node[anchor=west] {$\partial \Omega$};
%   \draw (0,1) .. controls (0,1) and (1,0) .. (1,1);
%   \draw (-0.5,1.25) .. controls (0,-0) and (-0.5,0) .. (-1.5,1.5);
   \draw[black,thick]
(-0.1\textwidth,0.1\textwidth) rectangle (0.1\textwidth,0.038\textwidth) { };
%  \filldraw[color=black, very thick]
%(-0.1\textwidth,0.08\textwidth) -- (0.075\textwidth,0.08\textwidth)
%(-0.1\textwidth,0.1\textwidth) -- (0.075\textwidth,0.1\textwidth);	
   \draw [color=black] (0,0.9) node[anchor=west] {$K$};
   \filldraw[color=red!60, fill=red!5, very thick] (-0.1\textwidth,-0.06\textwidth)
            arc[radius = 0.01\textwidth, start angle= 90, end angle= 270]--cycle;
   \filldraw[color=red!60, fill=red!5, very thick] (-0.075\textwidth,-0.08\textwidth)
            arc[radius = 0.01\textwidth, start angle= -90, end angle= 90]--cycle;
   \filldraw[color=red!60, fill=red!5, very thick]
            (-0.1\textwidth,-0.06\textwidth) -- (-0.075\textwidth,-0.06\textwidth)
            (-0.1\textwidth,-0.08\textwidth) -- (-0.075\textwidth,-0.08\textwidth);
   \draw[red,thick,dashed,rounded corners=0.025\textwidth]
        (-0.125\textwidth,-0.125\textwidth) rectangle (0.125\textwidth,-0.05\textwidth) { };
   \draw (-0.13\textwidth,-0.091\textwidth) node[anchor=west] [color = red] {$\Oe(y,\tau)$};
   \draw (0,-0.1\textwidth) node[anchor=west] [color=red]{$L_0$};
\end{tikzpicture}
\qquad
\begin{tikzpicture}
 \filldraw[color=red!60, fill=red!5, very thick] (0,0.05\textwidth) arc[radius = 0.05\textwidth, start angle= 90, end angle= 270]--cycle
          (0,-0.05\textwidth) -- (0,0.05\textwidth);
 \filldraw[color=red!60, fill=red!5, very thick] (0.25\textwidth,-0.05\textwidth)
          arc[radius = 0.05\textwidth, start angle= -90, end angle= 90]--cycle (0.25\textwidth,-0.05\textwidth) --(0.25\textwidth,0.05\textwidth);
 \filldraw[color=green, fill=green, very thick]
          (0,0.05\textwidth) -- (0.25\textwidth,0.05\textwidth) (0.25\textwidth,0.05\textwidth)--(0.25\textwidth,-0.05\textwidth)
          (0.25\textwidth,-0.05\textwidth) -- (0,-0.05\textwidth) (0,-0.05\textwidth) -- (0,0.05\textwidth);
 \draw (0,0\textwidth) -- (0.25\textwidth,0\textwidth) node[pos=0.7,above,blue]{$\Ses(y,\tau)$};
 \draw (0.0075\textwidth,0.025\textwidth) node[anchor=west] [color=red]{$\ve^2$};
 \draw (0.05\textwidth,0.06\textwidth) node[anchor=west] [color=green]{$\ve$};
 \draw (0.0075\textwidth,-0.025\textwidth) node[anchor=west] [color=red]{$\ve^2$};
 \draw (0.2\textwidth,0.06\textwidth) node[anchor=west] [color=green]{$\ve$};
 \draw (0.3\textwidth,0.025\textwidth) node[anchor=west]{$\textcolor{red}{\Oe(y,\tau)} \backslash \textcolor{green}{\Oep(y,\tau)}$};
 \draw [color=green] (0.15\textwidth,-0.025\textwidth) node[anchor=west]{$\Oep(y,\tau)$};
\end{tikzpicture}
\end{center}
\caption{\label{fig:L} $L_0$ does not touch $\partial \Omega$ and $K$ and can contain the stripe $\Oe(y,\tau)$.
         The scaling of the thin stripe: $\Oe(y,\tau) = \Oep(y,\tau) \cup \Oe(y,\tau) \backslash
         \Oep(y,\tau)$.}
\end{figure}
\begin{remark} \label{re:morphology}
 We use the terminology of morphological image analysis (see for instance \cite{Soi99}).
 Let $B \subseteq \R^2$ be a \emph{structuring element} and $A \subseteq \R^2$ an arbitrary set,
 then the \emph{dilation} of $A$ with respect to the structuring element $B$ is defined as follows:
 \begin{equation*}
  A \oplus B := \set{x + y : x \in A \text{ and } y \in B}.
 \end{equation*}
 In particular
\begin{equation*}
 \Ses(y,\tau) = \set{y} \oplus \Ses(0,\tau) \text{ and }
 \Oe (y,\tau) = \Ses(y,\tau) \oplus \mathcal{B}_\rho(0).
\end{equation*}
\end{remark}
\begin{assumption}[Functions]\label{ass:funct}
Let $0 < \kappa <\frac{1}{2}$. We define the following functions:
\begin{enumerate}[(i)]
  \item Recalling the definition of $K$, see Assumption \ref{ass:assumption_domains}, we define
  		\begin{equation}\label{eq:v}
           v = \kappa \chi_K + 1 \chi_{\Omega \backslash K}.
        \end{equation}
  \item For given $(y,\tau)$ such that $\Oe(y,\tau) \subseteq L_0$ we define $\vep:\Omega\rightarrow \mathbb{R}$ by
        \begin{equation}\label{eq:ve}
           \vep(x)=\begin{cases}
                    \kappa& x\in K \cup \overline{\Omega}_{\varepsilon}(y,\tau),\\
                        1 & \text{ otherwise.}
                    \end{cases}
        \end{equation}
        Therefore, we have
        \begin{equation}\label{eq:ve_difference}
           (\vep-v)(x)=\begin{cases}
                        \kappa-1 & x \in \overline{\Omega}_{\varepsilon}(y,\tau),\\
                        0 & \text{ otherwise.}
                       \end{cases}
        \end{equation}
 \end{enumerate}
\end{assumption}

\begin{remark}
We note that
\begin{equation} \label{eq:area_incl}
   \abs{\Oe}=\ve^3(4+\pi\ve) = \mathcal{O}(\ve^3), \abs{\Oep}= 4\ve^3 = \mathcal{O}(\ve^3) \text{ and }
   \abs{\Oe \backslash \Oep} = \pi \ve^4 = \mathcal{O}(\ve^4).
\end{equation}
\end{remark}
We finally stress that all along this paper $C$ denotes a generic constant, which can depend on
$\kappa,\delta_0,y,\tau,K,\Omega,L_0,\alpha$, and $f$
but not necessarily needs to depend on all of them. That is
\begin{equation} \label{eq:C}
C:=\const > 0.
\end{equation}
%%%% \end{notation}
\section{Asymptotic analysis}\label{le:tas}
Let $\zeta \in L_+^\infty(\Omega)$, $f\in L^{\infty}(\Omega)$, $\alpha > 0$ and $m\geq2$ fixed. In this section we are analyzing the following
functional defined on $\hzn$:
\begin{equation} \label{Jeps}
 \functional{u}{\zeta} := \frac{1}{2} \int_\Omega (u-f)^2 \dx + \frac{\alpha}{2} \int_\Omega \zeta \abs{\nabla^m u}^2 \dx.
\end{equation}
In the following we characterize the minimizer of $\functional{\cdot}{\zeta}$ as the solution of a
$2m$-order partial differential equation:
\begin{lemma} \label{le:pde}
 Let $v$ be defined as in \eqref{eq:v}, then there exists a unique minimum
 $u \in \hzn$ of $\functional{u}{v}$ and $u$ is also a weak solution of
 \begin{equation}\label{u_strong}
 	\begin{cases} 
 	\vspace{0.1cm}
     u + \alpha\ (-1)^m  \divz (v \nabla^m u)= f &\text{ in } \Omega,\\
     u = \normalderivative{u} =\cdots=\normalderivativem{u}{m-1}= 0 &\text{ on } \partial \Omega.
   \end{cases}
 \end{equation}
 In addition, let $\vep$ be defined as in \eqref{eq:ve}, then there exists a unique minimum
 $\uep \in \hzn$ of $\functional{\uep}{\vep}$ and $\uep$ is also a weak solution of
 \begin{equation} \label{ue_strong}
  \begin{cases}
  \vspace{0.1cm}
      \uep + \alpha\ (-1)^m  \divz (\vep \nabla^m \uep) = f &\text{ in } \Omega,\\
 	  \uep = \normalderivative{\uep} =\cdots=\normalderivativem{\uep}{m-1}= 0 &\text{ on } \partial \Omega.
  \end{cases}
 \end{equation}
 Moreover, the following energy estimate holds
 \begin{equation*}
  \max \set{\norm{u}_{H^m(\Omega)},\norm{\uep}_{H^m(\Omega)}} \leq C \norm{f}_{L^{\infty}(\Omega)}.
 \end{equation*}
\end{lemma}
See the Appendix \ref{append} for the proof of this Lemma.

We define the function $\wep := \uep-u$.  As a consequence of \autoref{le:pde}, we have 
	\begin{lemma}
	From \eqref{eq:ve_difference}, \eqref{u_strong} and \eqref{ue_strong}, function $\wep\in H^m_0(\Omega)$ is the weak solution to	
	\end{lemma}
 \begin{equation} \label{we_strong1}
  \begin{cases}
  \vspace{0.1cm}
   \wep +  \alpha\ (-1)^m \divz (v \nabla^m \wep) =
   \alpha\ (-1)^m  \divz((1-\kappa)\chi_\Oe \nabla^m \uep) &\text{ in } \Omega,\\
   \wep = \normalderivative{\wep}=\cdots=\normalderivativem{\wep}{m-1} =0 &\text{ on } \partial \Omega,
  \end{cases}
 \end{equation}
and 
	\begin{equation} \label{we_strong2}
		\begin{cases}
		\vspace{0.1cm}
		\wep +  \alpha\ (-1)^m \divz (\vep \nabla^m \wep) =
		\alpha\ (-1)^m  \divz((1-\kappa)\chi_\Oe \nabla^m u) &\text{ in } \Omega,\\
		\wep = \normalderivative{\wep}=\cdots=\normalderivativem{\wep}{m-1} =0 &\text{ on } \partial \Omega.
		\end{cases}
	\end{equation} 
\begin{proof}
	The proof follows by subtracting \eqref{u_strong} from\eqref{ue_strong} and then by adding and substracting $v\nabla^m \uep$ and $\vep\nabla^m u$ properly. In fact, it is straightforward to find that the weak formulations of the two problems \eqref{we_strong1} and \eqref{we_strong2}, are,  for all $\varphi \in \hzn$,
	\begin{equation} \label{EnEst_eq1a}
	\int_\Omega \wep \varphi  \dx
	+ \alpha\int_\Omega v \nabla^m \wep\cdot \nabla^m \varphi  \dx
	= \alpha(1-\kappa) \int_\Oe \nabla^m \uep \cdot \nabla^m\varphi \dx,
	\end{equation}
	and
	\begin{equation} \label{EnEst_eq1}
	\int_\Omega \wep \varphi \dx + \alpha \int_\Omega \vep \nabla^m \wep\cdot \nabla^m \varphi \dx
	= \alpha(1-\kappa) \int_\Oe \nabla^m u \cdot \nabla^m\varphi \dx,
	\end{equation}
	respectively. 
\end{proof}	

\subsection{Asymptotics of $\wep$}
We need the next estimates for $u$ which is a consequence of the local regularity results for polyharmonic equations which can be found in \cite{GazGruSwe10}.
\begin{lemma} \label{Reg}
Let $u$ be the solution of \eqref{u_strong}. Then $u \in H^{2m}(L_0)$ and there exists a constant $C$ such that
\begin{equation} \label{ineq1}
 \norm{\nabla^m u}_{L^\infty(L_0)} \leq C \norm{f}_{L^\infty(\Omega)}.
\end{equation}
\end{lemma}
\begin{proof}	
Since $v=1$ in $L_0$, the equation in \eqref{u_strong} is equal to 
	\begin{equation*}
		u+\alpha (-\Delta)^m u=f,\qquad \textrm{in}\,\, L_0
	\end{equation*}
hence, by interior regularity results for polyharmonic operators with smooth coefficients (see for instance
\cite{GazGruSwe10}, Theorem 2.20) it follows that  $u \in H^{2m}(L_0)$ and
\begin{equation*}
\norm{u}_{H^{2m}(L_0)} \leq C (\norm{f}_{L^\infty(\Omega)} + \norm{u}_{H^m(\Omega)}) \leq C \norm{f}_{L^\infty(\Omega)}.
\end{equation*}
Then, by using Sobolev's embedding theorem, \cite[Thm. 6.2]{Ada75}, we find that
	\begin{equation*}
		H^{2m}(L_0)\subset C^{m,\gamma}(L_0),\qquad \textrm{with}\ \ \gamma\in (0,1),
	\end{equation*}
hence
	\begin{equation*}
		\|\nabla^m u\|_{L^{\infty}(L_0)}\leq C \| u\|_{H^{2m}(L_0)}\leq C \norm{f}_{L^\infty(\Omega)}.
	\end{equation*}
\end{proof}
In the following lemma we get some asymptotic behaviour on the function $\wep$.
\begin{lemma} \label{EnEst}
Let $p,q\in\mathbb{R}$ be such that $\frac{1}{p}+\frac{1}{q}=1$, with $q\in(1,2)$ and $p\in(2,+\infty)$, and $\eta_{m,q,k}:=\frac{k}{m}(\frac{1}{q}-\frac{1}{2})$, for $k=1,\cdots,m-1$, where $\eta_{m,q,k}>0$ for every $k=1,\cdots,m-1$. Then there exists some positive constant $C$ independent of $\ve$ such that $\wep$ satisfies
\begin{equation}\label{EnEst_Hm}
 \norm{\wep}_{H^m(\Omega)} \leq C\abs{\Oe}^{\frac{1}{2}} \text{ and }
 \norm{\wep}_{H^{m-k}(\Omega)} \leq C\abs{\Oe}^{\frac{1}{2}+\eta_{m,q,k}},
\end{equation}
for every $k=1,\cdots,m$.
\end{lemma}
\begin{proof}
from \eqref{EnEst_eq1}, using the test function $\varphi=\wep$ (which is an element of $\hzn$ according to \autoref{le:pde}), we obtain
\begin{equation*}
\int_\Omega \wep^2  \dx
+ \alpha\int_\Omega v_{\ve} \abs{\nabla^m \wep}^2  \dx
= \alpha(1-\kappa)\int_\Oe \nabla^m u \cdot \nabla^m \wep\dx.
\end{equation*}
Application of the Cauchy-Schwarz inequality, the use of \eqref{ineq1}, \eqref{eq:ve} and the fact that $\Oe \subset\subset L_0$
(see \autoref{fig:L}), then give
\begin{equation*}
 \begin{aligned}
 \kappa \norm{\nabla^m \wep}_{L^2(\Omega)}^2
 &\leq (1-\kappa) \int_{\Oe} \nabla^m u \cdot \nabla^m \wep\dx
  \leq (1-\kappa) \norm{\nabla^m u}_{L^\infty(\Oe)} \int_\Oe  \abs{\nabla^m \wep} \dx\\
 &\leq C  \norm{\nabla^m \wep}_{L^2(\Omega)} \abs{\Oe}^{\frac{1}{2}}.
 \end{aligned}
\end{equation*}
Thus, there exists a positive constant $C$ such that
\begin{equation*}
\norm{\nabla^m \wep}_{L^2(\Omega)} \leq C \abs{\Oe}^{\frac{1}{2}},
\end{equation*}
and thus the first estimate in \eqref{EnEst_Hm} follows by means of the Poincar\'e's inequality, see \autoref{le:poincare}.

To prove the second inequality in \eqref{EnEst_Hm}, we consider an auxiliary function $\overline{W}_{\ve}$ which satisfies
	\begin{equation}\label{eq:We}
	\begin{cases}
	  \vspace{0.1cm}
		\overline{W}_{\ve} + \alpha\ (-1)^m (\nabla\cdot)^m(v\nabla^m \overline{W}_{\ve})=w_\ve,& \textrm{in}\ \ \Omega,\\
		\overline{W}_{\ve}=\normalderivative{\overline{W}_{\ve}}=\cdots=\normalderivativem{\overline{W}_{\ve}}{m-1}=0 & \textrm{on}\ \ \partial\Omega,
	\end{cases}
	\end{equation}
for which the weak solution is given by
\begin{equation*}
\int_\Omega \overline{W}_\ve\varphi+\alpha\int_\Omega v\nabla^m\overline{W}_\ve\cdot\nabla^m\varphi=\int_\Omega \wep\varphi,
\end{equation*}
for all $\varphi\in H^m_0(\Omega)$.
Inserting $\varphi=\wep$ into last equation, and $\varphi=\overline{W}_\ve$ into \eqref{EnEst_eq1a} and then subtracting the resulting equations, we find 
\begin{equation*}
\int_\Omega w^2_\ve \dx = \alpha(1-\kappa)\int_\Oe \nabla^m\uep \cdot \nabla^m\overline{W}_\ve \dx.
\end{equation*}
To estimate the left-hand side of the previous equation, we apply H\"{o}lder inequality on the term on the right-hand side. With this aim, we first observe that $\overline{W}_\ve$ is more regular in $\Omega_\ve$ because it is the solution of a polyharmonic operator with constant coefficients. We explain carefully this fact later, in \eqref{eq:gradmWe}. By hypothesis, we choose $p$ and $q$ such that
\begin{equation}
\label{eq:eta}
p \in (2,+\infty),\,\,  q\in(1,2)\ \ \text{and}\ \ \frac{1}{p} + \frac{1}{q}=1,
\end{equation}
hence
\begin{equation}\label{sum}
 \begin{aligned}
  \norm{\wep}_{L^2(\Omega)}^2 &\leq \alpha(1-\kappa) \norm{\nabla^m \uep}_{L^q(\Oe)} \norm{\nabla^m \overline{W}_\ve}_{L^p(\Oe)} \\
  & \leq \alpha(1-\kappa) \left(\norm{\nabla^m \wep}_{L^q(\Oe)}+\norm{\nabla^m u}_{L^q(\Oe)} \right) \norm{\nabla^m \overline{W}_\ve}_{L^p(\Oe)}.
 \end{aligned}
\end{equation}
Now, we estimate the terms on the right-hand side of \eqref{sum}.\ \\
\textit{Estimate of $\norm{\nabla^m u}_{L^q(\Oe)}$}. We apply \autoref{Reg}, in fact since $\Oe \subset\subset L_0$
(cf. \autoref{fig:L}) it follows that there exists some constant $C$ such that
\begin{equation}\label{eq:gradmu}
  \norm{\nabla^m u}_{L^q(\Oe)}^q \leq \norm{\nabla^m u}_{L^\infty(L_0)}^q \left(\int_\Oe  \dx \right)
  \leq C \abs{\Oe},\qquad \textrm{that is}\qquad  \norm{\nabla^m u}_{L^q(\Oe)}\leq C \abs{\Oe}^{\frac{1}{q}}.
\end{equation}
\textit{Estimate of $\norm{\nabla^m \wep}_{L^q(\Oe)}$}. Using again H\"older's inequality with $r=\frac{2}{q} \in (1,2)$ and $s = \frac{2}{2-q}$ it follows that
\begin{equation*} 
\begin{aligned}
 \norm{\nabla^m \wep}_{L^q(\Oe)}^q &= \int_\Oe \abs{\nabla^m \wep}^q  \dx
 \leq \left(\int_\Oe \abs{\nabla^m \wep}^2\dx	\right)^{\frac{q}{2}}
       \abs{\Oe}^{\frac{1}{s}}=\|\nabla^m \wep\|^q_{L^2(\Oe)} |\Oe|^{1-\frac{q}{2}}\\
   \end{aligned}
\end{equation*}
which then together with the first, already proven, inequality in \eqref{EnEst_Hm} gives
\begin{equation}
\label{eq:n2wq}
\norm{\nabla^m \wep}_{L^q(\Oe)} \leq \norm{\nabla^m \wep}_{L^2(\Oe)}
       \abs{\Oe}^{\frac{1}{q} - \frac{1}{2}} \leq C \abs{\Oe}^{\frac{1}{q}}.
\end{equation}
\textit{Estimate of $\norm{\nabla^m \overline{W}_\ve}_{L^p(\Oe)}$}. As $\Oe\subset\subset L_0$ and $v=1$ in $L_0$, for the estimate of this term we can apply the local regularity results for polyharmonic equations with constant coefficients, see \cite{GazGruSwe10}. Indeed, since $\overline{W}_\ve$ is the weak solution of the equation \eqref{eq:We} and $w_\ve\in L^2(\Omega)$, we have that $\overline{W}_\ve\in H^{2m}(L_0)$, hence
	\begin{equation*}
	 \norm{\overline{W}_\ve}_{H^{2m}(L_0)} \leq C \norm{\wep}_{L^2(\Omega)}.
	\end{equation*}
Finally, applying the Sobolev Embedding Theorem, see \cite{Ada75,GazGruSwe10}, we find that	
	\begin{equation}\label{eq:gradmWe}
	\norm{\nabla^m\overline{W}_\ve}_{L^p(L_0)}\leq \|\overline{\Wep}\|_{W^{m,p}(L_0)} \leq C \|\overline{\Wep}\|_{H^{2m}(L_0)}\leq C \norm{\wep}_{L^2(\Omega)}.
	\end{equation}
	
Therefore, inserting \eqref{eq:gradmu}, \eqref{eq:n2wq} and \eqref{eq:gradmWe} into \eqref{sum}, we get that there exists a positive constant $C$ such that
\begin{equation} \label{EnEst_L2}
  \norm{\wep}^2_{L^2(\Omega)} \leq C \abs{\Oe}^{\frac{1}{q}}\norm{w_\ve}_{L^2(\Omega)},
  \end{equation}
which finally implies that
\begin{equation}\label{qest}
 \norm{\wep}_{L^2(\Omega)} \leq C \abs{\Oe}^{\frac{1}{q}}.
\end{equation}
So far we have found the estimate of $\wep$ in $H^m(\Omega)$ and $L^2(\Omega)$. Finally, the assertion of the theorem, i.e., the second inequality of \eqref{EnEst_Hm} follows by the application of classical interpolation inequalities in Sobolev spaces, see \cite{LioMag72a}, with the results \eqref{qest} and the first inequality in \eqref{EnEst_Hm}. Indeed, for every $k=1,\cdots,m-1$, we have that
	\begin{equation*}
		\|\wep\|_{H^{m-k}}\leq C \|\wep\|^{1-\frac{k}{m}}_{H^m(\Omega)}\|\wep\|^{\frac{k}{m}}_{L^2(\Omega)}\leq C |\Oe|^{\frac{1}{2}+\frac{k}{m}(\frac{1}{q}-\frac{1}{2})},\qquad k=1,\cdots,m-1
	\end{equation*}
which gives, together with \eqref{qest}, the assertion of the theorem.
\end{proof}
\begin{remark}
	For $q$ and $m$ fixed, it is straightforward to observe that the following relations between $\eta_{m,q,k}$ hold: $\eta_{m,q,1}\leq \eta_{m,q,2}\leq \cdots\leq \eta_{m,q,m-1}$. 
\end{remark}
We define some auxiliary functions.
\begin{definition}\label{de:vij}
In the sequel we use the following notation: $\bm{i}=(i_1,i_2,\cdots,i_m)$, where $i_k=1,2$, for $k=1,\cdots,m$ and $x_{\bm{i}}=x_{i_{1}}\cdots x_{i_{m}}$. We denote the polynomial of degree $m$ and in the variables $x_1$ and $x_2$ by
 \begin{equation}\label{Vij_formel}
 V^{i_1\cdots i_m}(x) := \frac{1}{m!} x_{i_{1}}\cdots x_{i_{m}}=\frac{1}{m!} x_{\bm{i}}
 %\frac{1}{m!} x_1^{n}x_2^{m-n}
 \ \  \text{ for }\ \ x\in\Omega.
 \end{equation}
%where $n$ is the number of indices $i_k=1$, for $k=1,\cdots,m$.
To shorten the notation we define $V^{\bm{i}}:=V^{i_1\cdots i_m}$. \\ 
These functions satisfy
   	\begin{equation}\label{eq:Vi}
   		\divz(\nabla^m V^{{\bm{i}}})= 0 \text{ in } \Omega,
	\end{equation}

Moreover, we denote by $V_\ve^{{\bm{i}}}$ the solution of
\begin{equation} \label{Vijeps_problem}
 \begin{cases}
 \vspace{0.2cm}
  \divz(\gep \nabla^m V_\ve^{{\bm{i}}}))=0 &\text{ in } \Omega,\\
  \vspace{0.2cm}
  V_\ve^{{\bm{i}}} = V^{{\bm{i}}} & \text{ on } \partial \Omega\\
  \normalderivative{V_\ve^{{\bm{i}}}}= \normalderivative{V^{{\bm{i}}}},\cdots,\normalderivativem{V_\ve^{\bm{i}}}{m-1}=\normalderivativem{V^{\bm{i}}}{m-1} &\text{ on } \partial \Omega.
 \end{cases}
\end{equation}
where
 \begin{equation}\label{gammaeps}
           \gep(x)=\begin{cases}
                    \kappa& x\in  \Oe(y,\tau),\\
                        1 & \text{ otherwise.}
                    \end{cases}
        \end{equation}
\end{definition}
Consistently we denote
 \begin{equation}\label{gammanull}
  \gamma_0=1 \text{ in } \Omega.
 \end{equation}
In the following, we provide estimates for the functions $V_\ve^{{\bm{i}}}$ and $V^{{\bm{i}}}$.
\begin{lemma}
\label{EnEstV}
Under the assumptions of Lemma \ref{EnEst}, there exists some constant $C$ independent of $\ve$ such that
\begin{equation}\label{EnH2V}
  \norm{V_\ve^{{\bm{i}}}-V^{{\bm{i}}}}_{H^m(\Omega)} \leq C\abs{\Oe}^{\frac{1}{2}} \text{ and }
  \norm{V_\ve^{{\bm{i}}}-V^{{\bm{i}}}}_{H^{m-k}(\Omega)} \leq C\abs{\Oe}^{\frac{1}{2}+\eta_{m,q,k}},
\end{equation}
for any $k=1,\cdots,m$.
\end{lemma}
The proof is similar to the proof of \autoref{EnEst}.

\begin{lemma} \label{Equality}
Let $\eta_{m,q,1}$ be defined as in Lemma \ref{EnEst}. Then, for all test functions $\phi \in C_0^m(L_0)$ the following identity holds:
\begin{equation} \label{eq:approx_weak}
\abs{\int_\Oe (\nabla^m\uep \cdot \nabla^m V^{{\bm{i}}}) \phi \dx - \int_\Oe (\nabla^m V_\ve^{{\bm{i}}}\cdot \nabla^m u) \phi \dx}
= \mathcal{O}(\abs{\Oe}^{1+\eta_{m,q,1}}).
\end{equation}
\end{lemma}
\begin{proof}
We first observe that $\gamma_\ve=v_\ve$ in $L_0$. We use the weak formulations \eqref{EnEst_eq1a} and \eqref{EnEst_eq1} of $\wep$ where we choose $\varphi=V^{\bm{i}}\phi$ and $\varphi=V_\ve^{\bm{i}}\phi$, respectively, i.e., 
    \begin{equation} \label{EnEstV_eq1a}
    \int_\Omega \wep V^{\bm{i}}\phi  \dx
    + \alpha\int_\Omega v \nabla^m \wep\cdot \nabla^m (V^{\bm{i}}\phi)  \dx
    = \alpha(1-\kappa) \int_\Oe \nabla^m \uep \cdot \nabla^m(V^{\bm{i}}\phi) \dx,
    \end{equation}
    and
    \begin{equation} \label{EnEstVe_eq1}
    \int_\Omega \wep V_\ve^{\bm{i}}\phi \dx + \alpha \int_\Omega \vep \nabla^m \wep\cdot \nabla^m(V_\ve^{\bm{i}}\phi) \dx
    = \alpha(1-\kappa) \int_\Oe \nabla^m u \cdot \nabla^m(V_\ve^{\bm{i}}\phi) \dx.
    \end{equation}
Then subtracting \eqref{EnEstV_eq1a} from \eqref{EnEstVe_eq1}, and since $\phi$ has compact support in $L_0$, we find
	\begin{equation*}
	\begin{aligned}
		\int_{L_0}\wep (V^{\bm{i}}_\ve-V^{\bm{i}})\phi\, dx &+\alpha\int_{L_0}\vep \nabla^m \wep \cdot \nabla^m(V^{\bm{i}}_\ve\phi)\, dx - \alpha\int_{L_0} \nabla^m \wep \cdot \nabla^m(V^{\bm{i}}\phi)\, dx\\
		&=\alpha(1-\kappa)\int_{\Omega_\ve}\left[\nabla^m u\cdot \nabla^m(V^{\bm{i}}_\ve\phi)-\nabla^m\uep \cdot \nabla^m(V^{\bm{i}}\phi)\,\right] dx. 
	\end{aligned}	
	\end{equation*}
In the last expression we explicit the terms containing the derivative of maximum order of all the functions, i.e.,
	\begin{equation}\label{eq:Ve-V}
		\begin{aligned}
			&\int_{L_0}\wep (V^{\bm{i}}_\ve -V^{\bm{i}})\phi\, dx+\alpha\int_{L_0}\vep\nabla^m \wep \cdot \nabla^m V^{\bm{i}}_\ve \phi\, dx + \alpha \int_{L_0}\vep \nabla^m \wep \cdot \nabla^m \phi V^{\bm{i}}_\ve\, dx-\alpha\int_{L_0}\nabla^m \wep \cdot \nabla^m V^{\bm{i}}\phi\,dx \\
			&+\alpha\int_{L_0}\vep \nabla^m \wep \cdot\left[\sum_{n=1}^{m-1}\nabla^{m-1-n}\left(\nabla^nV^{\bm{i}}_\ve \otimes \nabla \phi + \nabla^n\phi \otimes \nabla V^{\bm{i}}_\ve\right)\right]\, dx	-\alpha \int_{L_0}\nabla^m \wep \cdot \nabla^m \phi V^{\bm{i}}\, dx\\
			&-\alpha\int_{L_0}\nabla^m \wep \cdot \left[\sum_{n=1}^{m-1}\nabla^{m-1-n}\left(\nabla^nV^{\bm{i}} \otimes \nabla \phi + \nabla^n\phi \otimes \nabla V^{\bm{i}}\right)\right]\, dx\\
			&=\alpha(1-\kappa)\Bigg\{\int_{\Omega_\ve}\nabla^m u \cdot \nabla^m V^{\bm{i}}_\ve \phi\, dx + \int_{\Omega_\ve}\nabla^m u \cdot \nabla^m \phi V^{\bm{i}}_\ve\, dx -\int_{\Omega_\ve}\nabla^m \uep \cdot \nabla^m V^{\bm{i}} \phi\, dx - \int_{\Omega_\ve} \nabla^m \uep \cdot \nabla^m\phi V^{\bm{i}}\, dx \\
			&+\int_{\Omega_\ve}\nabla^m u\cdot \left[\sum_{n=1}^{m-1}\nabla^{m-1-n}\left(\nabla^nV^{\bm{i}}_\ve \otimes \nabla \phi + \nabla^n\phi \otimes \nabla V^{\bm{i}}_\ve\right)\right]\, dx\\
			&-\int_{L_0}\nabla^m \uep \cdot \left[\sum_{n=1}^{m-1}\nabla^{m-1-n}\left(\nabla^nV^{\bm{i}} \otimes \nabla \phi + \nabla^n\phi \otimes \nabla V^{\bm{i}}\right)\right]\, dx
			\Bigg\}		
		\end{aligned}
	\end{equation}     
We next use the weak formulation of the equations satisfied by $V^{\bm{i}}$ and $V^{\bm{i}}_\ve$, see \eqref{eq:Vi} and \eqref{Vijeps_problem}, in $L_0$. The idea is to find an expression of the terms in $\wep$ with the maximum order of derivation, which are in the left-hand side of the previous formula, in terms of the derivatives of lower order with respect to $\wep$. Specifically, for all test functions $\wep \phi$, we have
	\begin{equation*}
			\int_{L_0} v_\ve \nabla^m V^{\bm{i}}_\ve \cdot \nabla^m(\wep\phi)\, dx=0,\qquad \textrm{and}\qquad \int_{L_0}\nabla^m V^{\bm{i}} \cdot \nabla^m (\wep\phi)\, dx=0,
	\end{equation*} 
hence, expliciting the terms $\nabla^m(\wep\phi)$, we get for the first equation
	\begin{equation}\label{eq:gradmVe}
	\begin{aligned}
		\int_{L_0} v_\ve \nabla^m V^{\bm{i}}_\ve \cdot \nabla^m \wep \phi\, dx&=-\int_{L_0} v_\ve \nabla^m V^{\bm{i}}_\ve \cdot \nabla^m \phi \wep\, dx \\
		&- \int_{L_0} v_\ve \nabla^m V^{\bm{i}}_\ve \cdot \left[ \sum_{n=1}^{m-1}\nabla^{m-1-n}\left(\nabla^n \wep \otimes \nabla \phi + \nabla^n \phi \otimes \nabla \wep\right)\right]\, dx
	\end{aligned}
	\end{equation}
and, analogously for the second equation,
 	\begin{equation}\label{eq:gradmV}
 	\begin{aligned}
 		\int_{L_0}  \nabla^m V^{\bm{i}} \cdot \nabla^m \wep \phi\, dx&=-\int_{L_0} \nabla^m V^{\bm{i}}_\ve \cdot \nabla^m \phi \wep\, dx \\
 		&- \int_{L_0} \nabla^m V^{\bm{i}} \cdot \left[ \sum_{n=1}^{m-1}\nabla^{m-1-n}\left(\nabla^n \wep \otimes \nabla \phi + \nabla^n \phi \otimes \nabla \wep\right)\right]\, dx.
 	\end{aligned}
 	\end{equation}
We use \eqref{eq:gradmVe} and \eqref{eq:gradmV} in \eqref{eq:Ve-V}. Then, in the resulting equation, we add and subtract $V^{\bm{i}}$ in each term containing $V^{\bm{i}}_\ve$ (but not in the terms containing already the difference $V^{\bm{i}}_\ve-V^{\bm{i}}$ and not in the first integral in the left-hand side of the equality \eqref{eq:Ve-V}). After some calculations and simplifications, we find
	\begin{equation*}
	\begin{aligned}
		&\alpha(1-\kappa)\int_{\Oe} (\nabla^m u\cdot \nabla^m V^{\bm{i}}_\ve -\nabla^m \uep \cdot \nabla^m V^{\bm{i}})\phi\, dx\\
		&=\int_{L_0}\wep(V^{\bm{i}}_\ve-V^{\bm{i}})\phi\, dx
		- \alpha\Bigg\{\int_{L_0} \vep\nabla^m(V^{\bm{i}}_\ve-V^{\bm{i}})\cdot \nabla^m \phi \wep \, dx
		-\int_{L_0} \vep \nabla^m\wep \cdot \nabla^m \phi (V^{\bm{i}}_\ve-V^{\bm{i}})\, dx\\
		&+ \int_{L_0} \vep \nabla^m(V^{\bm{i}}_\ve-V^{\bm{i}})\cdot\left[\sum_{n=1}^{m-1}\nabla^{m-1-n}\left(\nabla^n\wep \otimes\nabla\phi+\nabla^n\phi\otimes\nabla\wep\right)\right]\, dx\\
%		&+ \int_{L_0} \vep \nabla^m V^{\bm{i}}\cdot\left[\sum_{n=1}^{m-1}\nabla^{m-1-n}\left(\nabla^n\wep \otimes\nabla\phi+\nabla^n\phi\otimes\nabla\wep\right)\right]\, dx\\
		&- \int_{L_0} \vep \nabla^m \wep \cdot\left[\sum_{n=1}^{m-1}\nabla^{m-1-n}\left(\nabla^n(V^{\bm{i}}_\ve-V^{\bm{i}})\otimes\nabla\phi+\nabla^n\phi\otimes\nabla(V^{\bm{i}}_\ve-V^{\bm{i}})\right)\right]\, dx\\
		&+(\kappa-1) \left[\int_{\Oe} \nabla^mV^{\bm{i}}\cdot\left(\sum_{n=1}^{m-1}\nabla^{m-1-n}\left(\nabla^n\wep \otimes\nabla\phi+\nabla^n\phi\otimes\nabla\wep\right)\right)\, dx + \int_{\Oe} \nabla^m V^{\bm{i}} \cdot \nabla^m \phi \wep\, dx\right] \Bigg\}\\
		&-\alpha(1-\kappa)\Bigg\{
		 \int_{\Oe} \nabla^m u\cdot \nabla^m \phi (V^{\bm{i}}_\ve-V^{\bm{i}})\, dx \\ 		
		&+ \int_{\Oe} \nabla^m u\cdot\left[\sum_{n=1}^{m-1}\nabla^{m-1-n}\left(\nabla^n(V^{\bm{i}}_\ve-V^{\bm{i}})\otimes \nabla \phi
		+\nabla^n\phi \otimes \nabla(V^{\bm{i}}_\ve-V^{\bm{i}})\right)\right]\, dx
		\Bigg\}.
	\end{aligned}
	\end{equation*} 
To get the assertion of the theorem, i.e., the estimate of the integral on the left-hand side of the previous equality, in terms of $|\Oe|^{1+\eta_{m,q,1}}$, is a lengthy check. It is based on the application, on the terms in the right-hand side, of the Cauchy-Schwarz inequality, the fact that $\|\nabla^n \phi\|_{L^{\infty}(L_0)}\leq C$, for all $n:\ 0\leq n\leq m-1$, and $|v_\ve|\leq 1$ in $L_0$, and the estimates \eqref{EnEst_Hm}, and \eqref{EnH2V}. We do not make explicit the calculations but we only notice that all these integrals can be estimated in terms of the product $|\Oe|^{\frac{1}{2}}$ and $|\Oe|^{\frac{1}{2}+\eta_{m,q,1}}$.  
\end{proof}

\begin{remark}\label{rem:sub_ue_Ve}
The identity \eqref{eq:approx_weak} holds also with $\uep$ replaced by $V_\ve^{\bm{h}}$ and $u$ replaced by $V^{\bm{h}}$, respectively,
 since in $L_0$ the pairs satisfy the same equations.
\end{remark}
\begin{remark}
From \autoref{EnEst} it follows that
\begin{equation}\label{EnEst_Hm_o}
\norm{\wep}_{H^m(\Omega)} = \mathcal{O}(\ve^{\frac{3}{2}})\quad \text{ and }\quad
\norm{\wep}_{H^{m-k}(\Omega)} = \mathcal{O}(\ve^{\frac{3}{2}+3\eta_{m,q,k}}),\qquad \forall k=1,\cdots,m-1.
\end{equation}
Moreover, from \eqref{EnH2V}, it follows that
\begin{equation}\label{Enh2V_o}
\norm{V_\ve^{\bm{i}}-V^{\bm{i}}}_{H^m(\Omega)} = \mathcal{O}(\ve^{\frac{3}{2}}) \text{ and }
\norm{V_\ve^{\bm{i}}-V^{\bm{i}}}_{H^{m-k}(\Omega)} = \mathcal{O}(\ve^{\frac{3}{2}+3\eta_{m,q,k}}),\qquad \forall k=1,\cdots,m-1.
\end{equation}
\end{remark}

%%% \end{asymptotic}
\section{Definition and properties of the $2m$-order tensor}
\label{sec:pol}
We start this section with the definition of the tensor $\tensor$.
We note that the family of functions $\left(\frac{1}{\abs{\Oe}}\chi_\Oe \right)_{\ve > 0}$ is uniformly bounded in
$L^1(\Omega)$. Hence, by the Riesz's Representation theorem, we have that  the family  of measures
\begin{equation}
 d \mu_\ve :=\frac{1}{\abs{\Oe}}\chi_\Oe \dx
\end{equation}
is bounded in $C^0(\overline{\Omega})^*$
and hence by Banach-Alaoglu theorem (see for instance \cite{Rud87}), possibly up to the extraction of a subsequence,
we have that
\begin{equation}
 d \mu_\ve :=\frac{1}{\abs{\Oe}}\chi_\Oe \dx \to^\ast \dmu \quad \text{ for } \ve \to 0
\end{equation}
where $\to^\ast$ denotes the weak$^{\ast}$-convergence of $C^0(\overline{\Omega})^*$, that is
\begin{equation*}
\int_\Omega \frac{1}{\abs{\Oe}} \chi_\Oe \widetilde{\phi}\dx \to \int_\Omega \widetilde{\phi}\dmu \text{ for } \ve \to 0\quad
\text{ for all }\ \widetilde{\phi}\in C^0(\overline{\Omega}).
\end{equation*}
It is also immediate to see that, due to the form of $\Oe$ the measure  $\mu$ is concentrated at the point $y$ i.e. $\mu=\delta_y$.
Analogously, using  the energy estimates \eqref{EnH2V}, it follows that the family of functions
$\left( \frac{1}{\abs{\Oe}}\frac{\partial^m V_\ve^{\bm{i}}}{\partial x_{j_1}\cdots \partial x_{j_m}}\chi_\Oe \right)_{\ve > 0}$
is uniformly bounded in $L^1(\Omega)$ and therefore the family of measures
\begin{equation*}
 d\nu_\ve = \frac{1}{\abs{\Oe}} \frac{\partial^m V_\ve^{\bm{i}}}{\partial x_{j_1}\cdots \partial x_{j_m}} \chi_\Oe \dx
\end{equation*}
converges, possibly up to subsequences, to a Borel measure
\begin{equation} \label{partialVeps}
 d\nu_\ve = \frac{1}{\abs{\Oe}} \frac{\partial^m V_\ve^{\bm{i}}}{\partial x_{j_1}\cdots \partial x_{j_m}} \chi_\Oe \dx
          \to^\ast %d\tensor_{ijhk} 
          {\tensor_{i_1\cdots i_mj_1\cdots j_m} \delta_y}\,\,
          \text{ for } \ve \to 0,
\end{equation}
which is equivalent to say that 	
\begin{equation}\label{convV}
\frac{1}{\abs{\Oe}} \int_\Omega \frac{\partial^m V_\ve^{\bm{i}}}{\partial x_{j_1}\cdots \partial x_{j_m}} \chi_\Oe \widetilde{\phi} \dx
          \to \tensor_{i_1\cdots i_mj_1\cdots j_m}\widetilde{\phi}(y), \quad \text{ for all }\  \widetilde{\phi}\in C^0(\overline{\Omega}).
\end{equation}
To shorten the notation we define $\tensor_{\bm{i}\bm{j}}:=\tensor_{i_1\cdots i_mj_1\cdots j_m}$.

Before putting together the results in \autoref{Equality} and \eqref{convV}, we observe a local regularity results on $\uep$ and $V^{\bm{i}}_\ve$.
\begin{remark}\label{rem:reg_ue_Ve}
	We note that the function $\uep$ is $C^m(\Omega_\ve)$ since it satisfies, in the open set $\Oe$, a non homogeneous polyharmonic equation with a forcing term in $L^\infty$ (by assumption). Indeed, this result follows from local regularity theorems, see \cite{GazGruSwe10}, for which $\uep\in H^{2m}(\Oe)$, and consequently from the application of the embeddings theorems.\\
	This result also holds for $V^{\bm{i}}_\ve$ since these functions satisfy a homogeneous polyharmonic equation.
	
	From this regularity result, it follows that both $\nabla^m \uep$ and $\nabla^m V^{\bm{i}}_\ve$ are symmetric tensors in $\Oe$. 
\end{remark}
Therefore from the symmetries of $\nabla^m \uep$ in $\Oe$, for all $\phi\in C^m_0(L_0)$, by \autoref{Equality} and using \eqref{convV} (where we choose $\widetilde{\phi}=\nabla^m u \phi$), we also have that
\begin{equation}\label{convuep}
\begin{aligned}
\int_\Omega \frac{1}{\abs{\Oe}} \chi_\Oe \nabla^m \uep \cdot \nabla^m V^{\bm{i}} \phi\dx&\\
&\hspace{-2cm}=\int_\Omega \frac{1}{\abs{\Oe}} \chi_\Oe \frac{\partial^m \uep}{\partial  x_{i_1}\cdots \partial x_{i_m}} \phi\dx \to  
    \tensor\nabla^mu(y)\phi(y),\quad \textrm{for}\ \ \ve \to 0
\end{aligned}
\end{equation}
\begin{remark}\label{rem:ueps_weakstar}
 We notice that $\left(\frac{1}{\abs{\Oe}} \chi_\Oe \nabla^m \uep \cdot \nabla^m V^{\bm{i}}\right)_{\ve>0}$ 
 is uniformly bounded in $L^1(\Omega)$ due to \eqref{Vij_formel} and the use of energy estimates \eqref{EnEst_Hm} after summing and subtracting $u$ in $\nabla^m u_\ve$. Therefore this sequence converges in the weak* topology of $C(\overline{\Omega})$, up to a subsequence, to a Borel measure, i.e.,
 	\begin{equation}\label{eq:weaks_ueps}
 		\frac{1}{\abs{\Oe}} \chi_\Oe \nabla^m \uep \cdot \nabla^m V^{\bm{i}}\, dx\to^\ast d\mu_{\bm{i}}.
 	\end{equation} 	
Defining $L_1$ such that $\Omega_\ve\subset L_1\subset L_0$ and exploiting  \eqref{eq:weaks_ueps} and \eqref{convuep}, we deduce that 
	\begin{equation}
		\int_{\Omega}\phi\, d\mu_{\bm{i}}=\tensor\nabla^mu(y)\phi(y),\qquad \forall \phi\in C^m(\overline{L_1}),
	\end{equation}
and by the density of $C^m(\overline{L_1})$ in $C^0(\overline{L_1})$, we have that
	\begin{equation*}
		d\mu_{\bm{i}}=\mathbb{M}_{\bm{i}j_1\cdots j_m} \frac{\partial u}{\partial x_{j_1}\cdots \partial x_{j_m}} \delta_y.
	\end{equation*}		
\end{remark}
Let us now state some key properties of the tensor $\tensor$.
\begin{proposition}
\label{PTproperty}
We denote with $\sigma(\bm{i})$ any permutation of the indices $i_1,\cdots,i_m$. The $2m$-order tensor $\tensor$ has the full symmetries, i.e.,
\begin{equation}
\label{PTprop1}
\tensor_{\bm{i}\bm{j}}=\tensor_{\sigma(\bm{i})\bm{j}}=\tensor_{\bm{i}\sigma(\bm{j})}=\tensor_{\bm{j}\bm{i}}
\end{equation}
\end{proposition}
\begin{proof}
The first equality in  \eqref{PTprop1} follows immediately from the fact that $V^{\bm{i}}=V^{\sigma(\bm{i})}$, which implies $V_\ve^{\bm{i}}=V_\ve^{\sigma(\bm{i})}$ (from the uniqueness result for \eqref{Vijeps_problem}). The second identity follows from the regularity results for $V^{\bm{i}}_\ve$, see Remark \ref{rem:reg_ue_Ve}, which implies that we can interchange the order of differentiation in \eqref{convV}. \\
The identity $\tensor_{\bm{i}\bm{j}}=\tensor_{\bm{j}\bm{i}}$ follows substituting $u$ and $\uep$ with $V^{\bm{j}}$ and $V_\ve^{\bm{j}}$, respectively, in \autoref{Equality}, see Remark \ref{rem:sub_ue_Ve}, and using the symmetry of the tensor $\nabla^m V^{\bm{i}}_\ve$, see Remark \ref{rem:reg_ue_Ve}. In fact
\begin{equation}\label{eq:Ve_V_interch}
\frac{1}{\abs{\Oe}}\int_\Oe \nabla^m V_\ve^{\bm{i}} \cdot \nabla^m V^{\bm{j}} \phi \dx
= \frac{1}{\abs{\Oe}}\int_\Oe \nabla^m V_\ve^{\bm{j}} \cdot \nabla^m V^{\bm{i}} \phi  \dx + \mathcal{O}(\abs{\Oe}^{\eta_{m,q,1}}),
\end{equation}
for all test-functions $\phi \in C_0^m(L_0)$.

Therefore, using \eqref{partialVeps} and Remark \ref{rem:reg_ue_Ve}, we get for the left-hand side
\begin{equation*}
\label{eq:limit}
\frac{1}{\abs{\Oe}} \int_\Oe \nabla^m V_\ve^{\bm{i}} \cdot \nabla^m V^{\bm{j}} \phi \dx
=\frac{1}{\abs{\Oe}} \int_\Oe \frac{\partial^m V_\ve^{\bm{i}}}{\partial x_{j_1}\cdots \partial x_{j_m}}\phi  \dx
\to
\tensor_{\bm{i}\bm{j}}\phi(y).
\end{equation*}
On the other hand, the first term in the right-hand side of \eqref{eq:Ve_V_interch} gives
\begin{equation}
\label{eq:limit2}
\frac{1}{\abs{\Oe}} \int_\Oe \nabla^m V_\ve^{\bm{j}} \cdot \nabla^m V^{\bm{i}} \phi \dx
=\frac{1}{\abs{\Oe}} \int_\Oe \frac{\partial^m V_\ve^{\bm{j}}}{\partial x_{i_1}\cdots \partial x_{i_m}}\phi  \dx
\to
\tensor_{\bm{j}\bm{i}}\phi(y),
\end{equation}
that is the assertion.
% because $|\Oe|^{\eta_{m,q,1}}\to 0$ as $\ve\to 0$.
\end{proof}
By the symmetry properties of $\tensor$ we can consider, without loss of generality, the space of symmetric tensor of order $m$, i.e., 
	\begin{equation*}
		S^m(\mathbb{R}^2):=\Big\{A \ \textrm{such that}\ \ A_{\bm{i}}=A_{\sigma(\bm{i})}   \Big\}.
	\end{equation*}
%A=\sum_{i=(i_1,\cdots,i_m)=1}^{2} A_{i_1,\cdots,i_m} e_{i_1}\otimes\cdots \otimes e_{i_m},	
%where each $e_{i_k}$, for $k=1,\cdots,m$, is equal either to $e_1=(1,0)$ or $e_2=(0,1)$.
Therefore, we have that $\tensor:S^m(\mathbb{R}^2)\rightarrow S^m(\mathbb{R}^2)$.	
For a review on some results on symmetric tensors of generic order, their properties, their decomposition and their relations with hypermatrices see \cite{ComGolLimMou08} \cite{KoldBad09} and references therein. 

We introduce some auxiliary functions which are needed to get some bounds on the tensor of order $2m$. \\
\begin{definition}
Given $E\in S^m(\mathbb{R}^2)$, we define the auxialiary functions 
\begin{equation} \label{eq:VVE}
V := \sum_{\textbf{i}=(i_1,\cdots,i_m)=1}^{2} E_{\bm{i}}V^{\bm{i}}\qquad \text{ and }\qquad V_\ve:= \sum_{\bm{i}=(i_1,\cdots,i_m)=1}^{2} E_{\bm{i}}V_\ve^{\bm{i}}.
\end{equation}
\end{definition}
From \eqref{eq:Vi} and \eqref{Vijeps_problem}, we find that by construction $V$ and $V_\ve$ are solutions of
\begin{equation} \label{V_problem}
\begin{cases}
\vspace{0.2cm}
\divz (\nabla^m V) = 0 &\text{ in } \Omega,\\
\vspace{0.2cm}
V = \sum\limits_{\bm{i}=1}^{2} E_{\bm{i}}V^{\bm{i}} & \text{ on } \partial \Omega,\\
\normalderivative{V}=\frac{\partial}{\partial n}\left(\sum\limits_{\bm{i}=1}^{2}E_{\bm{i}}V^{\bm{i}}\right),\cdots, \normalderivativem{V}{m-1}=\frac{\partial^{m-1}}{\partial n^{m-1}}\left(\sum\limits_{\bm{i}=1}^{2}E_{\bm{i}}V^{\bm{i}}\right) &  \text{ on } \partial \Omega,
\end{cases}	
\end{equation}
and
\begin{equation} \label{Veps_problem}
\begin{cases}
\vspace{0.2cm}
\divz (\gep \nabla^m V_\ve)=0, & \text{ in } \Omega\\
\vspace{0.2cm}
V_\ve=V &\text{ on } \partial \Omega,\\
\normalderivative{V_\ve}=\normalderivative{V},\cdots,\normalderivativem{V_\ve}{m-1}=\normalderivativem{V}{m-1} &\text{ on } \partial \Omega,
\end{cases}
\end{equation}
respectively.
\begin{remark}
We observe that from the definition of $V^{\bm{i}}$, see \eqref{Vij_formel}, it follows 
\begin{equation}\label{eq:E}
\nabla^m V=\sum_{\bm{i}=1}^{2} E_{\bm{i}}\nabla^m V^{\bm{i}}=E.
\end{equation}
In fact, from \eqref{Vij_formel}, since $x_{i_l}=x_1$ or $x_{i_l}=x_2$, for $l=1,\cdots,m$, we have that
	\begin{equation*}
		x_{i_1}x_{i_2}\cdots x_{i_m}=x_1^{m-h}x_2^h,
	\end{equation*} 
where $h$ is the number of $x_{i_l}$, for $l=1,\cdots,m$ equal to $x_1$. In this way, we immediately get that
	\begin{equation}\label{eq:part_der_Vi}
		\frac{\partial^m V^{\bm{i}}}{\partial x_{j_1}\cdots \partial x_{j_m}}=\frac{1}{m!}\frac{\partial^m (x_{i_1}x_{i_2}\cdots x_{i_m})}{\partial x_{j_1}\cdots \partial x_{j_m}}=\frac{1}{m!}\frac{\partial^m(x_1^{m-h}x_2^h)}{\partial x_1^{m-k}\partial x_2^k}=
		\begin{cases}
		\frac{(m-k)!k!}{m!} & \textrm{if}\ \ h=k\\
		0 & \textrm{if}\ \ h\neq k.
		\end{cases}
	\end{equation}	
Then, since for hypothesis $E_{i_1i_2\cdots i_m}$ is a symmetric tensor, it follows that the number of tensors $E_{i_1\cdots i_m}=E_{\sigma(i_1\cdots i_m)}$, in the sum of \eqref{eq:E}, which have $m-k$ components equal to $1$ and $k$ components equal to $2$ is given by $\binom{m}{k}$. This consideration and \eqref{eq:part_der_Vi} give the identity in \eqref{eq:E}. 
%Therefore, it is now straightforward to derive that
%	\begin{equation*}
%		(\nabla^m V)_{\small\underbrace{1\cdots1}\limits_{m-k} \underbrace{2\cdots 2}\limits_{k}}=\frac{\partial^m V}{\partial x_1^{m-k}x_2^k}=
%		E_{\small\underbrace{1\cdots1}\limits_{m-k} \underbrace{2\cdots 2}\limits_{k}},\qquad\qquad \textrm{for}\ \ k=0,\cdots,m. 
%	\end{equation*}		
\end{remark}

\begin{definition}
Given $V$ and $V_\ve$  as in \eqref{eq:VVE}, we define 
	\begin{equation*}
		\Wep := V_\ve-V.
	\end{equation*}		
\end{definition}
By means of the difference of \eqref{V_problem} and \eqref{Veps_problem}, and then adding and subtracting $\divz(\gep\nabla^m V)$, according to \eqref{eq:ve}, we find that $\Wep$ satisfies
\begin{equation} \label{Weps_problem}
\begin{cases}
\vspace{0.2cm}
\divz  (\gep \nabla^m \Wep)= \divz  ((1-\kappa)\chi_{\Oe}\nabla^m V)  & \text{ in } \Omega,\\
\vspace{0.2cm}
\Wep=0 & \text{ on } \partial \Omega,\\
\normalderivative{\Wep}=\cdots=\normalderivativem{\Wep}{m-1}=0 &\text{ on } \partial \Omega.
\end{cases}
\end{equation}
	
Using these auxiliary functions, in the following, we determine some sharp bounds on $\tensor$.
\begin{proposition} \label{BoundsPT}
The $2m$-order tensor $\tensor$  is positive definite and satisfies
\begin{equation*}
\abs{E}^2 \leq \tensor  E\cdot E\leq \frac{1}{\kappa} \abs{E}^2, \,\,\text{ for all } E\in S^m(\R^{2}) .
\end{equation*}
\end{proposition}
\begin{proof}
The starting point to prove this proposition is the definition of the tensor $\tensor_{\bm{i}\bm{j}}$, see \eqref{convV}, which we specialize in the case where the test function $\phi\in C^m(\overline{\Omega})$. This choice of regularity will be clear in the second part of the proof where is needed in order to get uniform estimates. By \eqref{convV}, we have that
\begin{equation} \label{eq:M}
 \tensor_{\bm{i}\bm{j}}\phi(y) =
 \lim_{\ve \to 0} \frac{1}{\abs{\Oe}} \int_\Oe \frac{\partial^m V_\ve^{\bm{i}}}{\partial x_{j_1}\cdots \partial x_{j_m}} \phi \dx,\qquad
 \text{ for all } \phi \in C^m(\overline{\Omega}).
\end{equation}
Applying this identity with the test function $E_{\bm{i}} E_{\bm{j}} \phi$, with $\phi\in C^m(\overline{\Omega})$, it follows
from \eqref{eq:M} that
\begin{equation} \label{eq:MEE}
 \begin{aligned}
   \tensor E\cdot E \phi(y)
   = \sum_{\bm{i},\bm{j}=1}^{2} E_{\bm{i}} \tensor_{\bm{i}\bm{j}} E_{\bm{j}} \phi(y)
   = \sum_{\bm{i},\bm{j}=1}^{2}\Big\{\lim_{\ve \to 0} E_{\bm{i}} \frac{1}{\abs{\Oe}} \int_\Oe
      \frac{\partial^m V_\ve^{\bm{i}}}{\partial x_{j_1}\cdots\partial x_{j_m}}  E_{j_1\cdots j_m}\phi \dx \Big\}.
 \end{aligned}
\end{equation}
Using \eqref{eq:E} in \eqref{eq:MEE}, we can rewrite $E_{j_1\cdots j_m}=\frac{\partial^m V}{\partial x_{j_1}\cdots \partial x_{j_m}}$ that is it follows:
\begin{equation} \label{eq:can_write}
 \begin{aligned}
  & 
  \tensor E\cdot E \phi(y)
  = \lim_{\ve \to 0} \frac{1}{\abs{\Oe}}\int_\Oe
     \sum_{\bm{i},\bm{j}=1}^{2}  E_{\bm{i}} \frac{\partial^m V_\ve^{\bm{i}}}{\partial x_{j_1}\cdots\partial x_{j_m}} \frac{\partial^m V}{\partial x_{j_1}\cdots \partial x_{j_m}}
     \phi  \dx \\
  = & \lim_{\ve \to 0} \frac{1}{\abs{\Oe}}\int_\Oe \nabla^m \Wep \cdot \nabla^m V \phi  \dx + \lim_{\ve \to 0} \frac{1}{\abs{\Oe}}\int_\Oe \abs{\nabla^m V}^2 \phi  \dx
 \end{aligned}
\end{equation}
where in the righ-hand side of the second equality we have added and subtracted $V^{\bm{i}}$ in the term containing $V^{\bm{i}}_\ve$ and then we have used \eqref{eq:VVE}, i.e., the fact that $\sum_{\bm{i}=2}^{2}E_{\bm{i}}W^{\bm{i}}_\ve=W_\ve$. 
Now, we derive an expansion for the first limit in the right-hand side of the second equality of \eqref{eq:can_write}, i.e., to be more precise, we prove that 
	\begin{equation}\label{eq:gradmVgradmWe_res}
		\int_{\Oe}\nabla^m W_\ve\cdot \nabla^m V \phi\, dx=\frac{1}{1-\kappa}\int_{\Omega}\gamma_\ve|\nabla^m W_\ve|^2 \phi\, dx + \mathcal{O}(|\Oe|^{1+\eta_{m,q,1}}).
	\end{equation}
With this aim, we use the weak formulation of the problem \eqref{Weps_problem}, i.e., for all $\varphi\in H^m_0(\Omega)$
	\begin{equation}\label{eq:weakWe}
		\int_{\Omega}\gamma_\ve \nabla^m W_\ve \cdot \nabla^m \varphi\, dx= (1-\kappa)\int_{\Oe}\nabla^m V \cdot \nabla^m \varphi\, dx.
	\end{equation}
In this last equation, we choose $\varphi=W_\ve \phi$, where $\phi\in C^m(\overline{\Omega})$, hence
	\begin{equation*}
	\int_{\Omega}\gamma_\ve \nabla^m W_\ve \cdot \nabla^m(W_\ve \phi)\, dx= (1-\kappa)\int_{\Oe}\nabla^m V \cdot \nabla^m(W_\ve \phi)\, dx.
	\end{equation*}
Expliciting the term of higher order derivative with respect to $W_\ve$, we get
	\begin{equation*}
		\begin{aligned}
		\int_{\Omega} \gamma_\ve \nabla^m W_\ve \cdot \nabla^m W_\ve \phi\, dx &+ \int_{\Omega}\gamma_\ve \nabla^m W_\ve \cdot \nabla^m \phi W_\ve\, dx\\
		&+\int_{\Omega}\gamma_\ve \nabla^m W_\ve \cdot \left[\sum_{n=1}^{m-1}\nabla^{m-1-n}\left(\nabla^n W_\ve \otimes \nabla \phi + \nabla^n \phi \otimes \nabla W_\ve\right)\right]\, dx\\
		&\hspace{-4cm}=(1-\kappa)\Bigg\{\int_{\Oe}\nabla^m V \cdot \nabla^m W_\ve \phi\, dx + \int_{\Oe}\nabla^m V \cdot \nabla^m \phi W_\ve\, dx \\
		&+ \int_{\Oe}\nabla^m V \cdot \left[ \sum_{n=1}^{m-1}\nabla^{m-1-n}\left(\nabla^n W_\ve \otimes \nabla \phi + \nabla^n \phi \otimes \nabla W_\ve\right)\right]\, dx   
		\Bigg\}. 
		\end{aligned}
	\end{equation*} 
Therefore, we can rewrite the previous equality as
	\begin{equation}\label{eq:gradmVgradmWe}
	\begin{aligned}
		(1-\kappa)\int_{\Oe}\nabla^m V \cdot \nabla^m W_\ve \phi\, dx&- \int_{\Omega} \gamma_\ve |\nabla^m W_\ve|^2\phi\, dx\\
		&\hspace{-3.3cm}=\int_{\Omega}\gamma_\ve \nabla^m W_\ve \cdot \nabla^m \phi W_\ve\, dx 
		 +\int_{\Omega}\gamma_\ve \nabla^m W_\ve \cdot \left[\sum_{n=1}^{m-1}\nabla^{m-1-n}\left(\nabla^n W_\ve \otimes \nabla \phi + \nabla^n \phi \otimes \nabla W_\ve\right)\right]\, dx\\
		&-(1-\kappa)\Bigg\{\int_{\Oe}\nabla^m V \cdot \left[ \sum_{n=1}^{m-1}\nabla^{m-1-n}\left(\nabla^n W_\ve \otimes \nabla \phi + \nabla^n \phi \otimes \nabla W_\ve\right)\right]\, dx\\
		&\hspace{1.7cm}+\int_{\Oe}\nabla^m V \cdot \nabla^m \phi W_\ve\, dx\Bigg\}=:I_1+I_2+I_3+I_4.
	\end{aligned}	
	\end{equation}
As already done in the proof of Lemma \ref{Equality}, we can estimate each integral, in the right-hand side of the previous formula, in terms of the product of $|\Oe|^{\frac{1}{2}}$ and $|\Oe|^{\frac{1}{2}+\eta_{m,q,1}}$ by using the Cauchy-Schwarz inequality, the fact that $\phi\in C^m(\overline{\Omega})$, and the estimates in Lemma \ref{EnEstV}. \\
\textit{Estimate of $I_1$:}
\begin{equation}\label{eq:I1}
	|I_1|\leq \max_{\overline{\Omega}}|\nabla^m\phi| \ \|\nabla^mW_\ve\|_{L^2(\Omega)} \|W_\ve\|_{L^2(\Omega)}\leq C |\Oe|^{\frac{1}{2}} |\Oe|^{\frac{1}{2}+\eta_{m,q,1}}=C |\Oe|^{1+\eta_{m,q,1}}. 
\end{equation}
\textit{Estimate of $I_2$:} We note that in the sum only appear the derivatives of order equal or less than $m-1$. Therefore
\begin{equation}\label{eq:I2}
	|I_2|\leq c \sum_{n=1}^{m-1}\|\nabla^mW_\ve\|_{L^2(\Omega)} \|\nabla^{n}W_\ve\|_{L^2(\Omega)}\leq C \|W_\ve\|_{H^m(\Omega)} \|W_\ve\|_{H^{m-1}(\Omega)}= C |\Oe|^{1+\eta_{m,q,1}}.
\end{equation}
\textit{Estimate of $I_3$:} In the sum only appear the derivatives of order equal or less than $m-1$, hence
\begin{equation}\label{eq:I3}
	|I_4|\leq C |\nabla^m V| \sum_{n=1}^{m-1} \|\nabla^n W_\ve\|_{L^2(\Omega_\ve)} |\Oe|^{\frac{1}{2}}\leq C \ |\Oe|^{\frac{1}{2}} \|W_\ve\|_{H^{m-1}(\Omega_\ve)}\leq C |\Oe|^{1+\eta_{m,q,1}}.
\end{equation}
\textit{Estimate of $I_4$:}
\begin{equation}\label{eq:I4}
|I_3|\leq |\nabla^m V|\ \max_{\overline{\Omega}_\ve} |\nabla^m\phi|\ \|W_\ve\|_{L^2(\Omega_\ve)}\ |\Oe|^{\frac{1}{2}}\leq C |\Oe|^{1+\eta_{m,q,1}}. 
\end{equation}
Inserting \eqref{eq:I1}, \eqref{eq:I2}, \eqref{eq:I3} and \eqref{eq:I4} in \eqref{eq:gradmVgradmWe}, we get \eqref{eq:gradmVgradmWe_res}, i.e.,
	\begin{equation*}
		\int_{\Oe}\nabla^m W_\ve\cdot \nabla^m V \phi\, dx=\frac{1}{1-\kappa}\int_{\Omega}\gamma_\ve|\nabla^m W_\ve|^2 \phi\, dx + \mathcal{O}(|\Oe|^{1+\eta_{m,q,1}}).
	\end{equation*}
We are now in position to prove the two estimates in the statement of the proposition.\\
\textit{Inequality $\mathbb{M}E\cdot E\geq |E|^2$:} Using \eqref{eq:gradmVgradmWe_res} in \eqref{eq:can_write}, we find
	\begin{equation}\label{eq:lim}
		\phi(y) \mathbb{M}E\cdot E=\frac{1}{1-\kappa}\lim\limits_{\ve\to 0} \frac{1}{|\Oe|}\int_{\Omega}\gamma_\ve|\nabla^m W_\ve|^2 \phi\, dx + \lim\limits_{\ve\to 0} \frac{1}{|\Oe|} \left[\int_{\Oe}|\nabla^m V|^2 \phi\, dx + \mathcal{O}(|\Oe|^{1+\eta_{m,q,1}})\right].
	\end{equation}
In this last equation, we choose nonnegative test functions $\phi\in C^m(\overline{\Omega})$, i.e., $\phi\geq 0$, and recalling that $\gamma_\ve>0$ and $1-\kappa>0$, see Assumptions \ref{ass:funct} and \eqref{gammaeps}, we get 
	\begin{equation*}
		\phi(y)\mathbb{M}E\cdot E\geq \lim\limits_{\ve\to 0} \frac{1}{|\Oe|} \left[\int_{\Oe}|\nabla^m V|^2 \phi\, dx + \mathcal{O}(|\Oe|^{1+\eta_{m,q,1}})\right]. 
	\end{equation*}
By \eqref{eq:E}, we find
	\begin{equation*}
		\phi(y)\mathbb{M}E\cdot E\geq |E|^2 \lim\limits_{\ve\to 0} \frac{1}{|\Oe|}\int_{\Omega_\ve}\phi\, dx + \lim\limits_{\ve\to 0}\mathcal{O}(|\Oe|^{\eta_m,q,1})=|E|^2 \phi(y),
	\end{equation*}
from which it follows that 
	\begin{equation*}
		\mathbb{M}E\cdot E\geq |E|^2.
	\end{equation*}	
\textit{Inequality $\mathbb{M}E\cdot E\leq \frac{1}{\kappa}|E|^2$:} Taking \eqref{eq:gradmVgradmWe_res} and assuming again that $\phi\geq 0$, we find
	\begin{equation}\label{eq: estimate above}
	\begin{aligned}
		\int_{\Omega}\gamma_\ve |\nabla^m W_\ve|^2 \phi\, dx&=\sqrt{\kappa}\int_{\Oe}\frac{1-\kappa}{\sqrt{\kappa}}\sqrt{\phi}\nabla^m W_\ve \cdot \nabla^m V\sqrt{\phi}\, dx +\mathcal{O}(|\Oe|^{1+\eta_{m,q,1}})\\
		&\leq \left(\int_{\Oe} \kappa |\nabla^mW_\ve|^2 \phi\, dx\right)^{\frac{1}{2}} \left(\int_{\Oe}\frac{(1-\kappa)^2}{\kappa}|\nabla^m V|^2 \phi\, dx\right)^{\frac{1}{2}} +\mathcal{O}(|\Oe|^{1+\eta_{m,q,1}})\\
		&\leq \frac{1}{2} \int_{\Oe} \kappa |\nabla^mW_\ve|^2 \phi\, dx + \frac{1}{2}(1-\kappa)^2\int_{\Oe}\frac{1}{\kappa}|\nabla^m V|^2 \phi +\mathcal{O}(|\Oe|^{1+\eta_{m,q,1}}),
	\end{aligned}
	\end{equation}
where in the first inequality we have applied the Cauchy-Schwarz inequality and in the second one the Young's inequality. Then, since all the terms in the first integral on the right-hand side of the last inequality are positive, and $\gamma_\ve\geq \kappa$ in $\Omega$, see \eqref{gammaeps}, we find that
	\begin{equation*}
		\frac{1}{2} \int_{\Oe} \kappa |\nabla^mW_\ve|^2 \phi\, dx= \frac{1}{2} \int_{\Oe} \gamma_\ve |\nabla^mW_\ve|^2 \phi\, dx\leq \frac{1}{2} \int_{\Omega} \gamma_\ve |\nabla^mW_\ve|^2 \phi\, dx ,
	\end{equation*} 
hence using this estimate in \eqref{eq: estimate above}, and then summing up the resulting terms appropriately, we get	
	\begin{equation*}
		\int_{\Omega}\gamma_\ve |\nabla^m W_\ve|^2 \phi\, dx\leq (1-\kappa)^2\int_{\Oe}\frac{1}{\kappa}|\nabla^m V|^2\phi\, dx + \mathcal{O}(|\Oe|^{1+\eta_{m,q,1}}).
	\end{equation*}
Inserting this equation in \eqref{eq:lim}, we get
	\begin{equation*}
		\phi(y)\mathbb{M}E\cdot E \leq  \lim\limits_{\ve\to 0} \frac{1-\kappa}{|\Oe|} \int_{\Oe}\frac{1}{\kappa}|\nabla^m V|^2 \phi\, dx+ \lim\limits_{\ve\to 0} \frac{1}{|\Oe|} \left[\int_{\Oe}|\nabla^m V|^2 \phi\, dx + \mathcal{O}(|\Oe|^{1+\eta_{m,q,1}})\right]
	\end{equation*}
where using again \eqref{eq:E}, it follows
	\begin{equation*}
		\phi(y)\mathbb{M}E\cdot E \leq \frac{1}{\kappa}|E|^2\lim\limits_{\ve\to 0}\frac{1}{|\Oe|}\int_{\Oe}\phi\, dx + \lim\limits_{\ve\to 0}\mathcal{O}(|\Oe|^{\eta_{m,q,1}})= \frac{1}{\kappa}|E|^2 \phi(y),
	\end{equation*}
which implies that 
	\begin{equation*}
		\mathbb{M}E\cdot E\leq \frac{1}{\kappa}|E|^2.
	\end{equation*}
\end{proof}

\begin{remark}
	We want to emphasize that, up to here, our analysis can be generalized in a straighforward way to the case of measurable set $\Omega_\ve$ which tends to zero when $\ve\to 0$. In this case,
		\begin{equation*}
			\int_\Omega \frac{1}{\abs{\Oe}} \chi_\Oe \frac{\partial^m \uep}{\partial  x_{i_1}\cdots \partial x_{i_m}} \phi\dx \to  
			\int_{\Omega}\mathbb{M}\nabla^m u\, \phi\, d\widetilde{\mu},\qquad \textrm{as}\ \ \ve\to 0,
		\end{equation*}	 
where $d\widetilde{\mu}$ is a Borel measure supported in $\Omega$.		
\end{remark}
\subsection*{Spectral decomposition of $\tensor$}
In the following we derive the spectral decomposition of the tensor $\tensor$.
To simplify the notation, we assume without loss of generality that $\Oe := \Oe(0,{ e_1})$, where
${ e_1}=(1,0)$ and ${ e_2}=(0,1)$. The more general setting can be always obtained by rotation of $\Oe$.
From \cite{ComGolLimMou08} we know that the dimension of the space of symmetric tensors of order $m$ is equal to $\textrm{dim}(S^m(\mathbb{R}^2))=m+1$. In the following, we denote with $\sigma_m$ any permutation of the elements of $e_{i_1}\otimes\cdots\otimes e_{i_m}$, where $i_k=1,2$, for $k=1,\cdots,m$ with the clause that the resulting permuted object is not repeated if it coincides with one of the already existing outcome. We define with $E^h$, for $h=1,\cdots,m+1$ the orthonormal canonical basis of $S^m(\mathbb{R}^2)$ i.e.
\begin{equation*}
\begin{aligned}
 E^1&=\underbrace{{e_1} \otimes\cdots\otimes {e_1}}_{m-elements} \\
 E^h&=\frac{1}{\sqrt{\binom{m}{h-1}} }\sum_{\sigma_m}\underbrace{e_1\otimes\cdots\otimes e_1}_{(m-h+1)-elements}\otimes \underbrace{e_2\otimes \cdots \otimes e_2}_{(h-1)-elements},\qquad \textrm{for}\ \ h=2,\cdots,m, \\
 E^{m+1}&=\underbrace{{e_2} \otimes\cdots\otimes {e_2}}_{m-elements}
\end{aligned}
\end{equation*}
\begin{remark}\label{rem:numb_elem_can_basis}
	By the definition of the permutation $\sigma_m$ that we are adopting, it is straightforward to observe that each $E^h$, for $h=1,\cdots,m+1$, is the sum of only $\binom{m}{h-1}$ tensors of order $m$ which have $m-h+1$ elements equal to $e_1$ and $h-1$ elements equal to $e_2$.
	%$\underbrace{e_1\otimes\cdots\otimes e_1}_{(m-h+1)-elements}\otimes \underbrace{e_2\otimes \cdots \otimes e_2}_{(h-1)-elements}$. 
	For this reason, the quantity $1/ \sqrt{\binom{m}{h-1}}$ is only a normalization coefficient, i.e., it is such that $|E^h|^2=1$, for $h=2,\cdots,m$. 
\end{remark}
In order to derive the desired spectral value properties of $\tensor$, we use regularity properties of solutions of higher
order equations with discontinuous coefficients, recently derived by Barton in \cite{Bar16}.

\subsection*{Spectral Values along $E^h,\ h=1,\cdots,m$}
We will start by showing  that the quadratic form associated to $\tensor$, i.e., 
$\tensor E \cdot E$, attains the value $1$ along $E=E^h$, for $h=1,\cdots,m$.
\begin{proposition} \label{ExplicitFormPT}
Under the notational simplification that {$\Oe = \Oe(0,{{e_1}})$}, for all $h=1,\cdots,m$, it holds
\begin{equation*}
 \tensor E^h \cdot E^h =1.
\end{equation*}
\end{proposition}
\begin{proof}
The proof of this proposition is essentially based on the use of the equation \eqref{eq:can_write}, where we choose $\phi=1$, i.e., 
	\begin{equation}\label{eq:spect_val_Eh}
		\mathbb{M}E\cdot E=|E|^2 + \lim\limits_{\ve\to 0}\frac{1}{|\Omega_\ve|}\int_{\Omega_\ve}\nabla^m W_\ve \cdot E\, dx,
	\end{equation}
where we used the fact that $\nabla^m V=E$, see \eqref{eq:E}. 
Next, we use the equation satisfied by $W_\ve$, see \eqref{Weps_problem}, and the regularity estimates proved in \cite{Bar16} to get an estimate of the integral term in the previous formula, when we choose $E=E^h$, for $h=1,\cdots,m$.\\
We first note that $W_\ve\in H^m_0(\Omega)$ and in \eqref{Weps_problem} the source term $((1-\kappa)\chi_{\Oe} \nabla^m V)\in L^2(\Omega)$. Then, for all $p'\in (2,3)$ and $p\in (\frac{3}{2},2)$ such that $\frac{1}{p}+\frac{1}{p'}=1$, by Theorem 24 in \cite{Bar16}, it follows that there exists $C$, independent of $\ve$ and $E$, such that 
	\begin{equation}\label{eq:sup_nabla_mm1}
		\sup\limits_{\Oe}|\nabla^{m-1}W_\ve|\leq C \|\nabla^{m-1}W_\ve\|_{L^p(\Omega)} + \|(1-\kappa)\chi_{\Oe} E \|_{L^{p'}(\Omega)}. 
	\end{equation}
In addition, since $p\in(\frac{3}{2},2)$, we also have that
	\begin{equation*}
		\|\nabla^{m-1}W_\ve\|_{L^p(\Omega)}\leq C \|\nabla^{m-1} W_\ve\|_{L^2(\Omega)}\leq C \| W_\ve\|_{H^{m-1}(\Omega)}\leq C \| W_\ve\|_{H^m(\Omega)}.
	\end{equation*}
Last inequality and \eqref{eq:sup_nabla_mm1} imply that 
	\begin{equation}\label{eq:sup_gradmm1}
		\sup\limits_{\Oe}|\nabla^{m-1}W_\ve|\leq C \|W_\ve\|_{H^m(\Omega)}+ C' |E| |\Oe|^{\frac{1}{p'}}\leq C |\Oe|^{\min(\frac{1}{2},\frac{1}{p'})}\leq C |\Oe|^{\frac{1}{p'}}.
	\end{equation}
\textit{Spectral value along $E^1$.} Choosing $E=E^1$ in equation \eqref{eq:spect_val_Eh}, we get
	\begin{equation}\label{eq:E^1}
		\mathbb{M}E^1\cdot E^1=1+\lim\limits_{\ve\to 0}\frac{1}{|\Oe|}\int_{\Oe}\frac{\partial^m W_\ve}{\partial x_1^m}\, dx,
	\end{equation}
and moreover
	\begin{equation*}
		\int_{\Oe}\frac{\partial^m W_\ve}{\partial x_1^m}\, dx=\int_{\Oe\setminus \Oep}\frac{\partial^m W_\ve}{\partial x_1^m}\, dx + \int_{\Oep}\frac{\partial^m W_\ve}{\partial x_1^m}\, dx=:I^1_1+I^1_2.
	\end{equation*}
We estimate $I^1_1$ and $I^1_2$: 
	\begin{equation}\label{eq:I1E^1}
		|I^1_1|= \Big|\int_{\Oe\setminus \Oep}\frac{\partial^m W_\ve}{\partial x_1^m}\, dx \Big| \leq \Big\|\frac{\partial^m W_\ve}{\partial x^m_1}\Big\|_{L^2(\Oe\setminus\Oep)} |\Oe\setminus\Oep|^{\frac{1}{2}}\leq \|\nabla^m W_\ve\|_{L^2(\Omega)} |\Oe\setminus\Oep|^{\frac{1}{2}}=\mathcal{O}(|\Oe|^{1+\frac{1}{6}}),  
	\end{equation}
where in the last equality we have used the results in \eqref{eq:area_incl}. For the integral $I^1_2$ we have
	\begin{equation*}
		I^1_2=\int_{\Oep} \frac{\partial^m W_\ve}{\partial x_1^m}\, dx= \int_{-\ve^2}^{\ve^2}\int_{-\ve}^{\ve}\frac{\partial^m W_\ve}{\partial x^m_1}\, dx_1\, dx_2=\int_{-\ve^2}^{\ve^2}\left(\frac{\partial^{m-1}W_\ve (\ve,x_2)}{\partial x^{m-1}_1}- \frac{\partial^{m-1}W_\ve (-\ve,x_2)}{\partial x^{m-1}_1}\right)\, dx_2.
	\end{equation*} 
Then, using the result in \eqref{eq:sup_gradmm1}, we get that
	\begin{equation}\label{eq:I2E^1}
	\begin{aligned}
		|I^1_2|=\Bigg|\int_{-\ve^2}^{\ve^2}\left(\frac{\partial^{m-1}W_\ve (\ve,x_2)}{\partial x^{m-1}_1}- \frac{\partial^{m-1}W_\ve (-\ve,x_2)}{\partial x^{m-1}_1}\right)\, dx_2\Bigg| &\leq C \sup\limits_{x_2\in[-\ve^2,\ve^2]} \Bigg| \frac{\partial^{m-1} W_\ve}{\partial x^{m-1}_1}\Bigg| \ \ve^2\\
		& \leq C\|\nabla^{m-1}W_\ve\|_{L^{\infty}(\Oe)}\ \ve^2=\mathcal{O}(|\Oe|^{\frac{1}{p'}+\frac{2}{3}}),
	\end{aligned}
	\end{equation}
where, since for hypothesis $p'\in(2,3)$, we have that $\frac{1}{p'}+\frac{2}{3}=1+\delta$, $\delta>0$.	
Therefore, inserting the results in \eqref{eq:I1E^1} and \eqref{eq:I2E^1} into \eqref{eq:E^1}, we finally find the equation
	\begin{equation*}
		\mathbb{M}E^1\cdot E^1=1.
	\end{equation*}
\textit{Spectral value along $E^h$, for $h=2,\cdots,m$.} 	
Choosing $E=E^h$, for $h=2,\cdots,m$, in equation \eqref{eq:spect_val_Eh}, we get
	\begin{equation}\label{eq:E^h}
	\mathbb{M}E^h\cdot E^h=1 + \lim\limits_{\ve\to 0}\frac{1}{|\Omega_\ve|}\int_{\Omega_\ve}\nabla^m W_\ve \cdot E^h\, dx,
	\end{equation}
where, using the symmetries of $\nabla^m W_\ve$, (which come from the regularity property of the polyharmonic function $W_\ve$ in $\Omega_\ve$, see Equation \eqref{Weps_problem} and Remark \ref{rem:reg_ue_Ve}) and Remark \ref{rem:numb_elem_can_basis}, we get	
	\begin{equation*}
		\int_{\Oe}\nabla^m W_\ve \cdot E^h\, dx=\sqrt{\binom{m}{h-1}}\int_{\Oe}\frac{\partial^m W_\ve}{\partial x_1^{m-h+1}\partial x_2^{h-1}}\, dx_1\ dx_2,
	\end{equation*}
and moreover
	\begin{equation*}
		\int_{\Oe}\frac{\partial^m W_\ve}{\partial x_1^{m-h+1}\partial x_2^{h-1}}\, dx_1\ dx_2=\int_{\Oe\setminus\Oep}\frac{\partial^m W_\ve}{\partial x_1^{m-h+1}\partial x_2^{h-1}}\, dx_1\ dx_2 + \int_{\Oep}\frac{\partial^m W_\ve}{\partial x_1^{m-h+1}\partial x_2^{h-1}}\, dx_1\ dx_2=:I^h_1+I^h_2.
	\end{equation*}
As already done for the terms $I^1_1$ and $I^1_2$, we find that 
	\begin{equation}\label{eq:I1E^h}
	|I^h_1|= \Big|\int_{\Oe\setminus \Oep}\frac{\partial^m W_\ve}{\partial x_1^{m-h+1}\partial x_2^{h-1}}\, dx_1\ dx_2 \Big| \leq \|\nabla^m W_\ve\|_{L^2(\Oe\setminus\Oep)} |\Oe\setminus\Oep|^{\frac{1}{2}}=\mathcal{O}(|\Oe|^{1+\frac{1}{6}}),  
	\end{equation}
where in the last equality we have used the results in \eqref{eq:area_incl}.
For the integral $I^h_2$ we get
	\begin{equation}\label{eq:I2E^h}
	\begin{aligned}
		|I^h_2|\leq \Bigg|\int_{\Oe}\frac{\partial^m W_\ve}{\partial x^{m-h+1}_1\partial x^{h-1}_2}\Bigg| &=\Bigg|\int_{-\ve^2}^{\ve^2}\frac{\partial^{m-1}W_\ve(\ve,x_2)}{\partial x^{m-h}_1\partial x^{h-1}_2} - \frac{\partial^{m-1}W_\ve(-\ve,x_2)}{\partial x^{m-h}_1\partial x^{h-1}_2}\, dx_2\Bigg|\\
		& \leq C \|\nabla^{m-1}W_\ve\|_{L^{\infty}(\Oe)}\ \ve^2=\mathcal{O}(|\Oe|^{1+\delta}),
	\end{aligned}
	\end{equation}
where $\delta>0$. Inserting \eqref{eq:I1E^h} and \eqref{eq:I2E^h} in \eqref{eq:E^h} we get the assertion.	
\end{proof}

\subsection*{Spectral Values along $E^{m+1}$}
We will now show that the quadratic form associated to $\tensor$, $\tensor E\cdot E$, attains value
$\frac{1}{\kappa}$ along $E=E^{m+1}$. To prove this proposition we will need to show two preliminaries Lemmas.
The first one is an adaptation of \cite[Lemma 3]{CapVog06}.
\begin{lemma}
For all $E\in S^m(\mathbb{R}^2)$ we have
\begin{equation} \label{eq:CapVog06}
 (\kappa -1)  \tensor E \cdot E = \frac{\kappa -1}{\kappa} \abs{E}^2 +
         \lim_{\ve \to 0} \left(\frac{1}{{\abs{\Oe}}}
         \min_{\widetilde{W}\in H_0^m(\Omega)} \int_\Omega \gep
         \abs{\nabla^m \widetilde{W} + \frac{\kappa-1}{\kappa} \chi_\Oe E}^2\right).
        \end{equation}
\end{lemma}

\begin{proof}
Since $\gamma_{\epsilon} = \kappa$ in $\Oe$, see \eqref{eq:ve}, it follows by just squaring the quadratic term 
\begin{equation}\label{eq:auxWe}
 \begin{aligned}
  & \int_\Omega \gep  \abs{\nabla^m \Wep + \frac{\kappa-1}{\kappa} \chi_\Oe E}^2\dx \\
 =& \int_\Oe   \frac{(\kappa-1)^2}{\kappa} \abs{E}^2 \dx
     +\int_\Omega \gep  \abs{\nabla^m \Wep}^2 \dx + 2 (\kappa-1) \int_\Oe \nabla^m \Wep \cdot E \dx.
 \end{aligned}
\end{equation}
Then, using the weak formulation \eqref{eq:weakWe} of $W_\ve$, in which we choose $\varphi=W_\ve$ as a test function, we can rewrite the second integral in the right-hand side of the previous formula in the following way 
	\begin{equation}\label{eq:weakWeWe}
		\int_{\Omega}\gamma_\ve |\nabla^m W_\ve|^2\, dx=(1-\kappa)\int_{\Oe}\nabla^m W_\ve\cdot \nabla^m V\, dx= (1-\kappa)\int_{\Oe} \nabla^m W_\ve\cdot E\, dx,
	\end{equation}
where in the last equality we have used \eqref{eq:E}. 	
Inserting \eqref{eq:weakWeWe} in \eqref{eq:auxWe}, we find
	\begin{equation*}
		 \int_\Omega \gep  \abs{\nabla^m \Wep + \frac{\kappa-1}{\kappa} \chi_\Oe E}^2\dx=\abs{\Oe}\frac{(\kappa-1)^2}{\kappa}\abs{E}^2 + (\kappa-1)\int_\Oe \nabla^m \Wep \cdot E \dx,
	\end{equation*}
hence
\begin{equation*}
(\kappa-1) \int_\Oe \nabla^m \Wep\cdot E\, dx =
 -\abs{\Oe}\frac{(\kappa-1)^2}{\kappa}\abs{E}^2 +
  \int_\Omega \gep  \abs{\nabla^m \Wep + \frac{\kappa-1}{\kappa} \chi_\Oe E}^2.
\end{equation*}
Then, inserting this expression into \eqref{eq:spect_val_Eh}, we get
\begin{equation*}
(\kappa-1)\tensor E\cdot E = (\kappa-1)\abs{E}^2
- \frac{(\kappa-1)^2}{\kappa}\abs{E}^2
+ \lim_{\ve \to 0} \frac{1}{\abs{\Oe}}\int_\Omega \gep
        \abs{\nabla^m \Wep + \frac{\kappa-1}{\kappa} \chi_\Oe E}^2 \dx.
\end{equation*}
Finally, noting that $\Wep$ is the minimizer of the functional
$\widetilde{W} \in H_0^m(\Omega) \to \int_\Omega \gep
         \abs{\nabla^m \widetilde{W} + \frac{\kappa-1}{\kappa} \chi_\Oe E}^2$ the assertion follows.
\end{proof}
Next, we have that

\begin{lemma}\label{le:he}
There exists a function $ \owep \in H^m_0(\Omega)$ such that
 \begin{equation*}
   \int_\Omega \abs{\nabla^m  \owep-\chi_\Oe \nabla^mx_2^m}^2 \dx = o(\abs{\Oe}),\qquad \text{ for } \ve \to 0.
 \end{equation*}
\end{lemma}
To prove this lemma, we first need to introduce some parameters and three functions, defined on the real axis, which are involved in the definition of the function $\owep$:
\begin{enumerate}[(i)]
	\item Let  
	\begin{equation}\label{eq:a}
	a\in \left(\frac{m-1}{2m-1},\frac{4m-3}{2(2m-1)}\right)
	\end{equation}
	 and let $\psi_\ve \in C^\infty_0(\R)$ (see \cite{Nes03}) be the function which satisfies for some constant $C_1>0$
	\begin{enumerate}
		\item $\psi_\ve(x_2)=1;\qquad$
		$\forall x_2\in [-\ve^2,\ve^2]$;
		\item \label{it:psi:ii} $\text{supp}(\psi_\ve) \subset (-\ve^2-\ve^{2a},\ve^2+\ve^{2a})$;
		\item $0 \leq \psi_\ve \leq 1$;
		\item \label{it:psi:iv}
		$\abs{\frac{d\psi_\ve}{d x_2}} \leq \frac{C_1}{\ve^{2a}},\ \ \abs{\frac{d^2\psi_\ve}{d x^2_2}} \leq \frac{C_1}{\ve^{4a}},\cdots, \abs{\frac{d^m\psi_\ve}{d x^m_2}} \leq \frac{C_1}{\ve^{2ma}}$.
	\end{enumerate}
	\item Moreover, let  
	\begin{equation}\label{eq:beta}
	\beta\in\left(\frac{1}{2}, \frac{1}{2(2m-1)}+a\right),
	\end{equation}
	 and let $\varphi_\ve\in C^{\infty}_0(\mathbb{R})$ (see \cite{Nes03}) be the function which satisfies for some constant $C_2>0$
	\begin{enumerate}
		\item  \label{it:ve:i}$\varphi_\ve(x_1) = 1,\qquad$  $\forall x_1\in [-\ve,\ve]$;
		\item  \label{it:ve:ii}$ \text{supp}(\varphi_\ve) \subset (-\ve-\ve^{2\beta},\ve+\ve^{2\beta})$;
		\item \label{it:ve:iii}$0 \leq \varphi_{\ve} \leq 1$;
		\item  \label{it:ve:iv} $\abs{\frac{d\varphi_\ve}{d x_1}} \leq \frac{C_2}{\ve^{2\beta}},\ \ \abs{\frac{d^2\varphi_\ve}{d x^2_1}} \leq \frac{C_2}{\ve^{4\beta}},\cdots, \abs{\frac{d^m\varphi_\ve}{d x^m_1}} \leq \frac{C_2}{\ve^{2m\beta}}$.         
	\end{enumerate}
	\begin{remark}
	Parameters $a$ and $\beta$ vary in the triangular orange region in Figure \ref{Fig:lines}.
	\end{remark}
	\item Finally, let $ \overline{v}_\ve: \R \to \R$ be defined as follows:
	\begin{equation} \label{eq:vve}
	\overline{v}_\ve(x_2) :=
	\begin{cases}
	\vspace{0.2cm}
	\sum_{h=1}^{m}\frac{m!(-1)^{h-1} \ve^{2h} x^{m-h}_2}{h! (m-h)!}, & \textrm{if}\, x_2>\ve^2\\
	\vspace{0.2cm}
	x^m_2, & \textrm{if}\, x_2\in[-\ve^2,\ve^2]\\
	\sum_{h=1}^{m}\frac{m!(-1)^{2h-1} \ve^{2h} x^{m-h}_2}{h! (m-h)!}, & \textrm{if}\, x_2<-\ve^2.
	\end{cases}      
	\end{equation}	
\end{enumerate}
\begin{remark}
Function $\overline{v}_\ve$ is an element of $H^m$ on every compact interval of $\R$. To show that, it is sufficient to prove that $\overline{v}_\ve$ and its derivatives $\frac{d^{l}\overline{v}_\ve}{d x_2^l}$, for $l=1,\cdots,m-1$, are continuous polynomial functions across $x_2=\ve^2$ and $x_2=-\ve^2$. With this aim, we first notice that in $\{x_2>\ve^2\}$ 
	\begin{equation}\label{eq: veps_mod}
		\overline{v}_\ve(x_2)=\sum_{h=1}^{m}\frac{m!(-1)^{h-1} \ve^{2h} x^{m-h}_2}{h! (m-h)!}=\sum_{h=1}^{m}(-1)^{h-1}\binom{m}{h}\ve^{2h} x^{m-h}_2,\qquad \text{for}\,\, x_2>\ve^2,
	\end{equation}
and, analogously, the derivatives of $\overline{v}_\ve$ in $\{x_2>\ve^2\}$ satisfy	
	\begin{equation*}
	\begin{aligned}
		\frac{d^{l}\overline{v}_\ve}{d x_2^l}(x_2)&=\sum_{h=1}^{m-l}\frac{m!(-1)^{h-1} (m-h)(m-h-1)\cdots(m-h-l+1)\ve^{2h} x^{m-h-l}_2}{h! (m-h)!}\\
		&=m!\sum_{h=1}^{m-l}\frac{(-1)^{h-1}\ve^{2h}x^{m-h-l}_2}{h!(m-h-l)!},\qquad\qquad\qquad \text{for}\,\, x_2>\ve^2.
	\end{aligned}
	\end{equation*}
Multiplying and dividing for $(m-l)!$ in the last expression, we get
	\begin{equation}\label{eq: der_veps_mod}
	\begin{aligned}
		\frac{d^{l}\overline{v}_\ve}{d x_2^l}(x_2)&=\frac{m!}{(m-l)!}\sum_{h=1}^{m-l}(-1)^{h-1}\frac{(m-l)!}{h!(m-l-h)!}\ve^{2h}x_2^{m-h-l}\\
		&=m(m-1)\cdots(m-l+1)\sum_{h=1}^{m-l}(-1)^{h-1}\binom{m-l}{h}\ve^{2h}x_2^{m-h-l},\qquad \text{for}\,\, 1\leq l\leq m-1.
	\end{aligned}
	\end{equation}
From \eqref{eq: veps_mod} and \eqref{eq: der_veps_mod} it is straightforward to check that both $\overline{v}_\ve$ and its derivates $\frac{d^{l}\overline{v}_\ve}{d x_2^l}$ are continuous functions across $x_2=\ve^2$. In fact, 
	\begin{equation}\label{eq:cont_veps}
		\overline{v}_\ve(\ve^2)=-\ve^{2m}\sum_{h=1}^{m}(-1)^h \binom{m}{h}=-\ve^{2m}\left[-1+\sum_{h=0}^{m}(-1)^h \binom{m}{h}\right]=\ve^{2m}
	\end{equation}
where in the last equality we used the well-known fact that $\sum_{h=0}^{m}(-1)^h \binom{m}{h}=0$. Analogously, for all $1\leq l\leq m-1$, we have
	\begin{equation}\label{eq:cont_der_veps}
	\begin{aligned}
		\frac{d^{l}\overline{v}_\ve}{d x_2^l}(\ve^2)&=m(m-1)\cdots(m-l+1)\ve^{2(m-l)}\sum_{h=1}^{m-l}(-1)^{h-1}\binom{m-l}{h}\\
		&=-m(m-1)\cdots(m-l+1)\ve^{2(m-l)}\sum_{h=1}^{m-l}(-1)^{h}\binom{m-l}{h}\\
		&=-m(m-1)\cdots(m-l+1)\ve^{2(m-l)}\left[-1+\sum_{h=0}^{m-l}(-1)^{h}\binom{m-l}{h}\right]\\
		&=m(m-1)\cdots(m-l+1)\ve^{2(m-l)},
	\end{aligned}
	\end{equation}	
where in the last equality we used again the property $\sum_{h=0}^{m-l}(-1)^{h}\binom{m-l}{h}=0$. Values in \eqref{eq:cont_veps} and \eqref{eq:cont_der_veps} coincide with that ones assumed by $x_2^m$ and its derivatives $\frac{d^{l} x_2^m}{d x_2^l}$, for $l=1,\cdots,m-1$, across $x_2=\ve^2$. The same argument can be applied for $\overline{v}_\ve(x_2)$ in $\{x_2<-\ve^2\}$. 		
\end{remark}

\begin{proof}[of Lemma \ref{le:he}]
 In the following we denote by
 \begin{equation}\label{eq:rve}
  R_\ve=\set{(x_1,x_2):-\ve-\ve^{2\beta} \leq x_1 \leq \ve+\ve^{2\beta},-\ve^2 -\ve^{2a}\leq x_2\leq \ve^2+\ve^{2a}}
 \end{equation}
with $a$ and $\beta$ chosen as in \eqref{eq:a} and \eqref{eq:beta}, respectively. Note that for $\ve$ sufficiently small $R_\ve \subset \Omega$.

 \item We define a test-function $\owep : \R^2 \to \R$ as follows:
 \begin{equation} \label{eq:wep}
 \owep(x_1,x_2)= \overline{v}_\ve(x_2)\psi_\ve(x_2)\varphi_\ve(x_1).
 \end{equation}
 Function $ \owep$ satisfies the following properties:
 \begin{enumerate}[(a)]
  \item $ \owep\in H^m_0(\Omega)$,\label{it:wepi}
  \item  $ \owep(x_1,x_2)=
    	\begin{cases}
    	x_2^m & \textrm{in}\ \Oep\\
    	x_2^m\varphi_\ve(x_1) & \textrm{in}\ \COe
    	\end{cases}
    $.\label{it:wepii}	
    \end{enumerate}
We can now prove the assertion of the Lemma. 
From the definition of $\owep$ and property \ref{it:wepii}, it follows that
\begin{equation} \label{eq:int_0}
 \begin{aligned}
  \int_\Omega \abs{\nabla^m \owep-\chi_\Oe \nabla^mx_2^m}^2 \dx
  &=\int_\Oe \abs{\nabla^m \owep-\nabla^mx_2^m}^2 \dx+ \int_\OOe\abs{\nabla^m \owep}^2 \dx\\
  &=\int_\COe \abs{\nabla^m \owep-\nabla^mx_2^m}^2 \dx+\int_\OOe\abs{\nabla^m \owep}^2 \dx=:J_1+J_2.
 \end{aligned}
\end{equation}
Now, we estimate $J_1$ and $J_2$.\\
\textit{Estimate of $J_1$.} For the integral $J_1$ we first use the property \ref{it:wepii} of $\owep$, hence we find 
\begin{equation} \label{eq:int}
 \int_\COe \abs{\nabla^m \owep-\nabla^mx_2^m}^2 \dx = \int_\COe \abs{\nabla^m(x_2^m(\varphi_\ve-1))}^2 \dx.
\end{equation}
By properties \eqref{it:ve:i}-\eqref{it:ve:iv} of $\varphi_\ve$ and the fact that $|x_2|\leq \ve^2$, we observe that 
	\begin{equation*}
		\abs{\nabla^m(x_2^m(\varphi_\ve-1))}^2\leq C\left(\sum_{n=0}^{m}|x_2|^n \Bigg|\frac{d^{n}\varphi_\ve}{d x_1^n}\Bigg|\right)^2\leq C \sum_{n=0}^{m}\ve^{4n-4n\beta},
	\end{equation*}
hence, using this last inequality in \eqref{eq:int} and the fact that $|\Oe\setminus\Oep|\leq C \ve^4$, we get
	\begin{equation*}
		\int_\COe \abs{\nabla^m(x_2^m(\varphi_\ve-1))}^2 \dx\leq C |\Oe\setminus\Oep| \sum_{n=0}^{m}\ve^{4n-4n\beta} \leq C \sum_{n=0}^{m}\ve^{4+4n-4n\beta}=o(|\Oe|)
	\end{equation*}	
where the last equality derives from the fact that, for the range of $a$ chosen, we have that $\beta<1$ hence $4+4n-4n\beta>4$. Then it is sufficient to recall that $\ve^3 = \mathcal{O}(\abs{\Oe})$, see  \eqref{eq:area_incl}.\\
\textit{Estimate of $J_2$.} Taking into account the definitions of $\Oe$ and $ \Oep$ (see \autoref{fig:L}) it follows from the properties of $\owep$ and the definition of $R_\ve$, see \eqref{eq:rve}, that
\begin{equation*}
 \int_\OOe \abs{\nabla^m \owep}^2 \dx \leq \int_\OOep \abs{\nabla^m \owep}^2 \dx=\int_{R_\ve \backslash  \Oep}\abs{\nabla^m \owep}^2 \dx.
\end{equation*}
We divide $R_\ve$ into eight parts, consisting of
\begin{equation*}
 \begin{aligned}
  R_\ve^1 &:= \set{(x_1,x_2): -\ve\leq x_1\leq \ve, \ve^2\leq x_2\leq \ve^2+\ve^{2a}} \subseteq \Omega \backslash \Oe,\\
  R_\ve^2 &:=\set{(x_1,x_2): \ve\leq x_1\leq \ve+\ve^{2\beta}, \ve^2\leq x_2\leq \ve^2+\ve^{2a}} \subseteq \Omega \backslash \Oe,\\
  R_\ve^3 &:=\set{(x_1,x_2): \ve\leq x_1\leq \ve+\ve^{2\beta}, -\ve^2\leq x_2\leq \ve^{2}} \subseteq \Omega \backslash  \Oep,
 \end{aligned}
\end{equation*}
and its symmetric counterparts (see \autoref{fig:rve}).
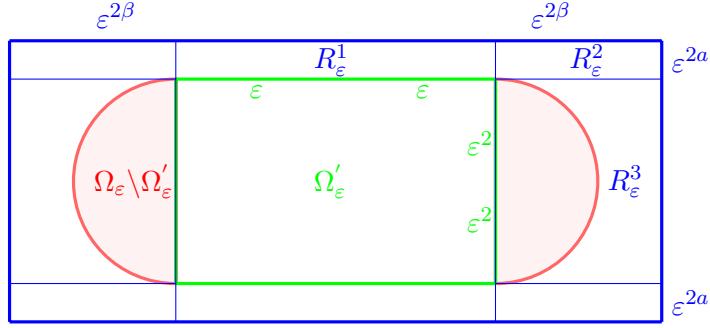
\begin{figure}
\begin{center}
\begin{tikzpicture}
\filldraw[color=red!60, fill=red!5, very thick] (0,0.08\textwidth) arc[radius = 0.08\textwidth, start angle= 90, end angle= 270]--cycle
(0,-0.08\textwidth) -- (0,0.08\textwidth);
\filldraw[color=red!60, fill=red!5, very thick] (0.25\textwidth,-0.08\textwidth)
arc[radius = 0.08\textwidth, start angle= -90, end angle= 90]--cycle (0.25\textwidth,-0.08\textwidth) --(0.25\textwidth,0.08\textwidth);
\filldraw[color=green, fill=green, very thick]
(0,0.08\textwidth) -- (0.25\textwidth,0.08\textwidth) (0.25\textwidth,0.08\textwidth)--(0.25\textwidth,-0.08\textwidth)
(0.25\textwidth,-0.08\textwidth) -- (0,-0.08\textwidth) (0,-0.08\textwidth) -- (0,0.08\textwidth);
\filldraw[color=blue, fill=blue, very thick]
(-0.13\textwidth,0.11\textwidth) -- (0.38\textwidth,0.11\textwidth)
(0.38\textwidth,0.11\textwidth)--(0.38\textwidth,-0.11\textwidth)
(0.38\textwidth,-0.11\textwidth) -- (-0.13\textwidth,-0.11\textwidth)
(-0.13\textwidth,-0.11\textwidth) -- (-0.13\textwidth,0.11\textwidth);
\filldraw[color=blue, fill=blue]
(-0.0\textwidth,-0.11\textwidth) -- (-0.0\textwidth,0.11\textwidth)
(-0.13\textwidth,0.08\textwidth)--(0.0\textwidth,0.08\textwidth)
(-0.13\textwidth,-0.08\textwidth)--(0.0\textwidth,-0.08\textwidth)
(0.25\textwidth,-0.11\textwidth) -- (0.25\textwidth,0.11\textwidth)
(0.25\textwidth,0.08\textwidth)--(0.38\textwidth,0.08\textwidth)
(0.25\textwidth,-0.08\textwidth)--(0.38\textwidth,-0.08\textwidth);
\draw (0.33\textwidth,0.0\textwidth) node[anchor=west] [color=blue]{$R_\ve^3$};
\draw (0.3\textwidth,0.095\textwidth) node[anchor=west] [color=blue]{$R_\ve^2$};
\draw (0.1\textwidth,0.095\textwidth) node[anchor=west] [color=blue]{$R_\ve^1$};
\draw (0.05\textwidth,0.07\textwidth) node[anchor=west] [color=green]{$\ve$};
\draw (0.18\textwidth,0.07\textwidth) node[anchor=west] [color=green]{$\ve$};
\draw (-0.07\textwidth,0.13\textwidth) node[anchor=west] [color=blue]{$\ve^{2\beta}$};
\draw (0.27\textwidth,0.13\textwidth) node[anchor=west] [color=blue]{$\ve^{2\beta}$};
\draw (0.38\textwidth,0.095\textwidth) node[anchor=west] [color=blue]{$\ve^{2a}$};
\draw (0.38\textwidth,-0.095\textwidth) node[anchor=west] [color=blue]{$\ve^{2a}$};
\draw (0.1\textwidth,0.0\textwidth) node[anchor=west] [color=green]{$ \Oep$};
\draw (-0.072\textwidth,0.0\textwidth) node[anchor=west] [color=red]{$\Oe \backslash  \Oep$};
\draw (0.22\textwidth,0.03\textwidth) node[anchor=west] [color=green]{$\ve^2$};
\draw (0.22\textwidth,-0.03\textwidth) node[anchor=west] [color=green]{$\ve^2$};
\end{tikzpicture}
\end{center}
\caption{\label{fig:rve} $R_\ve$ is divided into $9$ parts. The central object is $ \Oep$.}
\end{figure}
Then we estimate the integrals over these subdomains separately; however
because of symmetry of $\owep$ it suffices to only estimate the integrals over $R_\ve^i$, $i=1,2,3$.
\begin{enumerate}
\item Estimate of $\int_{R_\ve^1}\abs{\nabla^m \owep}^2 \dx$.
Using the fact that $\varphi_\ve=1$ in $R_\ve^1 \subseteq \Omega \backslash \Oe$ and the definition of $\overline{v}_\ve$ when $x_2>\ve^2$, i.e. $\overline{v}_\ve=\sum_{h=1}^{m}\frac{m!(-1)^{h-1} \ve^{2h} x^{m-h}_2}{h! (m-h)!}$, it follows that every term in $\nabla^m \owep$ containing at least one derivative with respect to $x_1$ is zero. Therefore, the only non zero term is given by
	\begin{equation}\label{eq:owep}
		\frac{\partial^m \owep}{\partial x_2^m}=\sum_{n=1}^{m-1}\left(\frac{d^n\overline{v}_\ve}{dx^n_2}\ \frac{d^{m-n} {\psi}_\ve}{dx^{m-n}_2}\right)+ \overline{v}_\ve \frac{d^{m} {\psi}_\ve}{dx^{m}_2}
	\end{equation}
where we have used the fact that $\frac{d^m\overline{v}_\ve}{dx^m_2}\psi_\ve=0$, since the polynomial $\overline{v}_\ve$ has degree $m-1$. 
Then, from \eqref{eq:owep}, we get
	\begin{equation}\label{eq:est_owep_R1}
		\Bigg|\frac{\partial^m \owep}{\partial x_2^m}\Bigg|\leq C \sum_{n=1}^{m-1} \Bigg|\frac{d^n\overline{v}_\ve}{dx^n_2}\Bigg|\ \Bigg|\frac{d^{m-n} {\psi}_\ve}{dx^{m-n}_2}\Bigg|+ |\overline{v}_\ve| \Bigg|\frac{d^{m} {\psi}_\ve}{dx^{m}_2}\Bigg|.
	\end{equation}
We first observe that, for $n=0,\cdots,m-1$, we get
	\begin{equation}\label{eq:der_v_ve}
	\Bigg|\frac{d^n\overline{v}_\ve}{d x^n_2}\Bigg|\leq C\sum_{h=1}^{m-n}\ve^{2h}\ |x_2|^{m-h-n}.
	\end{equation}
In addition, from the fact that $\ve^2<x_2<\ve^2+\ve^{2a}$
	\begin{equation*}
	|x_2|^{m-h-n}<(\ve^2+\ve^{2a})^{m-h-n}=\sum_{q=0}^{m-h-n}\ve^{2q}\ve^{2a(m-h-n-q)}.
	\end{equation*}
In the last sum, since $a<1$ for hypothesis, see \eqref{eq:a}, we take the power with minimum exponent which corresponds to the term with $q=0$, i.e.,  $\ve^{2a(m-h-n)}$, for $n=0,\cdots,m-1$ and $h=1,\cdots,m-n$, since all the other terms contain, at least, powers of $\ve^2$, hence 
	\begin{equation}\label{eq:estim1}
	|x_2|^{m-h-n}<(\ve^2+\ve^{2a})^{m-h-n}\leq C\ve^{2a(m-h-n)}.
	\end{equation}
Therefore, inserting \eqref{eq:estim1} in \eqref{eq:der_v_ve}, we find
	\begin{equation*}
	\Bigg|\frac{d^n\overline{v}_\ve}{d x^n_2}\Bigg|\leq C\sum_{h=1}^{m-n}\ve^{2h}\ve^{2a(m-h-n)},
	\end{equation*}
where, again, the maximum term corresponds to that one with minimum exponent, i.e., the index $h=1$, hence 
	\begin{equation}\label{eq:estimate_deriv}
		\Bigg|\frac{d^n\overline{v}_\ve}{dx^n_2}\Bigg|\leq C \ve^{2+2(m-1-n)a},\qquad \textrm{where}\,\, 0\leq n \leq m-1.
	\end{equation}
Using this last estimate, \eqref{eq:estimate_deriv}, in \eqref{eq:est_owep_R1} together with \eqref{it:psi:iv}, we find that 
	\begin{equation}\label{eq:est_R1_eps}
		\begin{aligned}
		\Bigg|\frac{d^n\overline{v}_\ve}{dx^n_2}\Bigg|\ \Bigg|\frac{d^{m-n} {\psi}_\ve}{dx^{m-n}_2}\Bigg|&\leq C \ve^{2-2a}\\
		|\overline{v}_\ve| \Bigg|\frac{d^{m} {\psi}_\ve}{dx^{m}_2}\Bigg|&\leq C \ve^{2-2a}.
		\end{aligned}
	\end{equation}
Therefore, from \eqref{eq:est_R1_eps} and \eqref{eq:est_owep_R1}, we get
	\begin{equation}\label{eq:R1}
		\int_{R_\ve^1}\abs{\nabla^m \owep}^2 \dx\leq C |R^1_\ve|\ \ve^{4-4a}= C \ve^{5-2a},
	\end{equation}
where we have used the fact that $|R^1_\ve|=2 \ve^{1+2a}$. Observe that, since $a<1$ for hypothesis, we have that $5-2a>3$, hence \eqref{eq:R1} gives
	\begin{equation}
	\int_{R_\ve^1}\abs{\nabla^m \owep}^2 \dx=o(|\Oe|).
	\end{equation}
\item Estimate of $\int_{R^2_\ve}\abs{\nabla^m \owep}^2 \dx.$ 
In this case $\owep(x_1,x_2)=\overline{v}_\ve(x_2) \psi_\ve(x_2) \varphi_\ve(x_1)$, where $\overline{v}_\ve$ is the polynomial of degree $m-1$ in $x_2>\ve^2$, i.e. $\overline{v}_\ve=\sum_{h=1}^{m}\frac{m!(-1)^{h-1} \ve^{2h} x^{m-h}_2}{h! (m-h)!}$. Then
	\begin{equation}
		|\nabla^m \owep|\leq C\Bigg| \sum_{n=0}^{m}\frac{d^{n}(\overline{v}_\ve\psi_\ve)}{d x^n_2}\ \ \frac{d^{m-n}\varphi_\ve}{d x_1^{m-n}}\Bigg|.
	\end{equation}	
In the previous equation, recalling that $\frac{d^m \overline{v}_\ve}{d x^m_2}=0$, we split the term related to the maximum order of derivative, i.e. $ \frac{d^{m}(\overline{v}_\ve\psi_\ve)}{d x^m_2}$, from the terms with derivative of lower orders, hence we have that	
	\begin{equation*}
	\begin{aligned}
		|\nabla^m \owep|&\leq C\Bigg| \sum_{n=0}^{m}\frac{d^{n}(\overline{v}_\ve\psi_\ve)}{d x^n_2}\ \ \frac{d^{m-n}\varphi_\ve}{d x_1^{m-n}}\Bigg|\\
		&=\Bigg| \varphi_\ve \sum_{r=0}^{m-1}\binom{m}{r} \frac{d^r \overline{v}_\ve}{d x^r_2} \frac{d^{m-r} \psi_\ve}{d x^{m-r}_2}+                            \sum_{n=0}^{m-1}\sum_{r=0}^{n}\binom{n}{r}\frac{d^r\overline{v}_\ve}{d x^r_2}\ \frac{d^{n-r}\psi_\ve}{d x^{n-r}_2}\ \frac{d^{m-n}\varphi_\ve}{d x^{m-n}_1}\Bigg|\\
		&\leq C\left[ |\varphi_\ve|\sum_{r=0}^{m-1}\Bigg|\frac{d^r \overline{v}_\ve}{d x^r_2}\Bigg|\ \Bigg|\frac{d^{m-r} \psi_\ve}{d x^{m-r}_2}\Bigg| + \sum_{n=0}^{m-1}\sum_{r=0}^{n}\Bigg|\frac{d^r\overline{v}_\ve}{d x^r_2}\Bigg|\ \Bigg|\frac{d^{n-r}\psi_\ve}{d x^{n-r}_2}\Bigg|\ \Bigg|\frac{d^{m-n}\varphi_\ve}{d x^{m-n}_1}\Bigg|\right]\\
		&\leq C\left[\sum_{r=0}^{m-1}\Bigg|\frac{d^r \overline{v}_\ve}{d x^r_2}\Bigg|\ \Bigg|\frac{d^{m-r} \psi_\ve}{d x^{m-r}_2}\Bigg| + \sum_{n=0}^{m-1}\sum_{r=0}^{n}\Bigg|\frac{d^r\overline{v}_\ve}{d x^r_2}\Bigg|\ \Bigg|\frac{d^{n-r}\psi_\ve}{d x^{n-r}_2}\Bigg|\ \Bigg|\frac{d^{m-n}\varphi_\ve}{d x^{m-n}_1}\Bigg|\right]=:S_1+S_2.\\
%		&\leq C\sum_{n=0}^{m}\left[ \sum_{r=0}^{n}\Bigg|\frac{d^r\overline{v}_\ve}{d x^r_2}\Bigg|\ \Bigg|\frac{d^{n-r}\psi_\ve}{d x^{n-r}_2}\Bigg|\ \Bigg|\frac{d^{m-n}\varphi_\ve}{d x^{m-n}_1}\Bigg|\right].
	\end{aligned}
	\end{equation*}
%To study $S_1$ and $S_2$, we use the explicit expression of $\overline{v}_\ve$ in $R^2_\ve$, see \eqref{eq:vve}, and the fact that in this rectangle we have $\ve^2<x_2<\ve^2+\ve^{2a}$. From \eqref{eq:estimate_deriv}, we know that 
%	\begin{equation}\label{eq:est_der_m}
%		\Bigg|\frac{d^r\overline{v}_\ve}{d x^r_2}\Bigg|\leq C \sum_{h=1}^{m-r}\ve^{2h} \ve^{2a(m-h-r)}= \mathcal{O}(\ve^{2+2a(m-1-r)}),\qquad \textrm{for all}\ \ r=0,\cdots,m-1,
%	\end{equation}	 
%hence, using 
Using \eqref{it:psi:iv} and \eqref{eq:estimate_deriv}, we find that
	\begin{equation*}
		S_1\leq C \ve^{2-2a}.
	\end{equation*}
Analogously, by means of \eqref{it:psi:iv}, \eqref{it:ve:iv} and \eqref{eq:estimate_deriv} we get that
	\begin{equation*}
		S_2\leq C \sum_{n=0}^{m-1}\ve^{2+2a(m-1)-2an-2\beta(m-n)}. 
	\end{equation*}
Then, using the fact that $|R^2_\ve|=\ve^{2a+2\beta}$, we find
	\begin{equation}\label{eq:estimate_R2_ve}
	\begin{aligned}
		\int_{R^2_\ve}\abs{\nabla^m \owep}^2 \dx &\leq C\Big(\ve^{4-2a+2\beta}+ \sum_{n=0}^{m-1}\ve^{4+4a(m-1)-4an-4\beta(m-n)+2a+2\beta}\Big)\\
		&=C\Big(\ve^{4-2a+2\beta}+ \sum_{n=0}^{m-1}\ve^{4+2(a-\beta)(2m-2n-1)}\Big). 
	\end{aligned}
	\end{equation} 
Therefore, from the hypothesis made for $a$ and $\beta$, see \eqref{eq:a} and \eqref{eq:beta}, we find that all the exponents in the previous formula are greater than $3$, i.e., $4-2a+2\beta>3$ and $4+2(a-\beta)(2m-2n-1)>3$. Indeed, $4-2a+2\beta>3$ is equivalent to $\beta>-\frac{1}{2}+a$, hence we immediately observe that, in Figure \ref{Fig:lines}, the orange region, where $a$ and $\beta$ vary, satisfies $\beta>-\frac{1}{2}+a$. On the other hand, condition $4+2(a-\beta)(2m-2n-1)>3$ is equivalent to $\beta\leq \frac{1}{2(2m-2n-1)}+a$, for all $n=0,\cdots,m-1$. We observe that the function $h(n):=\frac{1}{2(2m-2n-1)}+a$, is an increasing function with respect to $n$, hence the minimum value is $h(0)=\frac{1}{4m-2}+a$, see the red line in Figure \ref{Fig:lines}, which corresponds to the upper bound for $\beta$ in \eqref{eq:beta}. Even in this case, the orange region satisfies the required condition $\beta\leq \frac{1}{2(2m-2n-1)}+a$, for all $n=0,\cdots,m-1$.  
\begin{figure}[h!]
	\centering
	\includegraphics[scale=0.5]{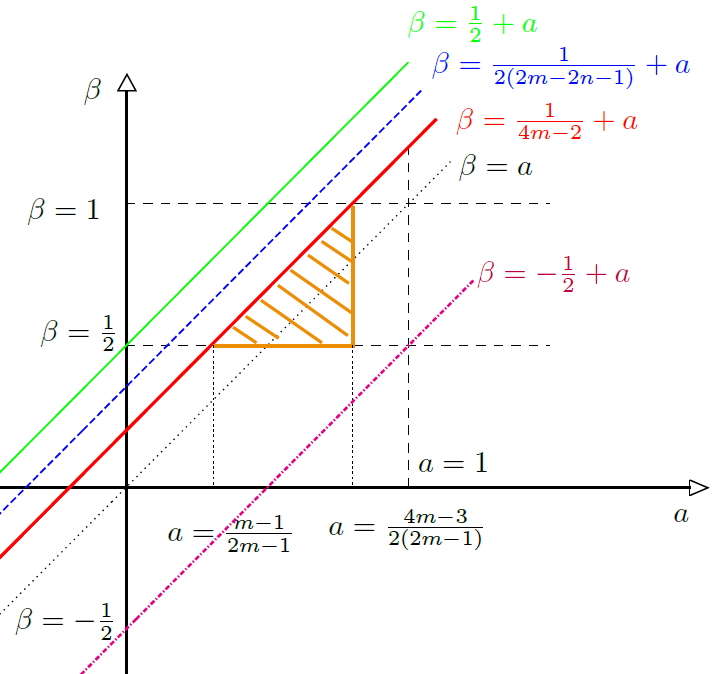}
	\caption{Parameters $a$ and $\beta$, see \eqref{eq:a} and \eqref{eq:beta}, vary in the orange region. Lines to show that, in equation \eqref{eq:estimate_R2_ve}, $4-2a-2\beta>3$ (related to the purple line) and $4+2(a-\beta)(2m-2n-1)>3$ (related to the blue line), for $n=0,\cdots,m-1$. The green line is the case $n=m-1$, the red one is the case $n=0$.}
	\label{Fig:lines}
\end{figure}
%	\begin{figure}[h]
%	\centering
%	\def\svgwidth{800pt}
%	\input{lines.pdf_tex}
%	\end{figure}
Therefore, for the choices of $a$ and $\beta$ in \eqref{eq:a} and $\eqref{eq:beta}$, we have that
\begin{equation*}
 \int_{R^2_\ve}\abs{\nabla^m \owep}^2 \dx=o(\abs{\Oe})
\end{equation*}
\item Estimate of $\int_{R^3_\ve}\abs{\nabla^m \owep}^2 \dx.$ 
In the set $R^3_\ve$, function $\psi_\ve(x_2)=1$ and $\overline{v}_\ve=x^m_2$, hence $\owep=x^m_2\varphi_\ve(x_1)$, which implies that
	\begin{equation*}
		|\nabla^m \owep|\leq C\sum_{n=0}^{m}|x^n_2| \Bigg|\frac{d^n \varphi_\ve}{d x^n_1}\Bigg|\leq C \sum_{n=0}^{m}\ve^{2n-2n\beta}.
	\end{equation*}
We notice that, since for hypothesis $\beta<1$, $2n(1-\beta)\geq 0$, for all $n=0,\cdots,m$, hence, using the fact that $|R^3_\ve|=2\ve^{2+2\beta}$, we get
	\begin{equation*}
	 \int_{R^3_\ve}\abs{\nabla^m \owep}^2 \dx\leq C |R^3_\ve| \sum_{n=0}^{m}\ve^{4n(1-\beta)}\leq C\ve^{2+2\beta},
	\end{equation*} 
where in the last inequality we have used the fact that $4n(1-\beta)+2+2\beta\geq 2+2\beta$, for all $n=0,\cdots,m$. For the choice made for $\beta$, i.e., $\beta>\frac{1}{2}$, we have that $2+2\beta>3$, which means that
	\begin{equation*}
	\int_{R^3_\ve}\abs{\nabla^m \owep}^2 \dx=o(|\Oe|).
	\end{equation*}
\end{enumerate}
\end{proof}

In the following we show that the maximal eigenvalue of $\tensor$ (as defined in \eqref{partialVeps})
is $\frac{1}{\kappa}$.
\begin{proposition} \label{ExplicitFormPT2}
Under the notational simplification that {$\Oe = \Oe(0,{{e_1}})$}
it follows that
\begin{equation} \label{eq:largest_ev}
\tensor E^{m+1}\cdot E^{m+1}=\frac{1}{\kappa}.
\end{equation}
\end{proposition}
\begin{proof}
We first observe that from the fact $\widehat{W} \in H_0^m(\Omega)$ it follows $\widetilde{W}:=- \frac{\kappa-1}{\kappa}\widehat{W} \in H_0^m(\Omega)$.
Therefore, from \eqref{eq:CapVog06} it follows that
\begin{equation*}
 \begin{aligned}
  \min_{\widetilde{W} \in H_0^m(\Omega)}\int_\Omega \gep  \abs{\nabla^m \widetilde{W} + \frac{\kappa-1}{\kappa} \chi_\Oe E}^2 \dx
  = & \min_{\widehat{W} \in H_0^m(\Omega)}\int_\Omega \gep  \abs{\frac{1-\kappa}{\kappa}\nabla^m \widehat{W} -
                                                        \frac{1-\kappa}{\kappa} \chi_\Oe E}^2 \dx \\
  =& \frac{(1-\kappa)^2}{\kappa^2}\min_{\widehat{W} \in H_0^m(\Omega)}\int_\Omega  \gep
         \abs{\nabla^m \widehat{W} -  \chi_\Oe E}^2 \dx.
 \end{aligned}
\end{equation*}
Now, in the previous equation, we choose $E=m! E^{m+1}$, hence
\begin{equation}\label{ineq}
 \begin{aligned}
  \min_{\widetilde{W} \in H_0^m(\Omega)} \int_\Omega \gep  \abs{\nabla^m \widetilde{W} + \frac{\kappa-1}{\kappa} \chi_\Oe m!E^{m+1}}^2 \dx
  &\leq
  \frac{(1-\kappa)^2}{\kappa^2} \int_\Omega  \gep  \abs{\nabla^m \owep -  \chi_\Oe m!E^{m+1}}^2 \dx\\
  & = \frac{(1-\kappa)^2}{\kappa^2} \int_\Omega  \gep  \abs{\nabla^m \owep -  \chi_\Oe \nabla^mx_2^m}^2 \dx.
 \end{aligned}
\end{equation}
Applying \autoref{le:he}, we get
\begin{equation}\label{testfunction}
\min_{\widetilde{W} \in H_0^m(\Omega)} \int_\Omega \gep  \abs{\nabla^m \widetilde{W}+ \frac{\kappa-1}{\kappa} \chi_\Oe m!E^{m+1}}^2 \dx=o( \abs{\Oe})
\end{equation}
as $\varepsilon\rightarrow 0$. Finally,  using \eqref{testfunction} into \eqref{eq:CapVog06} where we choose  $E=m!E^{m+1}$, we get
\begin{equation}
(m!)^2(\kappa -1)  \tensor E^{m+1} \cdot E^{m+1} = (m!)^2\frac{\kappa -1}{\kappa} +
         \lim_{\ve \to 0} \left(\frac{1}{{\abs{\Oe}}}
         \min_{\widetilde{W}\in H_0^m(\Omega)} \int_\Omega \gep
         \abs{\nabla^m \widetilde{W} + \frac{\kappa-1}{\kappa} \chi_\Oe m!E^{m+1}}^2\right) \dx,
        \end{equation}
which gives the assertion.
\end{proof}
\subsection*{Main Result on Spectral Decomposition of $\tensor$}
From the results of the previous section, we are now ready to prove the following spectral decomposition:
\begin{theorem} \label{th:main}
Under the geometrical simplification that $\Oe=\Oe(0,e_1)$, the tensor $\tensor$ has the following spectral decomposition
\begin{equation} \label{eq:main}
 \tensor=\sum_{n=1}^{m}E^n\otimes E^n+ \frac{1}{\kappa} E^{m+1} \otimes E^{m+1}.
\end{equation}
\end{theorem}
\begin{proof}
By Proposition \ref{PTproperty} and Proposition \ref{BoundsPT}, the tensor $\tensor$  is  symmetric and positive definite. Hence, its eigenvalues are positive and real and by
\eqref{BoundsPT} they lie between $1$ and $\frac{1}{\kappa}$. Furthermore by Proposition \ref{ExplicitFormPT}, $1$ is
eigenvalue with multiplicity $m$ and corresponding eigenvectors $E^1,\cdots,E^m$. By Proposition \ref{ExplicitFormPT2},
$\frac{1}{\kappa}$ is an eigenvalue with corresponding eigenvector $E^{m+1}$.
Hence, for all $E\in S^m(\mathbb{R}^2)$ which, using the basis, we can represent as
 \begin{equation*}
  E=\sum_{n=1}^{m+1}(E\cdot E^n)E^n,
 \end{equation*}
we find, by applying the tensor $\tensor$, that
 \begin{equation}
 \begin{aligned}
 \label{eq:rechen}
 \tensor E = \sum_{n=1}^{m+1}(E\cdot E^n) \tensor E^n
  &= \sum_{n=1}^{m}(E\cdot E^n)E^n + \frac{1}{\kappa}(E\cdot E^{m+1}) E^{m+1} \\
  &= \sum_{n=1}^{m}(E^n\otimes E^n)E + \frac{1}{\kappa}(E^{m+1}\otimes E^{m+1}) E,
\end{aligned}  
\end{equation}
 which implies \eqref{eq:main}.
\end{proof}
\begin{remark} \label{re:main}
 In the general setting of $\Oe(y,\tau)$, \autoref{eq:main} reads as follows
 \begin{equation} \label{eq:main_general}
  \tensor = \tensor(\tau) = \sum_{n=1}^{m}E^n \otimes E^n + \frac{1}{\kappa} E^{m+1}\otimes E^{m+1},
 \end{equation}
 where
\begin{equation}\label{eq:tensor_II}
\begin{aligned}
E^1&=\underbrace{{\tau} \otimes\cdots\otimes {\tau}}_{m-elements} \\
E^h&=\frac{1}{\sqrt{\binom{m}{h-1}} }\sum_{\sigma_m}\underbrace{\tau\otimes\cdots\otimes \tau}_{(m-h+1)-elements}\otimes \underbrace{\tau^\perp\otimes \cdots \otimes \tau^\perp}_{(h-1)-elements},\qquad \textrm{for}\ \ h=2,\cdots,m, \\
E^{m+1}&=\underbrace{{\tau^\perp} \otimes\cdots\otimes {\tau^\perp}}_{m-elements}
\end{aligned}
\end{equation}
where $\tau^\perp$ is the unit normal vector to the line segment $\sigma_\ve(y,\tau)$.
\end{remark}

\begin{remark}
Note that the proof to derive the spectral decomposition of the tensor of order $2m$ is more involved than the one in \cite{CapVog06} since we have to deal with higher order differential equations with discontinuous coefficients.
\end{remark}

%%%%%%%%%%%%%%%%%%%%%%%%%%%%%%%%%%%%%%%%

\section{Topological gradient} \label{sec:top_gradients}
We are now ready to derive the topological gradient of the functional $\functional{\cdot}{\cdot}$ as defined in \eqref{Jeps}.
\begin{theorem} \label{th:top_gradient}
Let $u, \uep$ be the solutions of \eqref{u_strong} and \eqref{ue_strong} respectively then,
for $\ve \to 0$, we have
\begin{equation}\label{eq:top_gradient}
\boxed{
\functional{\uep}{\vep}=\functional{u}{v} + 2\ve^3\alpha(\kappa-1)\mathbb{M}\nabla^mu(y)\cdot\nabla^mu(y)+o(\ve^3).}
\end{equation}
\end{theorem}
\begin{proof}
Recalling the definition of the functional $\functional{\cdot}{\cdot}$, see \eqref{Jeps}, we first prove that
\begin{equation}\label{loc}
\functional{\uep}{\vep}-\functional{u}{v}=\frac{\alpha(\kappa-1)}{2}\int_{\Oe} \nabla^m\uep\cdot\nabla^m u \dx.
\end{equation}
It follows from the first order optimality condition of $\functional{\cdot}{\cdot}$ with respect to the first component,
for fixed $\ve$ and $v$, that
\begin{equation}\label{eq51}
 \int_\Omega (\uep-f)\phi \dx +\alpha \int_\Omega \vep \nabla^m \uep \cdot \nabla^m\phi \dx =0,\quad \text{ for all }\phi\in H^m_0(\Omega)
\end{equation}
and
\begin{equation}\label{eq52}
 \int_\Omega (u-f)\phi \dx +\alpha \int_\Omega v\nabla^mu\cdot\nabla^m\phi \dx =0,\quad \text{ for all }\phi\in H^m_0(\Omega).
\end{equation}
Choosing $\phi=u$ in \eqref{eq51}  and $\phi=\uep$  in \eqref{eq52} and then subtracting \eqref{eq52} from \eqref{eq51}, we get
\begin{equation}\label{eq:asymp_formula1}
 \int_\Omega (\uep-u)f \dx =\alpha(1-\kappa)\int_{\Oe}\nabla^m\uep\cdot\nabla^mu \dx.
\end{equation}
On the other hand, inserting $\phi=\uep$ into \eqref{eq51} and $\phi=u$ into \eqref{eq52}, we obtain, respectively
\begin{equation}\label{eq53}
 \int_\Omega (\uep-f)\uep \dx +\alpha \int_\Omega \vep\abs{\nabla^m\uep}^2 \dx = 0
\end{equation}
and
\begin{equation}\label{eq54}
 \int_\Omega (u-f)u \dx +\alpha \int_\Omega v \abs{\nabla^mu}^2 \dx = 0.
\end{equation}
Now, from \eqref{Jeps}, we find
\begin{equation*}
\functional{\uep}{\vep}-\functional{u}{v}=\frac{1}{2} \int_\Omega (\uep-f)^2 \dx +
\frac{\alpha}{2} \int_\Omega \vep\abs{\nabla^m\uep}^2 \dx -\frac{1}{2} \int_\Omega (u-f)^2 \dx
-\frac{\alpha}{2} \int_\Omega v \abs{\nabla^mu}^2 \dx,
\end{equation*}
hence, by \eqref{eq53} and \eqref{eq54} we have that
\begin{equation*}
\functional{\uep}{\vep}-\functional{u}{v}=-\frac{1}{2} \int_\Omega (\uep-u)f \dx,
\end{equation*}
and using \eqref{eq:asymp_formula1}, we get \eqref{loc}.
Now, we estimate the right-hand side of the equation \eqref{loc}, first observing that
\begin{equation}\label{eq:55}
\int_{\Oe}\nabla^m\uep\cdot\nabla^mu \dx =\int_{\Oep}\nabla^m\uep\cdot\nabla^mu \dx +\int_{\Oe\backslash \Oep}\nabla^m\uep\cdot\nabla^mu \dx.
\end{equation}
In the second integral in the righ-hand side of the previous formula, we first add and subtract $\nabla^m u$, and then we apply the Schwarz's inequality, the regularity estimates \eqref{ineq1} and the energy estimates \eqref{EnEst_Hm}, that is
\begin{equation}\label{eq:56}
 \begin{aligned}
  \Bigg|\int_{\Oe \backslash \Oep} \nabla^m\uep \cdot \nabla^mu \dx\Bigg|
  &\leq
    \norm{u}_{L^2(\Omega)}\norm{\uep-u}_{H^m(\Omega)}+\norm{u}^2_{L^2(\Omega)} \\
  &\leq C(\norm{\uep-u}_{H^m(\Omega)} \abs{\Oe \backslash \Oep}^{1/2} + \abs{\Oe\backslash \Oep} \\
  &= o(\ve^3).
 \end{aligned}
\end{equation}
Therefore, inserting \eqref{eq:56} in \eqref{eq:55} and then the resulting equation in \eqref{loc}, it follows
\begin{equation*}
\functional{\uep}{\vep}-\functional{u}{v}=\frac{\alpha(\kappa-1)}{2}\int_{\Oep}\nabla^m\uep\cdot\nabla^mu \dx +o(\ve^3).
\end{equation*}
Next, choosing a bounded set $L_1$ such that $\Omega_\ve\subset L_1\subset L_0$, we use the result in Remark \ref{rem:ueps_weakstar}, hence
\begin{equation*}
\frac{1}{\abs{\Oep}} \int_{\Oep} \nabla^m\uep\cdot\nabla^mu  \dx \to \mathbb{M}\nabla^m u(y)\cdot \nabla^m u(y) \text{ as } \ve \to 0.
\end{equation*}
Recalling that $|\Oep|=4\ve^3$, see \eqref{eq:area_incl}, we finally derive
\begin{equation*}
 \functional{\uep}{\vep}-\functional{u}{v})=2\ve^3\alpha(\kappa-1)\mathbb{M}\nabla^m u(y)\cdot \nabla^m u(y)+o(\ve^3),
\end{equation*}
which concludes the proof.
\end{proof}

%%%%%%%%%%%%%%%%%%%%%%%%%%%%%%%%%%%%%%%%%%%%%%%%%%
%%%%%%%%%%%%%%%%%%%%%%%%%%%%%%%%%%%%%%%%%%%%%%%%%%
\section{Numerical Simulations}\label{sec:numerics}
%%%%%%%%%%%%%%%%%%%%%%%%%%%%%%%%%%%%%%%%%%%%%%%%%%
%%%%%%%%%%%%%%%%%%%%%%%%%%%%%%%%%%%%%%%%%%%%%%%%%%
We consider the problem of Quantitative Photoacoustic Tomography(qPAT) with piecewise constant parameters
$\mu$ and $D$ (absorption and diffusion coefficients, respectively), and a constant Gr\"uneisen parameter $\Gamma$ (see \eqref{eq:def_E} and \eqref{qpatproblem})
as outlined in the \autoref{sec:intro}. Parameters $\mu$ and $D$ can be detected from the set of
discontinuities of derivatives up to the $2$-order of the qPAT measurement data $\mathcal{H}$, which
is proportional to $\mathcal{E}$ under the assumption that $\Gamma$ is constant (see \eqref{eq:def_E}).
\autoref{fig:qpat} shows a typical example of qPAT data derived from piecewise constant material parameters.
\begin{figure}[!ht]
 \begin{center}
\includegraphics[scale=0.5]{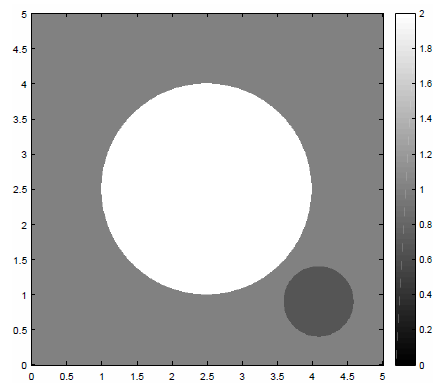}
\hspace{1ex}
\includegraphics[scale=0.5]{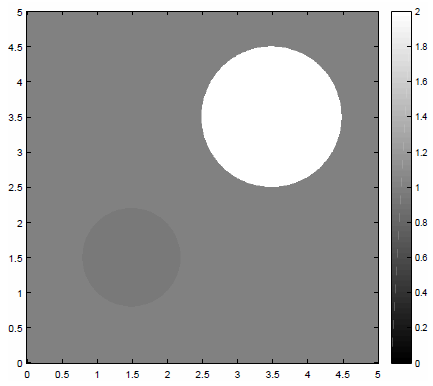}
\hspace{1ex}
\includegraphics[scale=0.5]{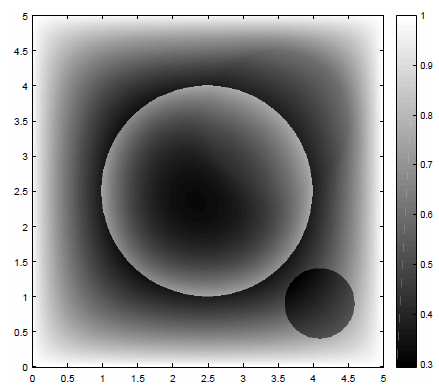}
\end{center}
   \caption{\footnotesize From the left: the piecewise constant absorption coefficient $\mu$, the piecewise constant
   diffusion coefficient $D$, the simulated qPAT data $f=\mathcal{E}$. The parameters are from \autoref{qpatproblem}.
   The test data is analogous as in \cite{BerMusNaeSch16}.}
\label{fig:qpat}
\end{figure}
In this section, we extend the topological based algorithms for edges detection in image data (see \cite[Algorithm 1 and Algorithm 2]{BerGraMusSch14}), by using elliptic differential equations of order $2m$, see \eqref{ue_strong}, with
$f=\mathcal{E}$, and the topological gradient provided in \eqref{eq:top_gradient}.  Note that in \cite[Algorithm 1 and Algorithm 2]{BerGraMusSch14} the
discontinuities of $f$ have been detected by using the discontinuity of the gradient of the solution of a second order elliptic equation along line segments.
We emphasize that according to \cite{NaeSch14} if the diffusion coefficient $D$ is jumping across an
interface but the absorption $\mu$ is constant, one observes a jump in the derivative of the gradient of $\mathcal{E}$.

The goal of this section is to show that using the topological gradient derived in \eqref{eq:top_gradient}, we are able to detect both the absorption and diffusion coefficients better than in the results obtained in \cite{BerGraMusSch14}, and the ones in \cite{BerMusNaeSch16} where a variational method based on an Ambrosio-Tortorelli approximation of a Mumford-Shah-like functional is used. 
At this aim, we use the asymptotic expansion provided in the previous section, specialized only to the case $m=2,3$, comparing the results with the case $m=1$, which was studied in \cite{BerGraMusSch14}, and with the numerical outcomes of \cite{NaeSch14}.   

Using the topological asymptotic expansion \eqref{eq:top_gradient} with
$v = \kappa \chi_K + 1 \chi_{\Omega \backslash K}$,
$\vep = \kappa \chi_{K \cup \overline{\Oe}(y,\tau)} + 1 \chi_{\Omega \backslash (K \cup \overline{\Oe}(y,\tau))}$
(see \eqref{eq:v}) where the according set $K$ is as introduced in Assumption \ref{ass:assumption_domains}, we get
\begin{equation*}
 \functional{\uep}{\vep} - \functional{u}{v} \sim 2\ve^3\alpha(\kappa-1)\tensor\nabla^mu(y)\cdot\nabla^mu(y).
\end{equation*}
To develop a stable algorithm, we follow the approach proposed in \cite{BerGraMusSch14} for the case $m=1$. We recall here the main idea: For every $v\in L^2(\Omega)$, we define 
	\begin{equation*}
		m_{\varepsilon}(v):=\inf\Big\{|S|: S\subset\mathbb{R}^2\times \mathbb{S}^1,\ v=v_{K}\ \ \textrm{with}\ K= \bigcup\limits_{(y,\tau)\in S}\Omega_{\varepsilon}(y,\tau) \Big\},
	\end{equation*}
where we set $m_{\varepsilon}(v):=+\infty$, if $v\neq v_K$ for every finite subset $S\subset \mathbb{R}^2\times \mathbb{S}^1$ with $K=\bigcup\limits_{(y,\tau)\in S}\Omega_\ve(y,\tau)$. The idea behind the algorithm is to introduce a slight modification of the functional $\mathcal{J}$ defined in \eqref{Jeps} in order to take into account a constraint on the perimeter of $K$, i.e.
	\begin{equation*}
		\mathcal{J}_\ve(u,v):=\frac{1}{2}\int_{\Omega}(u-f)^2\, dx +\frac{\alpha}{2}\int_{\Omega}v |\nabla^m u|^2\, dx +2\beta\ve m_{\ve}(v),
	\end{equation*}
where $\beta$ is a positive parameter, for all $u\in H^m_0(\Omega)$ and $v\in L^{\infty}(\Omega)$.	
It is shown in \cite{BerGraMusSch14} that for general $\Omega_\ve(y,\tau)\cap K=\emptyset$, it yields 
	\begin{equation*}
		\mathcal{J}_\ve(u,v_\ve)-\mathcal{J}_\ve(\hat{u},v)=\mathcal{J}(u,v_\ve)-\mathcal{J}(\hat{u},v)+2\beta\ve
	\end{equation*} 
where $\mathcal{J}$ is exactly the functional defined in \eqref{Jeps}, for all $u, \hat{u}\in H^m_0(\Omega)$. Therefore, by the equation \eqref{eq:top_gradient}, we have
	\begin{equation}\label{eq:compl_asympt_expans}
		 \mathcal{J}_\ve(u_\ve,v_\ve)-\mathcal{J}_\ve(u,v)= \functional{\uep}{\vep} - \functional{u}{v} +2\beta\ve \sim 2\ve^3\alpha(\kappa-1)\tensor\nabla^mu(y)\cdot\nabla^mu(y) +2\beta\ve.
	\end{equation}
We observe that, from Proposition \ref{BoundsPT}, we have that 
	\begin{equation}\label{eq:tens_der2_der2}
		\mathbb{M}(\tau)\nabla^m u(y)\cdot \nabla^m u(y)\leq \frac{|\nabla^m u(y)|^2}{\kappa}.
	\end{equation}
Therefore, substituting	$\mathbb{M}(\tau)\nabla^m u(y)\cdot \nabla^m u(y)\sim \frac{|\nabla^m u(y)|^2}{\kappa}$ in \eqref{eq:compl_asympt_expans}, which implies that the direction of the line segment has to be chosen parallel to the eigenvector associated to the eigenvalue of $\nabla^m u(y)$ with maximum absolute value, we find
 	\begin{equation*}
 	\mathcal{J}_\ve(u_\ve,v_\ve)-\mathcal{J}_\ve(u,v)\sim -2\ve^3\alpha\frac{(1-\kappa)}{\kappa}|\nabla^m u(y)|^2 +2\beta\ve,
 	\end{equation*}
hence we expect a decrease of the functional $\mathcal{J}_\ve$ in case
	\begin{equation}\label{eq:stab_crit}
		|\nabla^m u(y)|^2\geq \frac{\beta \kappa}{\alpha \ve^2 (1-\kappa)}.
	\end{equation}
In the particular case $m=2$, we can be more precise, finding the maximum value of $\mathbb{M}(\tau)\nabla^m u(y)\cdot \nabla^m u(y)$ and the direction of $\tau$ where it occurs. In fact, we are able to provide the explicit expression of the tensor $\tensor$, which is optimal with respect to $\nabla^2u(y)$, i.e., which maximizes the form $\mathbb{M}\nabla^2u(y)\cdot \nabla^2 u(y)$.
\begin{lemma}[$m=2$] \label{le:optimal_tensor}
	Let $\lambda_1:=\lambda_1(y)$ and $\lambda_2:=\lambda_2(y)$ denote the eigenvalues of $\nabla^2u(y)$ and the according eigenvectors 
	are denoted by $\hat{\tau}:=\hat{\tau}(y)$ and $\hat{\tau}^\bot:=\hat{\tau}^\bot(y)$, respectively.
	Moreover, we assume that $\abs{\lambda_1} \leq \abs{\lambda_2}$. 
	Then, for $\tensor$ as defined in \eqref{eq:main_general}, with $m=2$, we have 
	\begin{equation*}
	\tensor(\hat{\tau}) \nabla^2u(y)\cdot\nabla^2u(y) =\lambda_1^2 + \frac{\lambda_2^2}{\kappa} = \max_{{\tau} \in \mathbb{S}^1} \tensor({\tau}) \nabla^2u(y)\cdot\nabla^2u(y).
	\end{equation*}
%	and the matrices $E_{\bm{i}}$, $i=1,2,3$ according to $\tensor({\tau})$ 
%	are defined in \eqref{eq:tensor_II}.
	%where the vector $\tau$ is replaced by $\hat{\tau}$ in \eqref{eq:tensor_II}.
\end{lemma}
\begin{proof} 
	First of all, we note that, due to the symmetry of $\nabla^2u(y)$, we can consider an orthogonal decomposition of the matrix $\nabla^2u(y)$, i.e., 
	\begin{equation}\label{eq:decomp_matr}
	\nabla^2 u (y) = U \Lambda U^T 
	\end{equation}
	where $U$ is an orthogonal matrix, and the matrices are given by
	\begin{equation}\label{eq:lambda_orthogonal_mat}
	\Lambda=\begin{bmatrix}
	\lambda_1 & 0\\
	0 & \lambda_2
	\end{bmatrix},\qquad
	U = \begin{bmatrix} 
	\hat{\tau} & \hat{\tau}^\bot 
	\end{bmatrix}
	=\begin{bmatrix}
	\hat{\tau}_1  & -\hat{\tau}_2 \\
	\hat{\tau}_2  &  \hat{\tau}_1
	\end{bmatrix}
	\end{equation}
	Then, from \eqref{eq:main_general} (cf. \eqref{eq:rechen}), \eqref{eq:decomp_matr} and \eqref{eq:lambda_orthogonal_mat}, we get
	\begin{equation} \label{eq:first}
	\begin{aligned}
	\tensor(\tau) \nabla^2u(y)\cdot\nabla^2u(y)  &= \left[ (U\ \Lambda\ U^T) \cdot (\tau \otimes \tau) \right]^2 \\
	&+ \frac{1}{2}   \left[ (U\ \Lambda\ U^T) \cdot (\tau \otimes \tau^\bot +\tau^\bot \otimes \tau) \right]^2 
	+ 
	\frac{1}{\kappa}   \left[ (U\ \Lambda\ U^T) \cdot (\tau^\bot \otimes \tau^\bot) \right]^2  \\
	%%%%%%%%%%%%%%%%%%%%%%%%
	&=  
	\left[ (\tau^TU)\ \Lambda\ (\tau^TU)^T \right]^2 + \frac{1}{\kappa}   \left[((\tau^\bot)^TU)\ \Lambda \ ((\tau^\bot)^T U )^T \right]^2 \\
	&+ 
	\frac{1}{2}   \left[ (\tau^TU)\ \Lambda\ ((\tau^\bot)^TU)^T + 
	((\tau^\bot)^TU)\ \Lambda\ (\tau^T U )^T \right]^2\\
	\end{aligned}
	\end{equation}
	Now, we denote by $\chi = \tau^TU$ and $\nu = (\tau^\bot)^T U$, then we have $\nu=\chi^\bot$ because $U$ is orthogonal, in fact  
	\begin{equation*}
	\chi \cdot \nu = \nu^T \chi = U^T\tau^\bot \tau^T U =0.
	\end{equation*}
	Because $\norm{\nu}=\norm{\chi}=1$, it follows from \eqref{eq:first} and the explicit expression of $\Lambda$, see \eqref{eq:lambda_orthogonal_mat}, that
	\begin{equation} \label{eq:first2}
	\begin{aligned}
	\tensor(\tau) \nabla^2u(y)\cdot\nabla^2u(y)
	= & 
	\left[ \chi\ \Lambda\ \chi^T \right]^2+ 
	\frac{1}{2}   \left[ \chi\ \Lambda\ (\chi^{\bot})^T + 
	\chi^{\bot}\ \Lambda\   \chi^T \right]^2 + 
	\frac{1}{\kappa}   \left[ \chi^{\bot}\ \Lambda\ (\chi^{\bot})^T \right]^2\\
	&= (\lambda_1\chi_1^2+\lambda_2\chi_2^2)^2 + \frac{1}{\kappa} (\lambda_1\chi_2^2 +\lambda_2\chi_1^2)^2 
	+ 2(\lambda_2-\lambda_1)^2\chi_1^2\chi_2^2\\ 
	&= \lambda_1^2 \chi_1^4+\lambda_2^2 \chi_2^4 + 2\lambda_1\lambda_2\chi_1^2 \chi_2^2 \\
	& \quad + \frac{1}{\kappa}(\lambda_1^2\chi_2^4+\lambda_2^2\chi_1^4 + 2\lambda_1\lambda_2\chi_1^2 \chi_2^2)
	+ 2(\lambda_2-\lambda_1)^2\chi_1^2 \chi_2^2 =:
	\mathcal{T}.
	\end{aligned}
	\end{equation}
	Defining $\rho := \lambda_1^2 + \frac{1}{\kappa} \lambda_2^2$ and $\bar \rho := \lambda_2^2 + \frac{1}{\kappa} \lambda_1^2$ we see that the last term on the right hand side of 
	\eqref{eq:first2} equals
	\begin{equation*}
	\begin{aligned}
	\mathcal{T} = \left(\rho (\chi_1^4+2\chi_1^2 \chi_2^2 +\chi_2^4) - 2 \chi_1^2 \chi_2^2 \rho - \chi_2^4 \rho \right) + \chi_2^4 \bar \rho + 
	2\chi_1^2 \chi_2^2\lambda_1\lambda_2 \left(1+ \frac{1}{\kappa} \right) + 2(\lambda_2-\lambda_1)^2\chi_1^2 \chi_2^2.
	\end{aligned}
	\end{equation*}
	Now, we take into account that $\chi$ is a vector of norm $1$ and thus
	\begin{equation*}
	\begin{aligned}
	\chi_1^4+2\chi_1^2 \chi_2^2 +\chi_2^4 &=1,\\ 
	(-\rho  + \bar \rho) &= (\lambda_2^2-\lambda_1^2) \left(1-\frac{1}{\kappa} \right),\\
	\lambda_1\lambda_2 \left(1+ \frac{1}{\kappa} \right) + (\lambda_2-\lambda_1)^2 -\rho &= 
	(\lambda_1\lambda_2 - \lambda_2^2)\left(\frac{1}{\kappa}-1 \right),
	\end{aligned}
	\end{equation*}
	and thus 
	\begin{equation*}
	\begin{aligned}
	\mathcal{T} = \rho  + \chi_2^4 (\lambda_2^2-\lambda_1^2) \left(1-\frac{1}{\kappa} \right)-
	2\chi_1^2 \chi_2^2 (\lambda_1\lambda_2 - \lambda_2^2)\left(1-\frac{1}{\kappa}\right).
	\end{aligned}
	\end{equation*}
	By the assumptions on the eigenvalues, i.e. $\abs{\lambda_2} > \abs{\lambda_1}$, the sum of the last two terms is always negative, hence the maximum value of $\mathcal{T}$ is given by $\rho$. Then, we note that this maximum occurs when $\chi_2=0$, that is, in terms of $\tau$ this means that 
	\begin{equation*}
	0=\chi_2=(\tau^TU )_2 = -\tau_1\hat{\tau}_2+\tau_2\hat{\tau}_1 
	\end{equation*}
	which, equivalently, means that ${\tau} \bot \hat{\tau}^\bot$ or in other words ${\tau} =\pm \hat{\tau}$.
\end{proof}
\begin{remark}\label{rem:eigen} Since $\kappa-1 \leq 0$, as a consequence of \autoref{le:optimal_tensor}, we get that 
	\begin{equation*} \label{eq:optimality}
	\tensor({\tau}) \nabla^2u(y)\cdot\nabla^2u(y) \leq 
	\tensor(\hat{\tau}) \nabla^2u(y)\cdot\nabla^2u(y) = \lambda_1^2(y) + \frac{\lambda_2^2(y)}{\kappa}, \quad \forall {\tau} \in \mathbb{S}^1.
	\end{equation*}
Therefore, by considering the following approximation
	\begin{equation*}
	 		\tensor({\tau}) \nabla^2u(y)\cdot\nabla^2u(y)\sim \lambda_1^2(y) + \frac{\lambda_2^2(y)}{\kappa},
	\end{equation*}
equation \eqref{eq:compl_asympt_expans} becomes
	\begin{equation}\label{eq:asymp_expan_lambda}
		\mathcal{J}_\ve(u_\ve,v_\ve)-\mathcal{J}_\ve(u,v)\sim -2\ve^3\alpha(1-\kappa)\left(\lambda_1^2(y) + \frac{\lambda_2^2(y)}{\kappa}\right) +2\beta\ve.
	\end{equation}	 	 	
Therefore, we expect a dicrease of the functional $\mathcal{J}_\ve$ if
	\begin{equation}\label{eq:stop_crit_lambda}
		\left(\lambda_1^2(y) + \frac{\lambda_2^2(y)}{\kappa}\right)\geq \frac{\beta}{\alpha \ve^2 (1-\kappa)}.
	\end{equation}	
\end{remark}
\begin{remark}
	A similar result of Lemma \ref{le:optimal_tensor} for the case $m\geq 3$ is more involved to get, due to the fact that a decomposition of $\nabla^m u(y)$ in terms of its eigenvalues is not known a priori, see \cite[Section 8.2]{ComGolLimMou08}. In fact, in the real field, the number of eigenvalues of a $m$-order tensor could be different from the dimensional space (in our case 2). 
\end{remark}
Using the above facts, we implement, in the case $m=2$ and $m=3$, a two step algorithm for detecting line segments consisting of first detection of edge position, and next, determining its direction according to the rules:
\begin{enumerate}
\item\label{item:1} We detect only significant edges by selecting the point $y\in L$ satisfying the stabilizing criterion \eqref{eq:stab_crit};
\item\label{item:2} In this point, we create a line segment in the same direction of the eigenvector associated to the eigenvalue of $\nabla^m u(y)$ of maximum absolute value.
\end{enumerate}
Before presenting the numerical results, we make some other remarks on this algorithm.

\begin{remark}
	In the case $m=3$, to identify the eigenvalue of greatest absolute value of $\nabla^3 u(y)$ and, in particular, its corresponding eigenvector, we utilize the results in \cite[Theorem 7.3]{QiChenChen18} regarding the L-eigenvectors of a third order tensor. We recall here the main step: given $\nabla^3 u(y)$ we define the \textit{kernel tensor} $U_{ij}=\sum_{k,h=1}^{2}(\nabla^3 u(y))_{ikh} (\nabla^3 u(y))_{khj}$. Then, the eigenvector associated to the greatest eigenvalue of $\nabla^3 u(y)$, is equal to the eigenvector associated to the greatest eigenvalue of maximum module of the kernel matrix $U$, see \cite{QiChenChen18} for more details.
\end{remark}

\begin{remark}
	For the case $m=2$, we will also show the results given by steps \ref{item:1} and \ref{item:2} where we replace the stabilizing criterion with \eqref{eq:stop_crit_lambda} and choosing as the direction of the line segment the one given in Lemma \ref{le:optimal_tensor} (see Algorithm \ref{alg:m2} below).	
\end{remark}
We are now in position to develop and show three algorithms, \autoref{alg:noupdate}, \autoref{alg:m2} and \autoref{alg:update} below, generalizing those contained in \cite{BerGraMusSch14}.
These algorithms will be applied to the qPAT test data from \cite{NaeSch14}, represented in \autoref{fig:qpat}. Since we need to solve numerically higher order elliptic equations with finite element methods, the qPAT image is down-sampled to $(130\times 130)$ in order to save computational time. The numerical results can be applied
to more general problems which aim to detect discontinuities in an image $f$ through the discontinuity of higher order derivatives of a smoothed version of $f$.

\subsection*{\autoref{alg:noupdate}} Our first algorithm computes a smoothed version of the input image $f$, where the smoothed image is the solution of a $2m$-order elliptic equations with $v=1$, see \eqref{u_strong}. By only using this regularization function, we identify a sequence of thin stripes $K^{(j)}$,
where $K^{(j+1)}$ is formed by including to $K^{(j)}$ a thin stripe $\Oe(y^{(j)},\tau^{(j)})$
in the position $y^{(j)}$, for which $\abs{\nabla^m u(y)}^2$ is maximal, and along the direction $\tau^{(j)}$ until 
$\abs{\nabla^m u(y)}^2 \geq \frac{\beta \kappa}{\alpha \ve^2 (1-\kappa)}$. We note that the direction $\tau^{(j)}$ coincides with that one of the eigenvector associated to the greatest absolute eigenvalue, when $m=2$ or $m=3$, and is chosen orthogonal to the gradient when $m=1$ (see \cite{BerGraMusSch14} for this last case). See the scheme in Algorithm \ref{alg:noupdate}.

\SetAlCapSkip{2ex}
\begin{algorithm}[ht]
\DontPrintSemicolon
\SetAlgoLined
  %\SetLine
  \KwData{input data $f$ (for instance qPAT data $\mathcal{E}$), $m=1$ or $m=2$ or $m=3$, parameters $0 < \kappa \ll 1$, $\alpha$, $\beta > 0$,
          $\ve$, $\delta_0 > 0$, $\varrho_0>0$;}
  \KwResult{the set of line segments $S$ and the set of thin stripes $K$;}
  \BlankLine
  \SetKwInput{Initialize}{Initialization}
  \Initialize{set $S = \emptyset$, $K = \emptyset$  and $L = \Omega\setminus(\partial\Omega\oplus \mathcal{B}_{\delta_0}(0))$}
  \BlankLine
  compute the solution $u$ of
  \begin{equation*}
   \begin{cases}
  u+\alpha(-1)^m\divz(\nabla^m u)=f & \text{ in } \Omega,\\[0.1cm]
  u = \normalderivative{u}=\cdots=\normalderivativem{u}{m-1}= 0& \text{ on } \partial \Omega,
  \end{cases}
  \end{equation*}
  with a finite element method.

  \Repeat{$\max_{y \in L} \abs{\nabla^m u(y)}^2 < \frac{\beta \kappa}{\alpha \ve^2(1-\kappa)}$}{
    find $y \in L$ such that $\abs{\nabla^m u(y)}^2$ is maximal;

    compute the line segment  $\Sigma_\ve(y,\tau)$ and the thin stripe $\Oe(y,\tau)$;

    set $S \leftarrow S \cup \Sigma_\ve(y,\tau)$ and $K \leftarrow K \cup \Oe(y,\tau)$;

    set $L \leftarrow L \setminus (\Oe(y,\tau) \oplus \mathcal{B}_{\varrho_0}(0)) $;
  }
  \caption{Implementation without updates of $v$.}\label{alg:noupdate}
\end{algorithm}

\begin{remark}
	When $m=2$, we use Hsieh-Clough-Tocher $C^1$ finite elements to solve the fourth-order differential equations. Let call $\mathcal{T}^{\Delta}_h$ the sub mesh of $\mathcal{T}$ where all the triangles are split in 3 sub triangles at their barycenter, then
		\begin{equation*}
			HCT_h=\{\eta\in C^1(\Omega):\ \ \forall T\in\mathcal{T}^{\Delta}_h\,\,\ \eta_{|_{T}}\in P^3 \}
		\end{equation*}
	where $P^3$ is the set of polynomials of $\mathbb{R}^2$ of degrees less or equal to 3. The degree of freedom are the value and derivatives at vertices and normal derivative at middle edge point of initial meshes, see \cite{Poz14}.

	In the case $m=3$, we use a splitting method in order to solve the corresponding $6$-th order equation with a finite element method. In particular, we solve the system given by $v:=\Delta u$ and $u-\alpha \Delta^2 v=f$. From the viewpoint of the implementation, this means that we need only to modify properly the code got for the fourth-order equation. Certainly, the numerical results for this case can be improved using more sophisticated methods which are able to solve directly the sixth-order equation with the properly boundary conditions. In fact, in our implementation only two of the three prescribed boundary conditions are satisfied, namely $u=0$ and $\frac{\partial u}{\partial n}=0$. 
\end{remark}
%%%%%%%%%%%%%%%%%%%%%%%%%%%%%%%%%%%%%%
\subsection*{\autoref{alg:m2}} In this algorithm we focus the attention only to the case $m=2$, following the same ideas contained in Algorithm \ref{alg:noupdate} but we use, as stopping criterion and as identifier of the points where to insert a line segment, the relations in Remark \ref{rem:eigen}. See the scheme below of Algorithm \ref{alg:m2} for all details. 
\SetAlCapSkip{2ex}
\begin{algorithm}[ht]
	\DontPrintSemicolon
	\SetAlgoLined
	%\SetLine
	\KwData{input data $f$ (for instance qPAT data $\mathcal{E}$), parameters $0 < \kappa \ll 1$, $\alpha$, $\beta > 0$,
		$\ve$, $\delta_0 > 0$, $\varrho_0>0$;}
	\KwResult{the set of line segments $S$ and the set of thin stripes $K$;}
	\BlankLine
	\SetKwInput{Initialize}{Initialization}
	\Initialize{set $S = \emptyset$, $K = \emptyset$  and $L = \Omega\setminus(\partial\Omega\oplus \mathcal{B}_{\delta_0}(0))$}
	\BlankLine
	compute the solution $u$ of
	\begin{equation*}
	\begin{cases}
	u+\alpha(\nabla\cdot)^2(\nabla^2 u)=f & \text{ in } \Omega,\\[0.1cm]
	u = \normalderivative{u}= 0& \text{ on } \partial \Omega,
	\end{cases}
	\end{equation*}
	with a finite element method.
	
	\Repeat{$\max_{y \in L} \left(\lambda_1^2(y)+\frac{\lambda_2^2(y)}{k}\right)< \frac{\beta}{\alpha \ve^2(1-\kappa)}$}{
		find $y \in L$ such that $\lambda_1^2(y)+\frac{\lambda_2^2(y)}{k}$ is maximal;
		
		compute the line segment  $\Sigma_\ve(y,\tau)$ and the thin stripe $\Oe(y,\tau)$;
		
		set $S \leftarrow S \cup \Sigma_\ve(y,\tau)$ and $K \leftarrow K \cup \Oe(y,\tau)$;
		
		set $L \leftarrow L \setminus (\Oe(y,\tau) \oplus \mathcal{B}_{\varrho_0}(0)) $;
	}
	\caption{Implementation without updates of $v$ and using eigenvalues of the matrix of second order derivatives.}\label{alg:m2}
\end{algorithm}

\subsection*{\autoref{alg:update}}
We combine updates of the piecewise constant $v$ with updates of the function $u$.
We start with $v\equiv 1$. After adding a fixed number $s$ of thin stripes $\Oe(y,\tau)$ to the set $K$, using the same scheme as in \autoref{alg:noupdate}, we update the piecewise constant function $v$ by setting
\[v = \kappa \chi_{K} + 1 \chi_{\Omega\setminus K}\]
and then we compute a corresponding function $u$, solution of
 \begin{equation}
 \label{eq:u}
  \left\{
  \begin{array}{rcll}
   u+\alpha(-1)^m\divz(v \nabla^m u)&=&f & \text{ in } \Omega,\\[0.1cm]
   u = \normalderivative{u}=\cdots=\normalderivativem{u}{m-1}&=& 0& \text{ on } \partial \Omega,
  \end{array}
  \right.
  \end{equation}
  where $m=1$ or $m=2$, with a finite element method, which is then used for computing $\abs{\nabla^m u(y)}^2$ and selecting the next
  at most $s$ thin stripes $\Oe(y,\tau)$  for including to the set $K$. The process of alternating between the addition
  of stripes and updates of the smoothed function $u$ is repeated until no more admissible points $y \in L$ exist, i.e., when  
  the inequality $\abs{\nabla^m u(y)}^2<\frac{\beta \kappa}{\ve^2 \alpha (1-\kappa)}$ holds.
  
  Due to the complexity in solving a sixth order equation with discontinuous coefficients, we do not implement here the Algorithm \ref{alg:update} when $m=3$.
  
  The reason behind the update of $v$ lies in the fact that the asymptotic expansion derived above becomes increasingly inaccurate as the number of the stripes becomes larger. This means that at some point one has to update $v$ in order to get better reconstructions. The main drawback of this method lies on the fact that we cannot update $v$ at every step because this would imply to solve an elliptic equation which is a lengthy and costly procedure. Thus we choose the number $s$ sufficiently large in such a way that approximately less than $8$ computations of the $2m$-order elliptic equation are needed.  
\begin{algorithm}[ht]
\DontPrintSemicolon
\SetAlgoLined
   % \SetLine
    \KwData{input data $f$, $m=1$ or $m=2$, parameters $0 < \kappa \ll 1$, $\alpha$, $\beta > 0$, $\ve$, $\delta_0>0$, $\varrho_0>0$, $s\in\N$;}
    \KwResult{the set of line segments $S$ and the set of thin stripes $K$;}
    \BlankLine
    \SetKwInput{Initialize}{Initialization}
    \Initialize{set
      $S = \emptyset$, $K = \emptyset$ and $L = \Omega\setminus(\partial\Omega \oplus \mathcal{B}_{\delta_0}(0))$;}
    \BlankLine
    compute the solution $u$ of
  \begin{equation*}
  \left\{
  \begin{array}{rcll}
   u+\alpha(-1)^m\divz(\nabla^m u)&=&f & \text{ in } \Omega,\\[0.1cm]
   u = \normalderivative{u}=\cdots=\normalderivativem{u}{m-1}&=& 0& \text{ on } \partial \Omega.
  \end{array}
  \right.
  \end{equation*}
  with a finite element method.

    \Repeat{$\max_{y \in L} \abs{\nabla^m u(y)}^2 < \frac{\beta \kappa}{\alpha \ve^2 (1-\kappa)}$}{
      set $k = 1$;
      \Repeat{$k > s$ or $\abs{\nabla^m u(y)}^2 < \frac{\beta \kappa}{\alpha \ve^2 (1-\kappa)}$}{
          find $y\in L$ such that $\abs{\nabla^m u(y)}^2$ is maximal;

          compute the line segment  $\Sigma_\ve(y,\tau)$ and the thin stripe $\Oe(y,\tau)$;

          set $S \leftarrow S \cup \Sigma_\ve(y,\tau)$ and $K \leftarrow K \cup \Oe(y,\tau)$;

          set $L \leftarrow L \setminus (\Oe(y,\tau)\oplus \mathcal{B}_{\varrho_0}(0)) $;
          set $k \leftarrow k+1$;
      }
      set $v = \kappa \chi_{K} + 1 \chi_{\Omega\setminus K} $;
      compute the solution $u$ of
  \begin{equation*}
  \left\{
  \begin{array}{rcll}
   u+\alpha(-1)^m\divz(v \nabla^m u)&=&f & \text{ in } \Omega,\\[0.1cm]
   u = \normalderivative{u}=\cdots=\normalderivativem{u}{m-1}&=& 0& \text{ on } \partial \Omega.
  \end{array}
  \right.
  \end{equation*}
  with a finite element method.
    }
    \caption{Implementation with updates of $v$.}\label{alg:update}
  \end{algorithm}

All the algorithms described above have been implemented in Matlab.

\subsection*{Results of Numerical Experiments}
To test the proposed algorithms we have performed five experiments, see below Test 1, 2, 3, 4 and 5.
In Test 1, 2 and 3, the source term $f$ is given by the simulated qPAT data presented in \autoref{fig:qpat}, down-sampled to an image of size $130\times 130$ in order to save computational time. These tests are the results of the application of the Algorithms 1, 2 and 3, respectively.\\
Tests 4 and 5 are performed with the same qPAT data but corrupted by a small amount of noise. Specifically, in Test 4 and Test 5 we add to the image a Gaussian noise with standard deviation of $0.1\%$ and $2\%$, respectively, of the average signal of qPAT data. Test 4 serves for a direct comparison with the numerical outcomes in \cite{BerMusNaeSch16}.

For all the experiments, the parameters $\ve$, corresponding to the length of a line segment and the thickness of a stripe, is set to be $\ve = h$, where $h$ being the pixel size. The size of stripe's neighborhood, i.e. $\varrho_0$, is equal to $\varrho_0 = h$. Moreover, in all tests, we set $\delta_0=12h$, i.e. starting from and image of size $130\times 130$, the restricted set $L$ of $\Omega$, defined in Algorithm \ref{alg:noupdate}, \ref{alg:m2} and \ref{alg:update}, has dimensios $106\times 106$, which is indeed the size of the images in Test 1,2,3,4,5. 

The parameters used in the five tests are summarized in Table \ref{tab: test}. 

\begin{table}[!h]
	\caption{Parameters used in Test 1, 2, 3, 4 and 5. In the first column we show the number of the test. In the second one, we indicate if there is noise in the qPAT data, i.e. whether $f$ is corrupted by Gaussian noise. In the third and fourth columns, we specify the algorithm which we implement and the order of the derivatives we are considering, respectively. The last three columns give the principal parameters which are used to implement the various algorithms.}\label{tab: test}
	\begin{center}
		{\renewcommand{\arraystretch}{1.2}
			\begin{tabular}{|l|c|c|c|c|c|c|}
				\cline{1-7} & & & & & & \vspace{-0.2cm} \\
				Test    & Noise     &Algorithm & $m$ (order) & $\alpha$   & $\kappa$   & $\beta$ \\ \hline  
				&  &           & $1$       &            &            & $0.0072$  \\ \cline{4-4} \cline{7-7}
				Test 1  & no &     $1$   & $2$       &            &            & $1.1029\cdot 10^{-4}$  \\ \cline{4-4} \cline{7-7}
				&   &           & $3$       &            &            & $6.4453\cdot 10^{-5}$  \\ \cline{1-4} \cline{7-7}
				Test 2  & no  &    $2$    & $2$       & $10^{-1}$  & $10^{-2}$  &  $9.5803\times 10^{-5}$	\\ \cline{1-4}   \cline{7-7} 
				Test 3  & no &    $3$    & $1$       &            &            & $0.0072$ \\ \cline{4-4}   \cline{7-7}
				& &           & $2$       &            &            & $1.1029\cdot 10^{-4}$ \\ \cline{1-5} \cline{7-7}
				Test 4  & yes ($0.1\%$)  &       & $2$       &  & &  $9.288\cdot 10^{-4}$	\\ \cline{1-2} \cline{4-4}  \cline{7-7}  		
				& &       &  $1$      &         &            & 	$0.0724$					\\ \cline{4-4} \cline{7-7}
				Test 5	& yes ($2\%$) & $1$          & $2$       &  $1$          &           & $0.0020 $  \\ \cline{4-4} \cline{7-7} 
				&   &           & $3$       &            &            & $0.0015$  \\ \hline	
			\end{tabular}
		}
	\end{center}   
\end{table}

\textbf{Test 1.} We apply Algorithm \ref{alg:noupdate}. The numerical results are given in Figure \ref{fig:test1}. In the first column we have $|\nabla^m u|$, where $m=1,2,3$, respectively. On the right column, we give the line segments given by the application of Algorithm \ref{alg:noupdate}.
Despite the coefficients $v$ remains constant in all the iterations, the numerical outcomes of higher order elliptic equations, see the cases $m=2$ and $m=3$, bring to identify both the absorption ($\mu$) and the diffusion ($D$) coefficients. This confirms analytical results in \cite{NaeSch14,BerMusNaeSch16} which state that the union of jumps in coefficient $\mu$ and $D$ is contained in derivatives of $f$ up to the second order.  
	
\begin{figure}[!ht]
	\begin{center}
		\vspace{0.4cm}
		    \includegraphics[height=6cm,width=6cm]{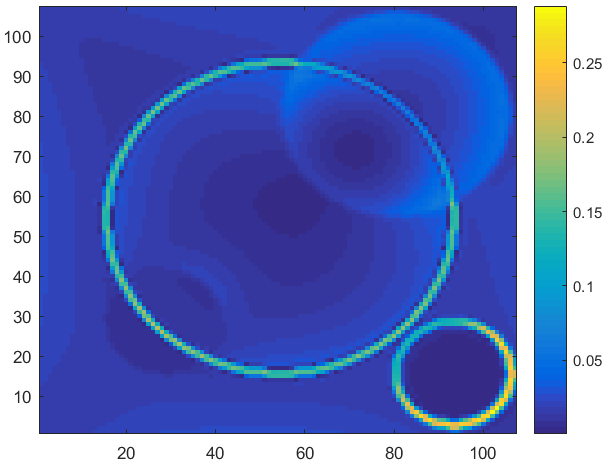}\hspace{1cm}
			\includegraphics[height=6cm,width=6cm]{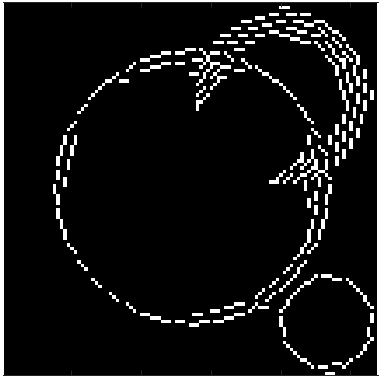}\vspace{0.8cm}
			\includegraphics[height=6cm,width=6cm]{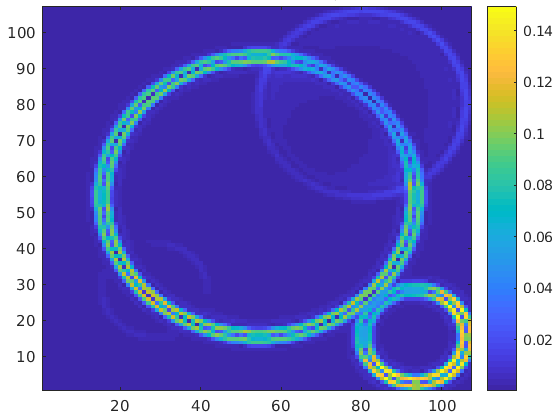}\hspace{1cm}
			\includegraphics[height=6cm,width=6cm]{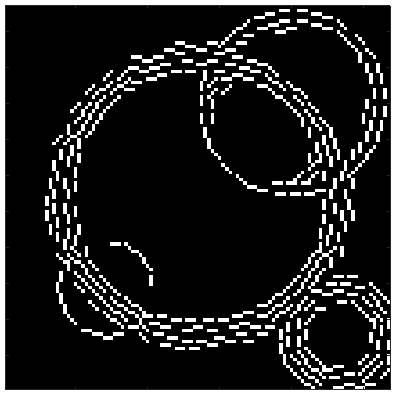}\vspace{0.8cm}
		    \includegraphics[height=6cm,width=6cm]{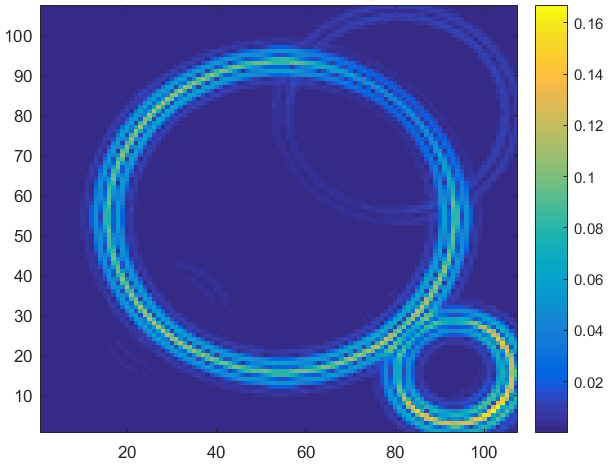}\hspace{1cm}
			\includegraphics[height=6cm,width=6cm]{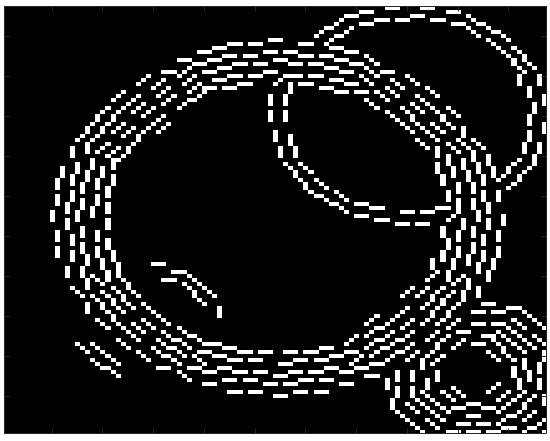}			
			\caption{Test 1. Implementation of the Algorithm 1. In the first column is represented $|\nabla^m u|$, where $m=1,2,3$. The second column gives the numerical outcomes of   constructed line segments. The results are related to: $m=1$ in the first row, $m=2$ in the second row and $m=3$ in the third row. We emphasize that with $m=1$ the small circle on the left bottom part of the image is not identified. Instead, it appears in the reconstruction with $m=2$ and $m=3$.}\label{fig:test1}
	\end{center}
\end{figure}

\textbf{Test 2.} We apply the Algorithm \ref{alg:m2}. The numerical results are given in Figure \ref{fig:test2}. We observe that the numerical outcomes, by using the results in Remark \ref{rem:eigen}, are perceptibly better than the case with the stopping rule $|\nabla^m u(y)|>\frac{\beta \kappa}{\ve \alpha(1-\kappa)}$. 
\begin{figure}[!ht]
	\begin{center}
		\vspace{0.4cm}
		\includegraphics[height=6cm,width=6cm]{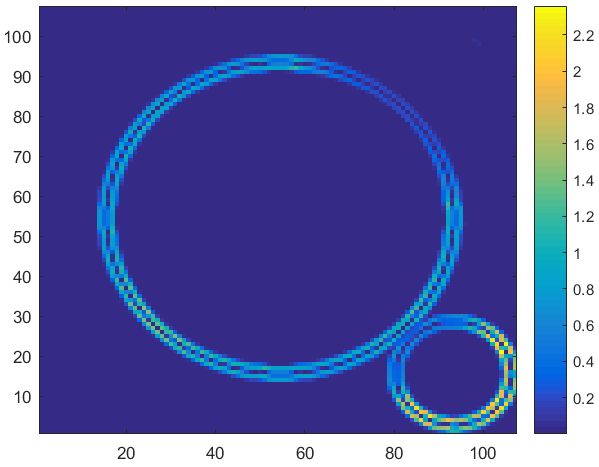}\hspace{1cm}
		\includegraphics[height=6cm,width=6cm]{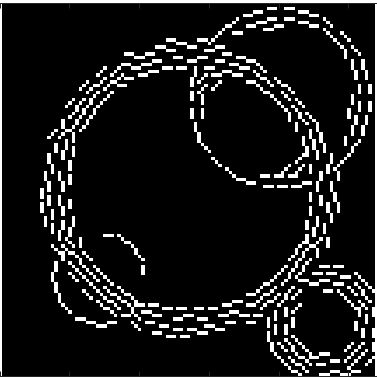}
		\end{center}
	\caption{Test 2. Implementation of Algorithm \ref{alg:m2}, i.e., we are using the criteria in Remark \ref{rem:eigen} to draw the line segments.}\label{fig:test2}
\end{figure}

\textbf{Test 3.} We apply the Algorithm \ref{alg:update}. The numerical results are given in Figure \ref{fig:test3}. In this case, we set the parameter $s$ to be $s=40$ for $m=1$, and $s=50$ for $m=2$. Comparing the numerical outcomes of this test with that ones of Test 1, we can appreciate how the updates of the coefficient $v$ bring to better reconstructions.   
\begin{figure}[!ht]
	\begin{center}
		\vspace{0.4cm}
		\includegraphics[height=6cm,width=6cm]{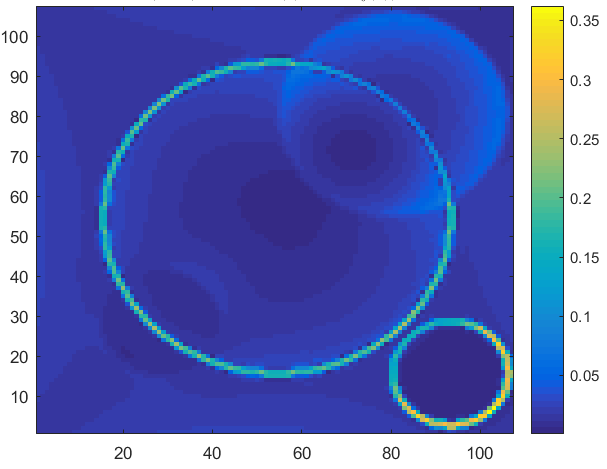}\hspace{1cm}
		\includegraphics[height=6cm,width=6cm]{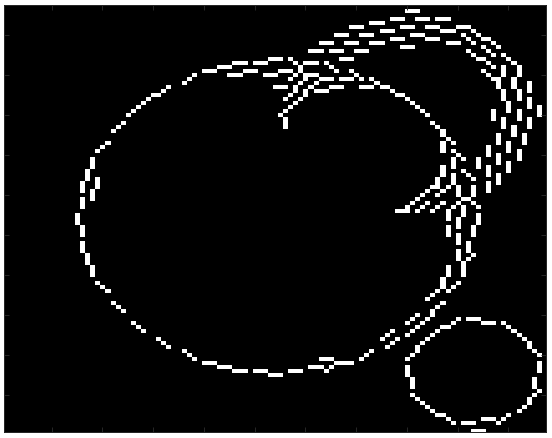}\vspace{0.8cm}
		\includegraphics[height=6cm,width=6cm]{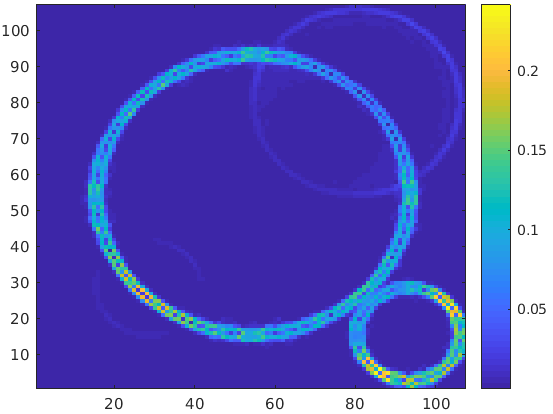}\hspace{1cm}
		\includegraphics[height=6cm,width=6cm]{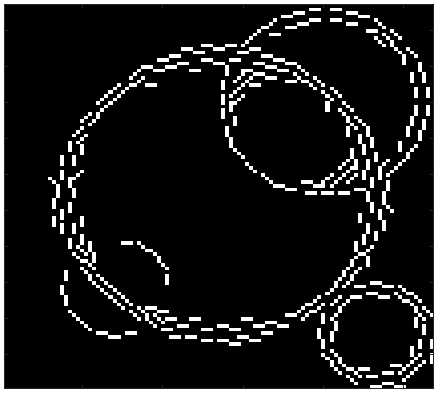}
	\end{center}
	\caption{Test 3. Implementation of Algorithm \ref{alg:update}. In the first row, we have the results given by the second order elliptic equation, i.e. we choose $m=1$ in the algorithm. In the second row, we provide the numerical outcomes of the fourth order elliptic equation, i.e. we choose $m=2$ in the algorithm. For $m=1$ we set $s=40$. For $m=2$, we use $s=50$. }\label{fig:test3}
\end{figure}

\textbf{Test 4.} We apply the Algorithm \ref{alg:noupdate}, with $m=2$ to an image which is corrupted by a small Gaussian noise with standard deviation of $0.1\%$ of the average signal value of the image. The numerical results are given in Figure \ref{fig:test4}. Due to the presence of noise, we use a greater value of $\alpha$ with respect to the previous case, which is set to be $\alpha=1$. 
The numerical results are more stable of the ones in \cite[Fig. 3]{BerMusNaeSch16}, in fact we are able to detect all the four circles in the image.    

\begin{figure}[!ht]
	\begin{center}
		\vspace{0.4cm}
		\includegraphics[height=6cm,width=6cm]{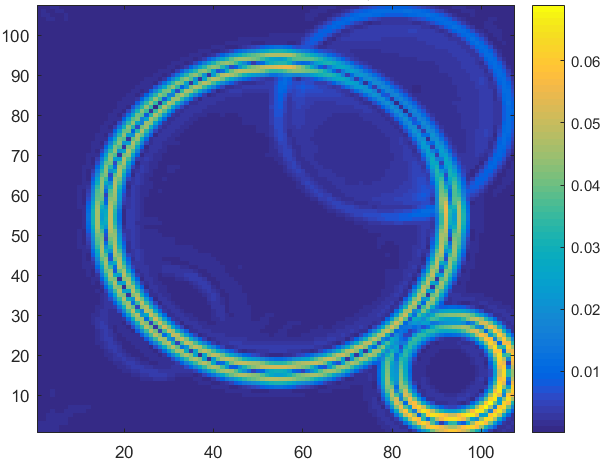}\hspace{1cm}
		\includegraphics[height=6cm,width=6cm]{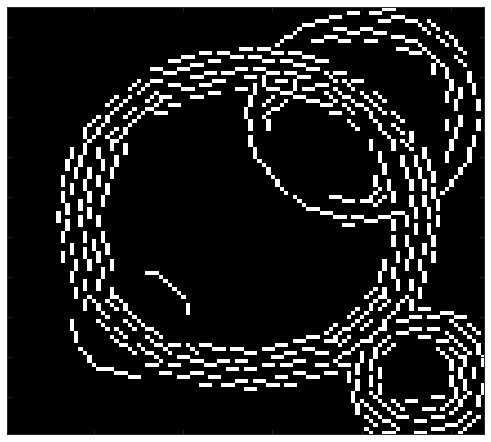}
	\end{center}
	\caption{Test 4. Implementation of Algorithm \ref{alg:noupdate}, in the case where the qPAT data are corrupted by a Gaussian noise of $0.1\%$.}\label{fig:test4}
\end{figure}

\textbf{Test 5.} We apply the Algorithm \ref{alg:noupdate} to an image which is corrupted by a Gaussian noise with standard deviation of $2\%$ of the average signal value of the image.  The numerical results are given in Figure \ref{fig:test5}. Comparing our results with the ones in \cite[Fig. 3]{BerMusNaeSch16}, we observe that they are rather good. We underline that, for the case $m=3$, we expect better results in the case where more sophisticated finite element methods to find the solution of the sixth order equation are applied. Certainly, the splitting method, used in this paper, introduces some errors in the reconstructions because, numerically, only two of the three prescribed boundary conditions are satisfied, namely $u=0$ and $\frac{\partial u}{\partial n}=0$.    

\begin{figure}[!ht]
	\begin{center}
		\vspace{0.4cm}
		\includegraphics[height=6cm,width=6cm]{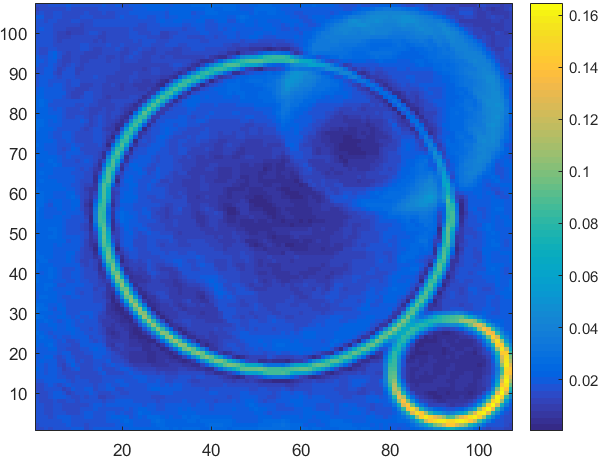}\hspace{1cm}
		\includegraphics[height=6cm,width=6cm]{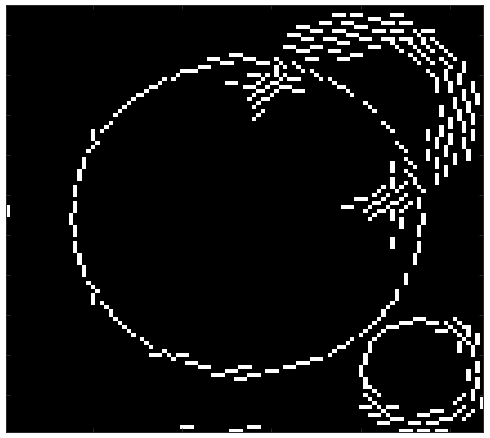}\vspace{0.5cm}
		\includegraphics[height=6cm,width=6cm]{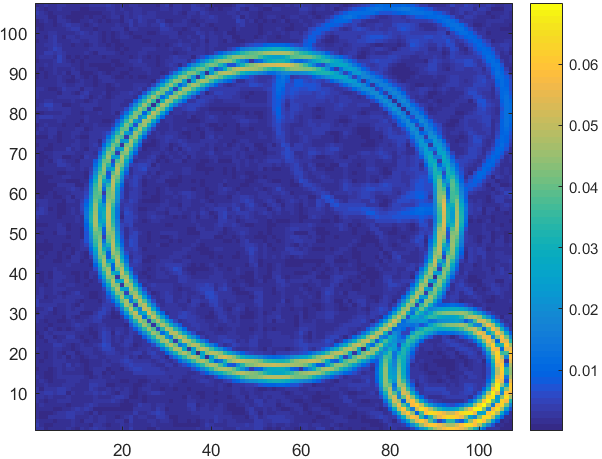}\hspace{1cm}
		\includegraphics[height=6cm,width=6cm]{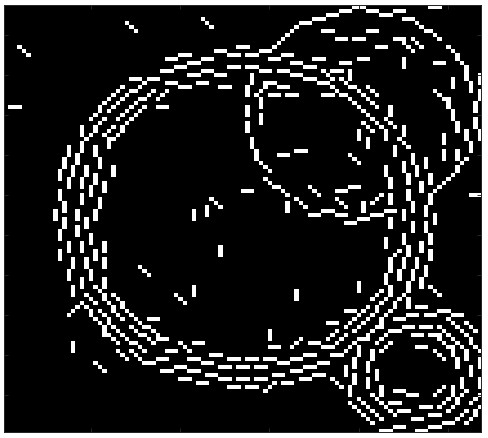}\vspace{0.5cm}
		\includegraphics[height=6cm,width=6cm]{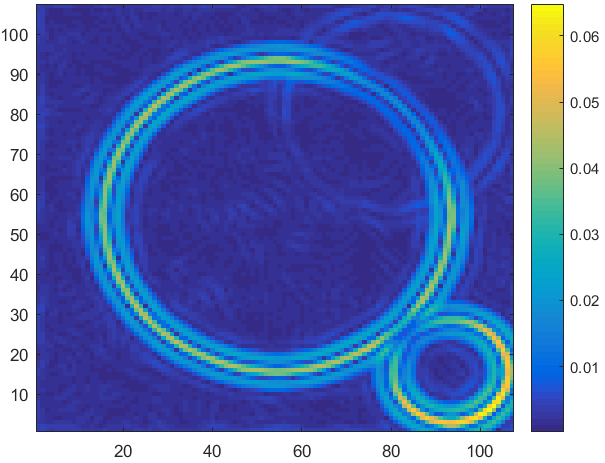}\hspace{1cm}
		\includegraphics[height=6cm,width=6cm]{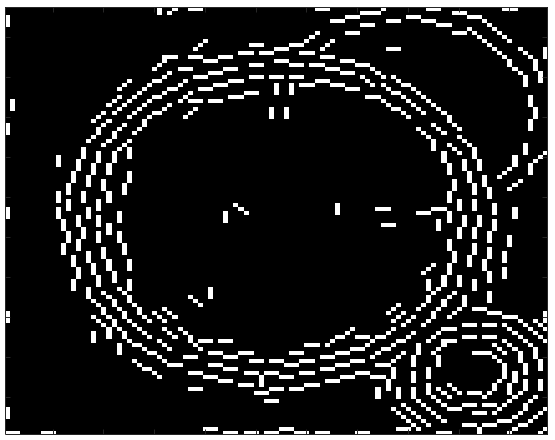}
	\end{center}
	\caption{Test 5. Implementation of Algorithm \ref{alg:noupdate}, in the case where the qPAT image is corrupted by a Gaussian noise of $2\%$. We provide the results for $m=1$ in the first row, for $m=2$ in the second row and finally for $m=3$ in the third row.}\label{fig:test5}
\end{figure}

\newpage
\section*{Conclusion}
In this paper we applied the method of asymptotic expansion of line segments for detection of \emph{discontinuities in derivatives} 
of some data $f$. Asymptotic expansions have been considered for detecting discontinuities in data (and not the derivatives). 
As a test example we considered Quantitative Photoacoustic Tomography with piecewise constant material parameters. This example 
has been considered before in \cite{BerMusNaeSch16}. The method of choice there is based on a differential Canny's edge detector. 
This method requires a pre-smoothing step, which is dependent on the order of discontinuity to be recovered. Conceptually 
our new approach includes the filtering into the detection algorithms and delivers also a tangential direction of the edge.

%- - - - - - - - - - - - - - - - - - - - - - - - - - - - - - - - - - - - - - - - - - - - - - - - - - - - - -
\subsection*{Acknowledgements}
%- - - - - - - - - - - - - - - - - - - - - - - - - - - - - - - - - - - - - - - - - - - - - - - - - - - - - -
OS is supported by the Austrian Science Fund (FWF), with SFB F68, project F6807-N36 (Tomography with Uncertainties) and 
I3661-N27 (Novel Error Measures and Source Conditions of Regularization Methods for Inverse Problems). EB and OS are grateful to 
BIRS, Banff center, for hosting a workshop ``Reconstruction Methods for Inverse Problems'', which supported finalizing the main 
mathematical ideas.

\appendix
\section{Proof of Lemma \ref{le:pde}}\label{append} 
To prove the well-posedness of the problem \eqref{u_strong} and \eqref{ue_strong}, we need the following generalized version of the Poincar\'e-type inequality 
\begin{lemma}[m-order Poincar\' e's Inequality, \cite{Ada75}] \label{le:poincare}
	The two norms $\norm{\cdot}_{H^m(\Omega)}$ and $\norm{\nabla^m \cdot}_{L^2(\Omega)}$ are
	equivalent in $\hzn$, i.e., there exists a positive constant $C$ such that
	\begin{equation} \label{pineq}
	\norm{\nabla^m u}_{L^2(\Omega)} \leq \norm{u}_{H^m(\Omega)} \leq C \norm{\nabla^m u}_{L^2(\Omega)} \text{ for all } u \in \hzn.
	\end{equation}
\end{lemma}
\begin{proof}
We only concentrate on the equation \eqref{u_strong} because the argument of the proof is identical for \eqref{ue_strong}. \\
We first find the weak formulation of the equation \eqref{u_strong} and then we study its well-posedness by apllying the Lax-Milgram theorem. Secondly, we show that the weak formulation of \eqref{u_strong} is in fact the optimality condition satisfied by the minimum of the functional \eqref{Jeps} (where $\zeta=v$) and we prove the equivalence between the minimum problem for the functional \eqref{Jeps} and the weak solution of \eqref{u_strong}. \\
\textit{Well-posedness of \eqref{u_strong}:} Multiplying the equation \eqref{u_strong} for a test function $\varphi\in H^m_0(\Omega)$, then integrating by parts $m$-times and using the boundary conditions (see \eqref{u_strong}), we find the weak formulation of the problem: Find $u\in\hzn$ such that 
\begin{equation}\label{eq:weak_form_u}
\alpha \int_\Omega v \nabla^m u \cdot \nabla^m\varphi \dx + \int_\Omega u\varphi \dx = \int_\Omega f \varphi \dx,\qquad
\text{ for all } \varphi \in \hzn,
\end{equation}
which can be equivalently written in the form
\begin{equation*}
a(u,\varphi)=F(\varphi)\quad \text{ for all } \varphi \in \hzn,
\end{equation*}
where $a$, $F$ denote the bilinear form and the linear functional, respectively,
\begin{equation*}
a(u,\varphi):=\alpha\int_\Omega v \nabla^m u \cdot \nabla^m\varphi \dx + \int_\Omega u \varphi \dx\quad \text{ and }\quad
F(\varphi):=\int_\Omega f \varphi \dx.
\end{equation*}
In order to apply the Lax-Milgram theorem we need prove continuity and coercivity of $a$ and continuity of $F$. Continuity follows by the application of the Cauchy-Schwarz inequality, in fact for all $\psi,\varphi \in \hzn$, we have 
\begin{equation*}
 |{a(u,\varphi)}| \leq C\norm{u}_{H^m(\Omega)} \norm{\varphi}_{H^m(\Omega)},\qquad 
 |F(\varphi)| = \Big|\int_\Omega f \varphi \dx\Big| \leq C \norm{f}_{L^{\infty}(\Omega)} \norm{\varphi}_{H^m(\Omega)}.
\end{equation*}
Coercivity of $a$ on $\hzn$ follows from the Poincar\'e inequality \eqref{pineq} and the definition of $v$ given in \eqref{eq:ve}
\begin{equation*}
 {a(u,u)} \geq C(\norm{\nabla^m u}^2_{L^2(\Omega)} + \norm{u}^2_{L^2(\Omega)}) \geq C \norm{u}^2_{H^m(\Omega)}.
\end{equation*}
Hence, by Lax-Milgram lemma (see \cite{Bab70}) there exists a unique weak solution $u \in \hzn$  to
\begin{equation*}
a(u,\varphi)=F(\varphi), \text{ for all } \varphi \in \hzn
\end{equation*}
and hence of \eqref{ue_strong}. The energy estimate follows by means of the coercivity of $a$ and the continuity of $F$, i.e., 
	\begin{equation*}
		 c \norm{u}^2_{H^m(\Omega)}\leq |{a(u,u)}|= |F(u)|\leq C \norm{f}_{L^{\infty}(\Omega)} \norm{u}_{H^m(\Omega)}.
	\end{equation*} 
\textit{Equivalence of the problems:} Let us assume that $u$ is the minimum of the functional $\mathcal{J}(u; v)$. Then we choose $\xi\in\mathbb{R}$, and for all $h\in H^m_0(\Omega)$ we define $w:=u+\xi h$. Trivially it holds $ \mathcal{J}(u; v)\leq \mathcal{J}(w; v)$. Then, by simple calculations, we have that
	\begin{equation*}
		\frac{\mathcal{J}(u+\xi h;v)-\mathcal{J}(u;v)}{\xi}=\int_{\Omega}(u-f)h\, dx +\alpha\int_{\Omega}v\nabla^m u \cdot \nabla^m h\, dx + \mathcal{O}(\xi),
	\end{equation*}
which gives, as $\xi\to 0$, the weak formulation \eqref{eq:weak_form_u}. On the contrary, assuming that $u$ is the solution to \eqref{eq:weak_form_u} then, taking $\varphi=\xi h$, for all $h\in H^m_0(\Omega)$ and $\xi\in\mathbb{R}$, we find
	\begin{equation*}
		\mathcal{J}(u+\xi h;v)=\mathcal{J}(u;v)+\frac{1}{2}\left(\int_{\Omega}\xi^2h^2\, dx +\alpha \int_{\Omega}\xi^2 v |\nabla^m h|^2\, dx\right).
	\end{equation*}
Now, noticing that 	$\int_{\Omega}\xi^2h^2\, dx +\alpha \int_{\Omega}\xi^2 v |\nabla^m h|^2\, dx\geq 0 $ for all $h\in H^m_0(\Omega)$, $\alpha>0$ and $v\in L^{\infty}_+(\Omega)$, we get
	\begin{equation*}
		\mathcal{J}(u+\xi h;v)-\mathcal{J}(u;v)\geq 0.
	\end{equation*}
\end{proof}
%%%%%%%%%%%%%%%%%%%%%%%%%%%%%%%
%%%% References
%%%%%%%%%%%%%%%%%%%%%%%%%%%%%%%

\end{document}